\documentclass[11pt]{article}
\usepackage{amsmath}
\usepackage{graphicx}
\usepackage{psfrag}
\usepackage{epsf}
\usepackage{enumerate}

\usepackage[round,longnamesfirst]{natbib}
\usepackage{etoolbox}

\makeatletter
\newcommand\bibstyle@comma{\bibpunct(),a,,}
\newcommand\bibstyle@semicolon{\bibpunct();a,,}
\makeatother

\pretocmd\cite{\citestyle{comma}}\relax\relax
\pretocmd\citep{\citestyle{semicolon}}\relax\relax

\usepackage{url} 
\usepackage{amsthm}
\usepackage{xpatch}
\usepackage{subcaption}
\usepackage{bm}
\usepackage{wrapfig}

\usepackage{longtable}
\usepackage{bm}
\usepackage{mathrsfs}
\usepackage{tikz}
\usepackage{color} 
\usepackage{amssymb}
\usepackage{latexsym}
\usepackage{xr}
\usepackage{epstopdf}
\usepackage{float}
\usepackage{multirow}
\usepackage{dsfont}
\usepackage{breakurl}

\usepackage[ruled]{algorithm2e}
\usepackage{lscape}
\usepackage{rotating}
\numberwithin{equation}{section}

\usepackage{xcolor}
\definecolor{lightblue}{HTML}{044E9E}
\usepackage[hyperfootnotes = true,colorlinks,citecolor=lightblue,urlcolor=lightblue, linkcolor=lightblue]{hyperref}

\usetikzlibrary{matrix}
\usepackage[edges]{forest}

\definecolor{values1Color}{HTML}{E66101} 
\definecolor{values2Color}{HTML}{5E3C99}

\def\RR{\mathbb R}
\def\ZZ{\mathbb Z}

\def\NN{\mathbb N}

\def\CC{\mathbb C}

\newcommand{\Var}{\operatorname{Var}} 

\newcommand{\vecop}{\operatorname{vec}}

\newcommand{\diag}{\operatorname{diag}} 
 
\newcommand{\tr}{\operatorname{tr}} 
\newcommand{\Prob}{\operatorname{P}} 
\newcommand{\E}{\operatorname{E}} 
\newcommand{\supp}{\operatorname{supp}} 

\newcommand{\acmfXhat}{\widehat{\bm{\Gamma}}_{X}}
\newcommand{\acmfX}{\bm{\Gamma}_{X}}
\newcommand{\acmfZhat}{\widehat{\bm{\Gamma}}_{Z}}
\newcommand{\acmfZ}{\bm{\Gamma}_{Z}}

\newcommand{\Dim}{d} 
\newcommand{\Lag}{L} 
\newcommand{\VO}{p} 
\newcommand{\TM}{\Psi} 
\newcommand{\maxF}{\max} 
\newcommand{\conZ}{\bm{c}_{Z}}

\newcommand{\capitalG}{G_{ij}^{n_{0},n_{1}}}
\newcommand{\smallg}{g_{ij}^{n_{0},n_{1}}}

\usepackage{mleftright}
\usepackage{xparse}

\NewDocumentCommand{\evaluat}{sO{\big}mm}{%
  \IfBooleanTF{#1}
   {\mleft. #3 \mright|_{#4}}
   {#3#2|_{#4}}%
}

\DeclareMathOperator*{\argmin}{arg\,min}
\DeclareMathOperator*{\esssup}{ess\,sup}

\newcommand{\norm}[1]{\left\lVert#1\right\rVert_{s}}

\newtheorem{theorem}{Theorem}[section]
\newtheorem{proposition}{Proposition}[section]
\newtheorem{lemma}{Lemma}[section]
\newtheorem{corollary}{Corollary}[section]

\theoremstyle{definition}
\newtheorem{remark}{Remark}[section]
\newtheorem{example}{Example}[section]

\newtheorem{mainassump}{Assumption}

\newtheorem{CIassump}{Assumption}

\newtheorem{VARassump}{Assumption}

\makeatletter
\xpatchcmd{\proof}{\@addpunct{.}}{\@addpunct{:}}{}{}
\makeatother

\DeclareFontFamily{U}{mathx}{\hyphenchar\font45}
\DeclareFontShape{U}{mathx}{m}{n}{<-> mathx10}{}
\DeclareSymbolFont{mathx}{U}{mathx}{m}{n}
\DeclareMathAccent{\widebar}{0}{mathx}{"73}

\newcommand{\mockalph}[1]{}

\allowdisplaybreaks

\begin{document}

\def\spacingset#1{\renewcommand{\baselinestretch}%
{#1}\small\normalsize} \spacingset{1}

\newtheorem*{assumptionBIC*}{\assumptionnumber}
\providecommand{\assumptionnumber}{}
\makeatletter
\newenvironment{assumptionBIC}[2]
 {%
  \renewcommand{\assumptionnumber}{Assumption #1#2}%
  \begin{assumptionBIC*}%
  \protected@edef\@currentlabel{#1#2}%
 }
 {%
  \end{assumptionBIC*}
 }
\makeatother


\title{
High-dimensional latent Gaussian count time series: \\
Concentration results for autocovariances and applications
\footnote{Marie D\"uker's reseach  was supported by NSF grant DMS-1934985, Robert Lund thanks NSF grant DMS-1407480, and Vladas Pipiras acknowledges NSF grants DMS-2113662 and DMS-2134107.}}

\author{
Marie-Christine D\"uker \\ Cornell University                             \and
Robert Lund                 \\ University of California, Santa Cruz    \and
Vladas Pipiras               \\ University of North Carolina - Chapel Hill}
\date{\today}

\maketitle

\bigskip

\begin{abstract}
\noindent
This work considers stationary vector count time series models defined via deterministic functions of a latent stationary vector Gaussian series. The construction is very general and ensures a pre-specified marginal distribution for the counts in each dimension, depending on unknown parameters that can be marginally estimated.  The vector Gaussian series injects flexibility into the model's temporal and cross-dimensional dependencies, perhaps through a parametric model akin to a vector autoregression. We show that the latent Gaussian model can be estimated by relating the covariances of the counts and the latent Gaussian series.  In a possibly high-dimensional setting, concentration bounds are established for the differences between the estimated and true latent Gaussian autocovariances, in terms of those for the observed count series and the estimated marginal parameters. The results are applied to the case where the latent Gaussian series is a vector autoregression, and its parameters are estimated sparsely through a LASSO-type procedure.
\end{abstract}

\allowdisplaybreaks

\section{Introduction}
\textit{The model:} This work concerns a strictly stationary multivariate count valued time series model.  The $\Dim$-dimensional count time series at time $t$ is denoted by $X_{t} = (X_{1,t}, \dots, X_{\Dim,t})'$, $t \in \ZZ$, where prime indicates transpose. Count valued means that $X_{i,t} \in \NN_{0} := \{0,1,2, \dots \}$; in practice, the counts could encode categorical or ordinal observations. By strict stationarity, the $i$th component series $\{ X_{i,t} \}$ has a time-invariant marginal cumulative distribution function (CDF) for each $i=1, \dots, \Dim$, which is denoted by  
\begin{equation} 
\label{eq:intro1}
F_{i}(x) = \Prob[ X_{i,t} \leq x ], \quad x \in \RR.
\end{equation}
We are interested in constructing such series through latent standardized Gaussian series. A simple way to ensure that the desired marginal CDF of the $i$th component is $F_i$ sets $X_{i,t} = F_{i}^{-1}(\Phi(Z_{i,t}))$, where $\Phi$ denotes the standard Gaussian CDF and $F_{i}^{-1}$ is the inverse of $F_{i}$ defined below.   
Thus, we define the functions 
\begin{equation*} 
\label{eq:intro2}
G_{i}(z_{i}) = F_{i}^{-1}(\Phi(z_{i})),
\hspace{0.2cm}
G(z) = (G_{1}(z_{1}),\dots, G_{\Dim}(z_{\Dim})), 
\hspace{0.2cm}
z \in \RR^{\Dim},
\end{equation*}
where
\begin{equation*} 
\label{eq:intro3}
F_{i}^{-1}(u) = \inf\{x ~|~ F_{i}(x) \geq u\}, \hspace{0.2cm} u \in (0,1),
\end{equation*}
is the generalized inverse (quantile function) of $F_{i}$.   Our multivariate count model sets
\begin{equation} 
\label{eq:subo}
X_{t} 
= (X_{1,t}, \dots, X_{\Dim,t})'
= (G_{1}(Z_{1,t}), \dots, G_{\Dim}(Z_{\Dim,t}))'
= G(Z_{t}), 
\end{equation}
where $Z_{t} = (Z_{1,t}, \dots, Z_{\Dim,t})'$ is a $\Dim$-dimensional stationary Gaussian series with zero mean and a unit variance:  $\E [Z_{i,t}] \equiv 0$ and $\E [Z_{i,t}^2] \equiv1 $. We write
\begin{equation} 
\label{eq:intro5}
\Gamma_{Z}(h) = R_{Z}(h) = \E [Z_{t+h} Z_{t}'], 
\hspace{0.2cm} h \in \ZZ,
\end{equation}
for the lag-$h$ matrix autocovariance and autocorrelation functions (ACVF and ACF) of $\{ Z_{t} \}$. The ACVF and ACF of the count series $\{ X_{t} \}$ will similarly be denoted by $\Gamma_{X}(h)$ and $R_{X}(h)$. Since the $X_{i,t}$s are not standardized, $\Gamma_{X}(h)$ and $R_{X}(h)$ are not necessarily equal in contrast to \eqref{eq:intro5}. Since the means of $X_{t}$ are not zero either, $\Gamma_{X}(h) = \E [X_{t+h}X_{t}'] - \E [X_{t+h}] \E [X_{t}]'$ at lag $h$. We will often write $\Gamma_X$ and $R_X$ to refer to the ACVF and ACF of $\{ X_{t} \}$ over some or all lags.

While the construction in \eqref{eq:subo} ensures $F_{i}$ as the marginal distribution of $X_{i,t}$, temporal and cross-sectional (spatial) dependencies are driven by the latent Gaussian series $\{ Z_{t}\}$ which will allow us to make inference based on second order properties of the latent process $\{ Z_{t}\}$. We assume that $F_{i}$ depends on an unknown parameter vector $\theta_{i} \in \RR^{K_{i}}$. Furthermore, the construction allows to impose a parametric model on the latent $\{ Z_{t}\}$, such as vector autoregressions (VARs) or dynamic factor models (DFMs).

Some advantageous features of the model are worth stating here.  Besides being able to accommodate any count marginal $F_i$ whatsoever, negative correlations in the counts are easily achieved.  In fact, the model's correlations are the most flexible possible in both a positive and negative sense; see Remark 2.2 in \cite{jia2023latent}.  

\textit{Literature review:}
Modeling discrete time series has been an active research area and is far less developed than the continuous case. There are several classes of different models including those based on thinning operators (e.g. \cite{mckenzie1985some,alzaid1993some}) and the generalized state-space models (e.g. \cite{davis2016handbook}), for example, Markov chain and Hidden Markov models, Bayesian dynamic models (e.g. \cite{gamerman2015dynamic}) and integer-valued autoregressive conditional heteroskedasticity modeling (e.g. \cite{ferland2006integer,fokianos2009poisson}). A recent survey by \cite{davis2021count} discusses several of these classes of count models, including (1.2), and their (dis)advantages.
The analysis of multivariate and potentially high-dimensional count series has received considerably less attention. For a review of approaches for multivariate counts, we refer to \cite{karlis2016models}.

The model (1.2) for $\Dim =1$, was popularized by \cite{jia2023latent}, where several parameter estimation approaches were suggested (Gaussian pseudo-likelihood estimation, a Yule-Walker based approach for latent autoregressive models and particle filtering). For related work by (subsets of) the same authors, see \cite{kong2023seasonal} for the seasonal case when $\Dim =1$. \cite{CountHurricane} used the discussed model to predict hurricane counts assuming Poisson marginals in the case $\Dim =2$. 

In the multivariate and potentially high-dimensional setting, \cite{kim2023latent} used latent Gaussian dynamic factor 
series in the model (1.2) for the purpose of gaining insight into and forecasting of count time series.
In Chapter 5 of their dissertation, \cite{Kim2023diss} considered latent Gaussian series parametrized by a possibly sparse vector autoregression model. The author introduced possibly regularized estimation methods for the latent process, including a numerical study to assess estimation performance.

Other work related to (1.2) uses independent and identically distributed (i.i.d.) $Z_{t}$s, where $t$ may no longer refer to time but, for example, different individuals. 
In psychometrics, models of the type (1.2) have been used extensively for discrete data (e.g. \cite{lebo1978intraindividual}) and related models are termed ``polychoric correlations" for ordinal data. Developed in structural equation models (e.g. \cite{lee1992structural}), they have made their way into various software packages, e.g. Mplus 7.11 \cite{muthen2013mplus}, and more recently into the popular R package {\tt lavaan}; see \cite{rosseel2012lavaan}.
Applications to psychology under consideration of temporal dependences have been considered in more recent works; see \cite{kim2023latent}, \cite{Kim2023diss}.

In the statistical literature, \cite{liu2012high, mitra2014multivariate, wegkamp2016adaptive, HanLiu2017,FanLiuNing2017,feng2019high}, and \cite{Dirksen2022} study Gaussian copula models in possibly high-dimensional settings. These authors derive theoretical results guaranteeing consistent estimation of the latent correlation structure. These publications concentrate on Spearman’s rho and Kendall’s tau matrices, or subsets of these quantities. Under suitable assumptions, the entries of the copula correlation matrix relate to the entries of the Kendall’s tau or Spearman’s rho matrices through an explicit link function that does not need to be estimated.

\textit{Our contributions:} 
As the literature review shows, the considered latent Gaussian count model has found popularity in the more applied literature. In particular, the works mentioned above have pushed forward the development of feasible algorithms to estimate the parameters of the latent Gaussian process under different parametric assumptions including VAR (Chapter 5 in \cite{Kim2023diss}) and DFM (\cite{kim2023latent}). In this work, we aim to give a theoretical justification of a method that has been proven to work well in practice.
The applied and count time series literatures aside, we also note that our theoretical results generalize existing results significantly and push forward the analysis of latent models. In contrast to the existing statistical literature, our results capture a much more general class of functions and incorporate the potentially necessary estimation of the transformation.

We are broadly interested in making inferences about $\Gamma_{Z}$ from the observed counts $X_{1}, \dots, X_{T}$, especially in the high-dimensional setting where $\Dim$ can be much larger than $T$. We do so by first considering an estimator $\widehat{\Gamma}_{Z}$ of $\Gamma_{Z}$ defined informally as follows. As shown below, there is a deterministic function $\ell$, depending only on marginal CDF parameters $\theta_{i}$, such that
\begin{equation} 
\label{eq:intro6}
\Gamma_{X} = \ell(\Gamma_{Z})
\end{equation}
for all lags $h$.  If $\widehat{\theta}_{i}$ is an estimator of $\theta_{i}$ used to construct $\widehat{\ell}$, an estimator of $\ell$, $\Gamma_{Z}$ can be estimated via 
\begin{equation*} 
\label{eq:intro7}
\widehat{\Gamma}_{Z} = \widehat{\ell}^{-1}(\widehat{\Gamma}_{X}),
\end{equation*}
where $\widehat{\Gamma}_{X}$ is a standard ACVF estimator of $\Gamma_{X}$ based on $X_{1}, \dots, X_{T}$. This work consists of providing concentration bounds on $\| \widehat{\Gamma}_{Z} - \Gamma_{Z} \|$ in terms of those for $\| \widehat{\Gamma}_{X} - \Gamma_{X} \|$ and $\| \widehat{\theta}_{i} - \theta_{i} \|$, where the norms $\| \cdot \|$ are suitably chosen. Concentration bounds on $\| \widehat{\Gamma}_{X} - \Gamma_{X} \|$ and $\| \widehat{\theta}_{i} - \theta_{i} \|$ are extracted from available results in the literature. To the best of our knowledge, our main results are new even in the i.i.d.~setting (i.e., with no temporal dependence).

The derived concentration bounds for $\| \widehat{\Gamma}_{Z} - \Gamma_{Z} \|$ make inferences possible for the parameters in $\Gamma_{Z}$. We illustrate this with an application to a latent VAR series $\{ Z_{t} \}$ assuming that its coefficient matrices are suitably sparse. Adapting a LASSO-type approach, we show how the VAR coefficient matrices can be estimated sparsely from $\widehat{\Gamma}_{Z}$, and then use our concentration results to establish consistency of the matrix estimates, including high-dimensional cases. Perhaps somewhat surprisingly, our results here with a latent $\{ Z_t \}$ are of the same flavor as those in the seminal work of \cite{basu2015regularized}, who worked with an  observable VAR series $\{ Z_{t} \}$.

\textit{Organization:}
The rest of the paper is structured as follows. Section \ref{se:prelim} deals with some preliminaries, including issues related to \eqref{eq:intro6}, quantities of interest, assumptions, and some moment quantities. The main concentration results for $\| \widehat{\Gamma}_{Z} - \Gamma_{Z} \|$ are stated in Section \ref{se:concentration},
The application to sparse latent VAR series is considered in Section \ref{se:AppVAR}.
Section \ref{se:conclusion} provides a discussion and conclusions. Technical proofs and additional material is contained in Appendices \ref{se:proofs:main}--\ref{se:Discussion on assumption}.

\textit{Notation:}
For the reader's convenience, notations used throughout the paper are collected here. The maximum and minimum eigenvalues of a symmetric matrix $A$ are denoted by $\lambda_{\max}(A)$ and $\lambda_{\min}(A)$, respectively. To indicate that a matrix $A$ is positive (semi)definite, we write $A \succ 0$ $(A \succcurlyeq 0)$.   The matrix inequality $A \geq B$ means that $A_{ij} \geq B_{ij}$ for $i,j =1, \dots, \Dim$.  A range of different norms are used here, including the maximum norm and the spectral norm, defined respectively as $\| A \|_{\max}=\max_{1 \leq i,j \leq \Dim} |A_{ij}|$, $\| A \|=\sqrt{\lambda_{\max}(A'A)}$ for a matrix $A \in \RR^{\Dim \times \Dim}$.  The $\ell_{1}$-norm $\| v \|_{1} = \sum_{j =1}^{\Dim} |v_{j}|$, the Euclidean norm $\| v \|^2 = \sum_{j =1}^{\Dim} |v_{j}|^2$, and the norm that counts all non-zero elements in $\| v \|_{0}$ for a vector $v \in \RR^{\Dim}$ are also used. For a $\Dim \times N$ matrix composed of $N$ $\Dim$-dimensional vectors $v_{1}, \dots, v_{N}$, we write $[v_{1} : \dots : v_{N}]$. Our proofs use the Hadamard product $A \odot B$ of two matrices $A,B \in \RR^{\Dim \times \Dim}$ and the related notations $A \odot A = A^{\odot 2}$, $A^{\odot \frac{1}{2}} =  (A_{ij}^{\frac{1}{2}})_{i,j =1,\dots,\Dim}$ and $A^{\odot -1} = (1/A_{ij})_{i,j =1,\dots,\Dim}$.  The componentwise application of absolute values has $|A| = (|A_{ij}|)_{i,j =1,\dots, \Dim}$. For two quantities $a$ and $b$, we use $a \succsim b$ if there exists an absolute constant $c$, independent of the model parameters, such that $a \geq cb$.   We write $\nabla_{x} f$ for the gradient of the function $f$ with respect to a vector $x \in \RR^{\Dim}$. When the gradient is evaluated at a specific value $\widetilde{x}$, we write $\evaluat*{\nabla_{x} f }{\widetilde{x}}$. The derivative with respect to a scalar is denoted as $\frac{\partial}{\partial x_{1}} f$ so that $\nabla_{x} f = (\frac{\partial}{\partial x_{1}} f, \dots, \frac{\partial}{\partial x_{\Dim}} f )'$.

\section{Preliminaries} 
\label{se:prelim}
This section first relates the autocovariance matrices of the latent and observed processes in Section \ref{se:relationship}.  Our goals are formulated and the estimators are clarified in Section \ref{se:QoI}.   We then state our assumptions and main results in Section \ref{se:assumptions} and introduce notation that allows our results and proofs to be compactly presented in Section \ref{sec:notation}.

\subsection{Autocovariance matrices and their relationships} 
\label{se:relationship}
Recall that $\Gamma_{X}(h) = \E[X_{t+h}X_{t}'] - \E[X_{t+h}] \E[ X_{t}]'$ denotes the autocovariance matrix function at lag $h$ of a stationary time series and $R_{X}(h)$ its corresponding lag $h$ autocorrelation.  Individual entries are denoted by $\Gamma_{X,ij}(h)$ and $R_{X,ij}(h)$ for $i,j=1,\dots,\Dim$.

The ACVFs of $\{ X_{t} \}$ and $\{ Z_{t} \}$ in \eqref{eq:subo} can be related using Hermite expansions for the components in $G$:
\begin{equation} 
\label{eq:Hermiteexpansion}
G_{i}(z) = \sum_{k =0}^{\infty} \frac{c_{i,k}}{k!} H_{k}(z)
\end{equation}
with the $k$th Hermite polynomial defined as
\begin{equation*}
H_{k}(z) = (-1)^{k} e^{z^2/2} \frac{\partial^{k}}{\partial z^{k}} e^{-z^2/2}.
\end{equation*}
The Hermite coefficients are
\begin{equation} 
\label{eq:Hermitecoeff}
c_{i,k} = \E(G_{i}(Z_{i,0})H_{k}(Z_{i,0}));
\end{equation}
see Chapter 5 in \cite{pipiras2017long} for more details on Hermite polynomials.

As stated in Proposition 5.1.4 in \cite{pipiras2017long}, the autocovariances of $\{ X_{t} \}$ can be written as
\begin{align} 
\label{eq:covHermitecor}
\Gamma_{X}(h)  = \left( \sum_{k=1}^{\infty} \frac{c_{i,k} c_{j,k}}{k!} R_{Z,ij}(h)^k \right)_{i,j = 1, \dots, \Dim};
\end{align}
the corresponding autocorrelation matrix is
\begin{align*}
R_{X}(h)  = \left( \sum_{k=1}^{\infty} \frac{c_{i,k} c_{j,k}}{k!} \frac{1}{ (\Gamma_{X,ii}(0) \Gamma_{X,jj}(0))^{\frac{1}{2}}} R_{Z,ij}(h)^k \right)_{i,j =1, \dots, \Dim}.
\end{align*}
Following \cite{jia2023latent}, we write
\begin{align*}
L(u) = (L_{ij}(u))_{i,j = 1, \dots, \Dim}
\hspace{0.2cm} \text{ with } \hspace{0.2cm} 
L_{ij}(u) = \sum_{k=1}^{\infty} \frac{c_{i,k} c_{j,k}}{k!} \frac{u^k}{ (\Gamma_{X,ii}(0) \Gamma_{X,jj}(0))^{\frac{1}{2}}}
\end{align*}
and, for what we will refer to as the link function, 
\begin{align} 
\label{eq:function_ell}
\ell(u) = (\ell_{ij}(u))_{i,j = 1, \dots, \Dim}
\hspace{0.2cm} \text{ with } \hspace{0.2cm} 
\ell_{ij}(u) = \sum_{k=1}^{\infty} \frac{c_{i,k} c_{j,k}}{k!} u^k.
\end{align}
From Section 2.3 in \cite{jia2023latent}, the Hermite coefficients \eqref{eq:Hermitecoeff} admit the representation
\begin{equation} 
\label{eq:representation_coef}
c_{i,k} = \frac{1}{\sqrt{2\pi}} \sum_{n =0}^{\infty} e^{ -Q_{i,n}^2/2} H_{k-1}( Q_{i,n} )
\end{equation}
with $Q_{i,n} = \Phi^{-1}(C_{i,n})$ and $C_{i,n} = \Prob[X_{i,t} \leq n]$. In general, $Q_{i,n}$ depends on $\theta_{i}$, which contains all CDF parameters for the $i$th component series. We write $Q_{n}(\theta_{i}):=Q_{i,n}$ and $C_{n}(\theta_{i}):=C_{i,n}$ to emphasize dependence on $\theta_{i}$. 
We use a different notation $C_{i,n} = \Prob[X_{i,t} \leq n]$ instead of potentially more natural $F_{i}(n)$ to bring out the dependence on $\theta_{i}$ as the argument (e.g.\ to be differentiated with respect to).

Proposition 2.1 in \cite{jia2023latent} provides an explicit representation for the first derivative of $\ell$ in \eqref{eq:function_ell}. While that result lies in a univariate setting, the representation here is stated for each individual entry of the covariance matrices and allows different marginal distributions in each dimension.

\begin{proposition} 
\label{prop:prop2.1}
Let $\ell$ be as in \eqref{eq:function_ell}. Then, for $u \in (-1,1)$,
\begin{equation} 
\label{eq:ellprime2}
\ell^{\prime}_{ij}(u) = \frac{1}{2\pi \sqrt{1-u^2}} \sum_{n_{0},n_{1} = 0}^{\infty} \exp\left(-\frac{1}{2 (1-u^2)} (Q_{i,n_{0}}^{2} + Q_{j,n_{1}}^{2} - 2uQ_{i,n_{0}}Q_{j,n_{1}})\right).
\end{equation}
\end{proposition}

As Proposition \ref{prop:prop2.1} shows, $\ell_{ij}$ has a positive derivative and is therefore monotonically strictly increasing. The function $\ell_{ij}$ maps $[-1, 1]$ into $[\ell_{ij}(-1),\ell_{ij}(1)]$ with $\ell_{ii}(1) = \Gamma_{ii}(0)$ and crosses zero at $u=0$ due to \eqref{eq:function_ell}.
Note, that since $\ell_{ij}$ is a strictly increasing function with $\ell_{ij}(0) = 0$, negative or positive correlation of the latent process gets inherited by the observed count series. Figure 2 in the supplementary material of \cite{jia2023latent} plots several link functions. Since $\ell_{ij}$ is strictly increasing, so is its inverse $g_{ij} = \ell_{ij}^{-1}$, which is defined on $[\ell_{ij}(-1),\ell_{ij}(1)]$. Later, we extend the domains of $\ell$ and $g$ to define an appropriate estimator. To use a plug in estimator for the autocovariance matrices, $\ell$ and $g$ will need to be evaluated at any point.

Related to \eqref{eq:ellprime2}, we introduce, for $x,y \in (-1,1)$, the function 
\begin{equation*} 
h_{ij}(x,y) = \frac{1}{2\pi \sqrt{1-x}} \sum_{n_{0},n_{1} = 0}^{\infty} \exp\left(-\frac{1}{2 (1-x)} (Q_{i,n_{0}}^{2} + Q_{j,n_{1}}^{2} - 2yQ_{i,n_{0}}Q_{j,n_{1}})\right).
\end{equation*}
Observe that $\ell'_{ij}(u) = h_{ij}(u^{2},u)$. We also define 
\begin{equation} 
\label{eq:def:smallzcapitalZ}
z_{ij}(y) := \frac{1}{2\pi } (h_{ij}(0,y))^{-1},
\hspace{0.2cm}
Z_{ij}(x) := \frac{1}{2\pi \sqrt{1-x}} (h_{ij}(x,\Gamma_{Z,ij}(h)))^{-1}.
\end{equation}
The corresponding matrix-valued quantities can be defined accordingly; for instance, $\ell'(\Gamma_{Z}(h)) = h(\Gamma_{Z}(h)^{\odot 2}, \Gamma_{Z}(h))$. 

\begin{example} 
\label{example:bernoulli1}
Proposition \ref{prop:prop2.1} explicitly provides the derivative of the link function in a general setting that allows any marginal distribution. This representation simplifies in many cases. For instance, suppose that the marginals in dimension $i$ have a Bernoulli distribution with success probability $p_{i}$ for $i=1, \dots, \Dim$. Then $G$ in \eqref{eq:subo} is $G: \RR^{\Dim} \to \{0,1\}^{\Dim}$ and its components can be written as $G_{i}(z) = F_{i}^{-1}(\Phi(z)) = \mathds{1}_{ \{ z > \Phi^{-1}(1-p_{i}) \} }$. In particular, $Q_{i,n} = \Phi^{-1}(1-p_{i}) =: q_{i}$. The first derivative of the link function simplifies to
\begin{align*}
\ell'_{ij}(u) 
= \frac{1}{2 \pi (1-u^2)^{\frac{1}{2}}} \exp\left( - \frac{1}{2(1-u^2)} (q_{i}^2 + q_{j}^2 - 2uq_{i}q_{j}) \right).
\end{align*}
In particular, for $i=j$, the expression in the exponential reduces to $q_{i}^2/(1+u)$. 
\end{example}

\begin{example}
Following up on the previous Example \ref{example:bernoulli1}, we simplify even further by assuming that we have a $d=2$ dimensional count series with Bernoulli marginals and $p_{i} = \frac{1}{2}$ for $i=1,2$.
Then, $G_{i}(z) = F_{i}^{-1}(\Phi(z)) = \mathds{1}_{ \{ z > \Phi^{-1}(1-p_{i}) \} } = \mathds{1}_{ \{ z > 0 \} }$ for $i =1,2$.
With $\Gamma_{Z}(h) = (\Gamma_{Z,ij}(h))_{i,j=1,2}$, we get for the observed count series
\begin{equation*}
\Gamma_{X}(h) = \frac{1}{2 \pi}
\begin{pmatrix}
\arcsin(\Gamma_{Z,11}(h)) & \arcsin(\Gamma_{Z,12}(h)) \\
\arcsin(\Gamma_{Z,21}(h)) & \arcsin(\Gamma_{Z,22}(h))
\end{pmatrix}.
\end{equation*}
In other words, in this particular case, where the CDF parameters are known, it is possible to calculate an explicit link function. This example also illustrates how the count series inherits negative or positive correlation of the latent process.
We refer to Lemma 4.1 in \cite{CountHurricane} and also Section III of \cite{van1966spectrum} for this example.
\end{example}

\subsection{Quantities of interest and estimation} 
\label{se:QoI}
Our goal here is to extract information about the latent process $\{ Z_{t} \}$, including cases of high-dimensions.   Section \ref{se:AppVAR} below estimates coefficient matrices for the latent process, which is assumed to follow a VAR model.

To take advantage of the sparsity assumed on the latent process, it is crucial to derive deviation bounds in a specific norm. Introduce the set $\mathcal{K}(2s) = \{ v \in \RR^{\Dim \Lag} : \| v \| \leq 1, \| v \|_{0} \leq 2s \}$ and define the mapping
\begin{align} 
\label{eq:matrixmapping}
A \mapsto \| A \|_{s} := \sup_{v \in \mathcal{K}(2s)} | v' A v|.
\end{align}
Besides submultiplicativity, the mapping in \eqref{eq:matrixmapping} satisfies all properties of a matrix norm.  We derive and collect some properties of \eqref{eq:matrixmapping} in Section \ref{se:normproperties}. All results will be expressed in terms of $\| \cdot \|_{s}$. The mapping allows us to impose sparsity on a matrix $A$ through the vectors $v$. Since all properties of \eqref{eq:matrixmapping} in Section \ref{se:normproperties} are also satisfied by the spectral norm, our main result remains true for the spectral norm. This said, results for the spectral norm do not seem to be particularly relevant in our setting.

We aim to derive concentration inequalities for estimates of $\acmfZ = (\Gamma_{Z}(r-s))_{r,s=1, \dots, \Lag}$, which collects autocovariances at different low lags.   That is, our goal is to bound the probability
\begin{equation*}
\Prob\left[ \| \acmfZhat - \acmfZ \|_{s} > \delta \right]
\end{equation*}
for $\delta > 0$.  A natural estimator of $\acmfZ$ is $\ell^{-1}(\acmfXhat)$ for a known link function $\ell$. We also use the notation $g = \ell^{-1}$, and $g_{ij} = \ell^{-1}_{ij}$.   In principle, $\acmfXhat$ could take on any value, even beyond the domain $[\ell_{ij}(-1), \ell_{ij}(1)]$ of $g_{ij}$. Therefore, we assume throughout the paper that $g_{ij}(x) = g_{ij}(\ell_{ij}(1))$ for all $x > \ell_{ij}(1)$  (and similarly at the left border).

Assuming that the observations have a zero mean, the autocovariance matrix $\acmfX = (\Gamma_{X}(r-s))_{r,s=1, \dots, \Lag}$ can be estimated as $\acmfXhat = N^{-1} \mathcal{X}_{X}' \mathcal{X}_{X}$ with $N = T - \Lag$ and
\begin{equation} 
\label{eq:mathcalXofZ1}
\mathcal{X}_{X} =
\begin{pmatrix}
X_{\Lag}' & \dots  & X_{1}' \\
\vdots     & \ddots & \vdots \\
X_{T-1}'  & \dots  & X_{T-\Lag}'
\end{pmatrix}.
\end{equation}
With a slight abuse of notation, we write both $\acmfX = \ell(\acmfZ)$ and $\Gamma_{X}(h) = \ell(\Gamma_{Z}(h))$.

The link function $\ell_{ij}$ defined in \eqref{eq:function_ell} depends on the CDF parameter vectors $\theta_{i}$ and $\theta_{j}$. We collect these marginal distribution parameters across dimensions into $\theta = (\theta'_{1}, \dots, \theta'_{\Dim})'$ and denote their estimators as $\widehat{\theta}$ and $\widehat{\theta}_{i}$ for $i =1, \dots, \Dim$. Each component series is allowed a different marginal distribution, which potentially depends on a different number of parameters. Let $K_{i}$ denote the number of parameters for the marginal distribution of the $i$th component series and set $K = \max_{i=1,\dots,\Dim} K_{i}$. Explicit dependence on the vectors $\theta_{i}$ is often omitted in our notations; antipodally, we write $Q_{n}(\theta_{i}) = Q_{i,n}$ and $C_{n}(\theta_{i}) = C_{i,n}$ when this dependence needs to be emphasized. 

Estimates of $\ell$ and $g$ and other related functions are written as $\widehat{\ell}$ and $\widehat{g}$ and are computed by replacing $C_{i,n}$ with $\widehat{C}_{i,n} = C_{n}( \widehat{\theta}_{i})$ in $Q_{i,n} = \Phi^{-1}(C_{i,n})$. Our estimator for $\acmfZ$ is written as $\acmfZhat = \widehat{g}(\acmfXhat)$ for an unknown link function.   

Our main results relate the probability of autocovariance matrix estimator deviations of the latent process to the analogous probability in the observations. To state these, several assumptions are needed.

\subsection{Assumptions} 
\label{se:assumptions}
We will work with two sets of assumptions. The first assumption set applies to our main result, which relates the probability of how much $\acmfZhat$ deviates from $\acmfZ$ to the analogous probabilities in the observed $\{ X_t \}$.  Assumption \ref{ass:finitemoments2} is shown to hold for several common count distributions under Assumption \ref{ass:finitemoments} in the Appendix \ref{se:Discussion on assumption}.

\begin{mainassump} 
\label{as:bound for entries}
There is a constant $\conZ \in (0,1)$ such that $| \Gamma_{Z,ij}(h) | < \conZ$ for $h \neq 0$, $i, j=1, \dots, \Dim$ and $| \Gamma_{Z,ij}(0) | < \conZ$ for all $i \neq j$.
\end{mainassump}

\begin{mainassump} 
\label{ass:finitemoments}
For each $\theta_{i} = (\theta_{i1}, \dots, \theta_{iK_{i}})'$, there exists an open neighborhood $S$ of $\theta_{i}$ such that  
the moment $\sup_{\theta_{i} \in S}\E[ |X_{i,t}|^p] = \sup_{\theta_{i \in S}}\E_{\theta_{i}}[ |X_{i,t}|^p] <\infty$ for some $ p > 2$.
\end{mainassump}

\begin{mainassump} 
\label{ass:finitemoments2}
For each $\theta_{i} = (\theta_{i1}, \dots, \theta_{iK_{i}})'$, there exists an open neighborhood $S$ of $\theta_{i}$ such that 
\begin{align*}
&\sup_{\theta_{i} \in S} \sum_{n = 0}^{\infty} (1-C_{n}(\theta_{i}))^{-\frac{1}{2}} 
\sum_{j=1}^{K_{i}}
\left| \frac{\partial}{\partial \theta_{ij}} C_{n}(\theta_{i}) \right| 
\\&=
\sup_{\theta_{i} \in S} \sum_{n = 0}^{\infty} (\Prob[X_{i,t} > n])^{-\frac{1}{2}} 
\sum_{j=1}^{K_{i}}
\left| \frac{\partial}{\partial \theta_{ij}} \Prob[X_{i,t} > n] \right| < \infty.
\end{align*}
\end{mainassump}

\begin{mainassump} 
\label{ass:finitemoments3}
For each $\theta_{i} = (\theta_{i1}, \dots, \theta_{iK_{i}})'$, there exist an open neighborhood $S$ of $\theta_{i}$ and at least one $n$ such that $\inf_{\theta_{i} \in S} C_{n}(\theta_{i}) > 0$.
\end{mainassump}

Note that we require our moment conditions to hold uniformly in a neighborhood around $\theta_{i}$. This allows us to infer finiteness on a compact subset of the parameter space of $\theta_{i}$.



The following two assumptions ensure consistent estimation of $\acmfZ$ with a $\log(\Lag \Dim^2)/T$ convergence rate. Other bounds may also lead to consistency. Section \ref{se:AppVAR} establishes these assumptions for a causal VAR series $\{ Z_t \}$. 

\begin{CIassump} 
\label{ass:CI-VAR}
There exist finite positive constants $c_{1}$ and $c_{2}$ such that for any $v \in \mathcal{K}(2s)$ and any $\delta >0$,
\begin{equation*}
\Prob [ | v' ( \acmfXhat - \acmfX ) v | >  c_{0}(s) \delta ] \leq \
c_{1} \exp \left(- c_{2} N \min\{\delta, \delta^{2} \} \right)
\end{equation*}
with $N=T-\Lag$ and $c_{0}(s)\geq1$.
\end{CIassump}

\begin{CIassump} 
\label{ass:CItheta}
There exist finite positive constants $c_{1}$ and $c_{2}$ such that for any $\varepsilon >0$, 
\begin{equation*}
\Prob[ \| \widehat{\theta} - \theta \|_{\maxF} > \varepsilon] \leq c_{1} \Dim K \exp\left(-c_{2} T \min\{ \varepsilon, \varepsilon^{2} \} \right).
\end{equation*}
\end{CIassump}

The quantity $c_{0}(s)$ in Assumption \ref{ass:CI-VAR} should be of lower order than $q$ and describe some kind of lower dimensional structure imposed through the vectors $v \in \mathcal{K}(2s)$. We provide a discussion on Assumption \ref{ass:CI-VAR} and $c_{0}(s)$ in Section \ref{se:VARdisc} below. The assumption that $c_{0}(s)\geq1$ has only aesthetic reasons.

\subsection{Moments} 
\label{sec:notation}
We now collect some notation used in the proofs of the main results. Proposition \ref{prop:prop2.1} shows that $\ell$ is differentiable on the open interval $(-1,1)$ and gives an explicit form for this derivative.  In our proofs, the on- and off-diagonal elements in the difference $\acmfZhat - \acmfZ$ are handled separately.   In fact, bounds for the off-diagonal elements can be expressed in terms of $\ell'$ due to \eqref{eq:ellprime2}.

Our main results are cast in terms of moment conditions for $\{ X_{t} \}$. The following notation allows us to express our results compactly. As the diagonal terms will be treated separately, quantities used to bound these terms are considered first.  Set 
\begin{equation} 
\label{eq:moment0}
\Delta_{i} 
= \sum_{n = 0}^{\infty} n \left\| \nabla_{\theta_{i}}  C_{i,n} \right\|_{1}
= \frac{1}{\sqrt{2\pi}} \sum_{n=0}^{\infty} \exp\left(-\frac{1}{2u} Q^{2}_{i,n} \right) n
\| \nabla_{\theta_{i}} Q_{i,n} \|_{1},
\end{equation}
which is uniformly bounded in a neighborhood of $\theta_{i}$ by Lemma \ref{le:diagonalsderivativethetafinite} after the assumptions in  \ref{ass:finitemoments} and \ref{ass:finitemoments2} are invoked. 
The second representation emphasizes the similarity to the subsequently introduced quantities and is further explained in the proof of Lemma \ref{le:diagonalsderivativethetafinite}.

Also define
\begin{equation} 
\label{eq:momentlikemk}
m^{(k)}_{i}(u) = \frac{1}{\sqrt{2\pi }} \sum_{n = 0}^{\infty} \exp\left(-\frac{1}{2 u} Q^{2}_{i,n} \right) | Q_{i,n} |^{k}.
\end{equation}
Lemma \ref{le:finitesumpdelta} ensures that \eqref{eq:momentlikemk} is uniformly bounded in a neighborhood of $\theta_{i}$ under Assumption \ref{ass:finitemoments}. Furthermore, Assumption \ref{ass:finitemoments3} ensures that \eqref{eq:momentlikemk} is uniformly bounded from below since at least one of the summands in the series expansion is nonzero. 
We will occasionally write $m^{(k)}_{\theta_{i}}(u)$ to emphasize that $m^{(k)}_{i}(u)$ depends on $\theta_{i}$.  Analogously, we introduce the derivative
\begin{equation} 
\label{eq:momentlikemkderiv}
\mu_{i}^{(k)}(u) = 
\frac{1}{\sqrt{2\pi}} \sum_{n=0}^{\infty} \exp\left(-\frac{1}{2u} Q^{2}_{i,n} \right) |Q_{i,n}|^k 
\| \nabla_{\theta_{i}} Q_{i,n} \|_{1}.
\end{equation}
This expression will be further simplified in Lemma \ref{le:derivativethetafinite} and is uniformly bounded in a neighborhood of $\theta_{i}$ under Assumption \ref{ass:finitemoments2}.  See Section \ref{se:Discussion on assumption} for further discussion on Assumption \ref{ass:finitemoments2}.

Our probability bounds will be expressed in terms of 
\begin{equation} 
\label{eq:diagonal_d(epsilon)}
\Delta(\varepsilon) = \max_{i=1, \dots, \Dim} \sup_{\theta_{i} \in \Theta(\varepsilon)} \Delta_{i}
\end{equation}
for the diagonal terms and the following four quantities:
\begin{equation} 
\begin{gathered}
\label{eq:mumoments}
M(\conZ,\varepsilon)
=
\max_{\substack{ i =1, \dots, \Dim \\ k = 0, 3}} \sup_{\theta_{i} \in \Theta(\varepsilon)} m_{i}^{(k)}\left( 1+\conZ \right),
\\
\mu(\conZ,\varepsilon)
=
\max_{\substack{ i =1, \dots, \Dim \\ k = 0, 3}} \sup_{\theta_{i} \in \Theta(\varepsilon)} \mu_{i}^{(k)} \left( 1+\conZ \right),
\end{gathered}
\end{equation}
\begin{equation} 
\begin{gathered}
\label{eq:mumoments2}
M_{1}(\conZ,\varepsilon)
=
\max_{i =1, \dots, \Dim} \sup_{\theta_{i} \in \Theta(\varepsilon)}
\frac{1}{
\left( m_{i}^{(0)}\left( 1-\conZ \right) \right)^{2}},
\\
M_{2}(\conZ,\varepsilon)
=
\max_{\substack{ i =1, \dots, \Dim \\ k = 0, 2}} \sup_{\theta_{i} \in \Theta(\varepsilon)}
\frac{ \left(m^{(k)}_{i}\left( \frac{1}{1-\conZ} \right) \right)^{2} }{ \left( m^{(0)}_{i}\left( 1-\conZ \right) \right)^{4}}
\end{gathered}
\end{equation}
with $\Theta(\varepsilon) = \{ \theta \in [\Delta_{11},\Delta_{12}] \times \cdots \times [\Delta_{K_{i}1},\Delta_{K_{i}2}] ~|~ \| \theta - \theta_{i} \|_{\max} \leq \varepsilon \}$. Here, $[\Delta_{r1},\Delta_{r2}]$ refers to the interval of admissible estimates for the $r$th parameter in the vector $\theta_{i}$. 

Note that Lemmas \ref{le:diagonalsderivativethetafinite}--\ref{le:derivativethetafinite} in combination with Assumptions \ref{ass:finitemoments}--\ref{ass:finitemoments3} ensure that the expressions \eqref{eq:moment0}--\eqref{eq:momentlikemkderiv} are uniformly bounded in a neighborhood of $\theta_{i}$. Consequently, we can infer that the quantities \eqref{eq:diagonal_d(epsilon)}--\eqref{eq:mumoments2} in our probability bounds are finite on the compact set $\Theta(\varepsilon)$.

\begin{example} 
\label{example:bernoulli2}
Returning to Example \ref{example:bernoulli1}, consider the Bernoulli case where $\Prob[ X_{i,t}=1]= p_{i}$ for $i=1, \dots, \Dim$. Then, \eqref{eq:momentlikemk} and \eqref{eq:momentlikemkderiv} simplify to
\begin{equation*}
m^{(k)}_{i}(u) = \frac{1}{\sqrt{2\pi }} \exp\left(-\frac{1}{2 u} q^{2}_{i} \right) | q_{i} |^{k}
\end{equation*}
and
\begin{align*}
\mu_{i}^{(k)}(u) 
= 
\frac{1}{\sqrt{2\pi}} \exp\left(-\frac{1}{2u} q^{2}_{i} \right) |q_{i}|^k 
\left| \frac{\partial }{\partial p_{i}} q_{i} \right|
= 
\frac{1}{\sqrt{2\pi}} \exp\left(-\frac{1-u}{2u} q^{2}_{i} \right) |q_{i}|^k.
\end{align*}
since $\frac{\partial }{\partial p_{i}} q_{i}  = \frac{1}{\phi( q_{i})}$, where $\phi(\cdot)$ is the standard normal density function.
\end{example}

\section{Concentration inequalities for autocovariance matrix estimates} 
\label{se:concentration}
This section presents our main results.  These results allow one to make inferences about $\{ Z_t \}$ from $\{ X_t \}$.   Proofs are delegated to Section \ref{se:proof} and subsequent appendices.

\begin{proposition} 
\label{prop:cons}
Suppose that Assumptions \ref{as:bound for entries}--\ref{ass:finitemoments3} hold. Then, for any $\delta, \widetilde{\delta}, \varepsilon, \widetilde{\varepsilon} >0$, 
\begin{equation} 
\label{eq:prop:cons}
\begin{aligned}
\Prob\left[ \| \widehat{\bm{\Gamma}}_{Z} - \bm{\Gamma}_{Z} \|_{s} > Q(\acmfZ) \delta \right]
&\precsim
\Prob[ \| \acmfXhat - \acmfX \|_{s} > \delta \wedge \widetilde{\delta} ] 
+
\Prob[ \| \acmfXhat - \acmfX \|_{s}^{2} > \delta  ]
\\&\hspace{0.2cm}+
\Prob [  \| \widehat{\theta} - \theta \|_{\maxF} > \delta \wedge \varepsilon \wedge \widetilde{\varepsilon} ]
+
\Prob[
\| \widehat{\theta} - \theta \|^2_{\maxF} > \delta ],
\end{aligned}
\end{equation}
where $Q(\acmfZ) := Q(\widetilde{\delta}, \varepsilon, \widetilde{\varepsilon}, \acmfZ)$ is a function of $\Delta$, $M, \mu$, $M_{1}$, and $M_{2}$ as defined in \eqref{eq:diagonal_d(epsilon)}, \eqref{eq:mumoments}, and \eqref{eq:mumoments2}. The explicit representation of $Q(\acmfZ)$ can be found in Section \ref{se:constants}.
\end{proposition}

We refer the reader to Section \ref{se:roadmap} for an outline of the proof and its associated  challenges.  Section \ref{se:roadmap} also provides intuition on how $\delta, \widetilde{\delta}, \varepsilon$ and $\widetilde{\varepsilon}$ arise on the right hand side of \eqref{eq:prop:cons}.   While the result could be simplified by writing, for instance, 
$\Prob[ \| \widehat{\theta} - \theta \|^2_{\maxF} > \delta ] = \Prob[ \| \widehat{\theta} - \theta \|_{\maxF} > \delta^{\frac{1}{2}} ]$, we purposely bookkept second order terms to emphasize the proof's strategy. For some insight on the quantity $Q(\acmfZ)$, we refer to Example \ref{example:bernoulli3} below, where we discuss some of its behavior in the Bernoulli case.

Note that once we estimate $\theta$s for specific marginal distributions, the parameters might be restricted to a certain interval. Then, $\theta$ does not only have to satisfy 
$\theta_{i} - \varepsilon \leq \widehat{\theta}_{i} \leq \theta_{i} + \varepsilon$ but also comply with the parameter space imposed by the marginal distribution. In other words, $\widehat{\theta}_{i}$ lies in the intersection of $[\theta_{i} - \varepsilon, \theta_{i} + \varepsilon]$ and the parameter space given by the marginal distribution. 
The prior knowledge of a set of feasible estimates results in our constants depending on the set $\Theta(\varepsilon) = \{ \theta \in [\Delta_{11},\Delta_{12}] \times \cdots \times [\Delta_{K_{i}1},\Delta_{K_{i}2}] ~|~ \| \theta - \theta_{i} \|_{\max} \leq \varepsilon \}$. Here, $[\Delta_{r1},\Delta_{r2}]$ refers to the interval of admissible estimates for the $r$th parameter in the vector $\theta_{i}$. 
Since $\varepsilon$ is supposed to go to zero, it will eventually be small enough to ensure for $[\theta_{i} - \varepsilon, \theta_{i} + \varepsilon]$ to be shorter than the interval of admissible estimates.
Similarly, $\widetilde{\delta}$ will be small enough to ensure
$\widehat{\Gamma}_{X,ij}(h) \leq \widetilde{\delta} + \Gamma_{X,ij}(h) < \ell_{ij}(1)$ for $ i \neq j$ since $ \Gamma_{X,ij}(h) < \ell_{ij}(\conZ) < \ell_{ij}(1)$.

To simplify Proposition \ref{prop:cons}, one can work with a single quantity $\nu = \delta \wedge \widetilde{\delta} \wedge \varepsilon \wedge \widetilde{\varepsilon}$, as stated in Corollary \ref{cor:CI} below. The corollary is a consequence of Proposition \ref{prop:cons} and illustrates how convergence rates for the observed process can be used to extract a high probability bound on deviations between $\acmfZhat$ and $\acmfZ$.

\begin{corollary} 
\label{cor:CI}
Suppose that Assumptions \ref{as:bound for entries}--\ref{ass:finitemoments3} and \ref{ass:CI-VAR}--\ref{ass:CItheta} hold. Then, for any $\delta, \widetilde{\delta}, \varepsilon, \widetilde{\varepsilon} >0$, there exists finite constants $c_{i,1}, c_{i,2} > 0$, $i =1,2$, such that
\begin{equation} \label{eq:cor:CI}
\begin{aligned}
&
\Prob\left[ \| \widehat{\bm{\Gamma}}_{Z} - \bm{\Gamma}_{Z} \|_{s} > Q(\acmfZ) c_{0}(s) \delta \right]
\\&\hspace{1cm}\leq 
c_{1,1} \exp \left( -c_{1,2}N \min\left\{ 1, \nu^2, \nu / c_{0}(s) \right\} + 2s \log(\Lag \Dim) \right)
\\
&\hspace{2cm}+
c_{2,1} \Dim K \exp\left(-c_{2,2} T \min\{1, \nu^{2}\} \right)
\end{aligned}
\end{equation}
with $Q(\acmfZ) := Q(\widetilde{\delta}, \varepsilon, \widetilde{\varepsilon}, \acmfZ)$ defined in Proposition \ref{prop:cons}, $c_{0}(s) \geq 1$ as in Assumption \ref{ass:CI-VAR} and $\nu = \delta \wedge \widetilde{\delta} \wedge \varepsilon \wedge \widetilde{\varepsilon}$.
\end{corollary}

From Corollary \ref{cor:CI}, one can further infer the existence of constants $c_{1}$ and $ c_{2}>0$ such that
\begin{equation} \label{eq:consequenceCor3.1}
\begin{aligned}
&
\Prob \left[ \| \widehat{\bm{\Gamma}}_{Z} - \bm{\Gamma}_{Z} \|_{s} > Q(\acmfZ) c_{0}(s) \delta \right]
\\&\leq 
c_{1} \exp \Big( -c_{1,2}N \min\left\{ 1, \nu^2, \nu / c_{0}(s) \right\} 
\\&\hspace{1cm}+ 2s \log(\Lag \Dim) -c_{2,2} T \min\{1, \nu^{2}\} + \log(\Dim K) \Big)
\\&\leq 
c_{1} \exp \left( -c_{2}N \min\left\{ 1, \nu^2, \nu / c_{0}(s) \right\} \right),
\end{aligned}
\end{equation}
where we assume that $N \succsim \max\{ c_{0}(s)/\nu, \nu^{-2}, 1 \} \max\{ s\log(\Lag \Dim), \log( \Dim K) \}$.  Choosing $\delta = \widetilde{\delta} = \varepsilon = \widetilde{\varepsilon} =  \sqrt{ \frac{\log(\Lag \Dim^2)}{N}}$ so that $\nu = \sqrt{ \frac{\log(\Lag \Dim^2)}{N}}$, we infer the existence of positive constants $c_{1}$ and $c_{2}$ such that
\begin{align}
&
\Prob \left[ \| \widehat{\bm{\Gamma}}_{Z} - \bm{\Gamma}_{Z} \|_{s} > Q(\acmfZ) c_{0}(s) \sqrt{ \frac{\log(\Lag \Dim^2)}{N}} \right]
\nonumber
\\&\leq 
c_{1} \exp \left( -c_{2} \min\{ N, \log(\Lag \Dim^2), \sqrt{N \log(\Lag \Dim^2)/ c^2_{0}(s)} \}  \right)
\nonumber
\\&\leq 
c_{1} \exp\left( -c_{2} \log(\Lag \Dim^2) \right) 
\label{al:ghghasasas}
\\&
= c_{1}q^{-c_{2}},
\nonumber
\end{align}
whenever $N \succsim c^2_{0}(s) \log(q)$ with $q = \Lag \Dim^2$.

For a sense of what parts of $\| \widehat{\bm{\Gamma}}_{Z} - \bm{\Gamma}_{Z} \|_{s}$ contribute to the probability bound, we focus on Corollary \ref{cor:CI} and compare the statement to some existing results. The first part of the bound, $c_{1,1} \exp \left( -c_{1,2}N \min\left\{ 1, \nu^2, \nu / c_{0}(s) \right\} + 2s \log(\Lag \Dim) \right)$, can be separated into two parts. 
The first summand in the exponential bound is almost the same as the one in Proposition 2.4 in \cite{basu2015regularized} 
for expressions of the form $v'( \widehat{\bm{\Gamma}}_{Z} - \bm{\Gamma}_{Z} )v$. This makes use of  Gaussianity for $\{ Z_{t} \}$. 
The only difference is $\nu / c_{0}(s)$. However, $\nu / c_{0}(s)$ results from our second order terms and is asymptotically negligible as seen in \eqref{al:ghghasasas}.
The second summand $2s \log(\Lag \Dim)$ in the exponential arises after applying Lemma F.2 in the supplementary material of \cite{basu2015regularized}; this effectively extends results uniformly over all sparse vectors $v \in \mathcal{K}(2s)$.

The major difference between our bounds and existing results for autocovariance estimation of Gaussian series is the summand $c_{2,1} \Dim K \exp\left(-c_{2,2} T \min\{1, \nu^{2}\} \right)$. 
\\
This exponential term comes from estimating the unknown parameters $\theta$ in the marginal distributions and link function $\ell$.  In particular, for a known link function, this summand does not show up in the bound; see also Lemma \ref{le:main2APP} below.

\begin{example} 
\label{example:bernoulli3}
Returning to Example \ref{example:bernoulli2}, consider the Bernoulli case where $\Prob[ X_{i,t}=1]= p_{i}$ for $i=1, \dots, \Dim$. We aim to shed some light on the constant $Q(\acmfZ)$ in \eqref{eq:prop:cons}.
The quantity $Q(\acmfZ)$ is a function of \eqref{eq:mumoments} and \eqref{eq:mumoments2} which are functions of \eqref{eq:momentlikemk} and \eqref{eq:momentlikemkderiv}. 
In the Bernoulli case, those values depend on the success probabilities $p_{i}$, $i=1,\dots,\Dim,$ and $\bm{c}_{Z}$ as in Assumption \ref{as:bound for entries}.
As an example, we consider $M_{1}(\conZ,\varepsilon)$ for known probabilities $p_{i}$ such that $\varepsilon = 0$. Then, 
\begin{equation}
M_{1}(\conZ,0)
=
\max_{i =1, \dots, \Dim}
\frac{1}{
\left( m_{i}^{(0)}\left( 1-\conZ \right) \right)^{2}}
=
\max_{i =1, \dots, \Dim}
\frac{1}{
\left( \frac{1}{\sqrt{2\pi }} \exp\left(-\frac{1}{2 (1-\conZ)} q^{2}_{i} \right) \right)^{2}}
\end{equation}
with $q_{i} := \Phi^{-1}(1-p_{i})$. In Figure \ref{fig:DGP2}, we plot $M_{1}(\conZ,0)$ for $p:=p_1=\dots = p_{\Dim}$ and as function of $\bm{c}_{Z} \in (0,1)$ and $p \in (0,1)$. 
As expected, for strong temporal and cross-sectional correlation, i.e. $\bm{c}_{Z}$ close to one as well as for very small and large probabilities, the constants can get quite large. 
In general, with growing dimension $\Dim$, and more and more values $p_{i}$ contributing, we can get values close to the boundary which results in large $Q(\acmfZ)$. One can avoid this phenomenon by restricting the set of possible values to a closed subset of $(0,1)$.
\end{example}

\begin{figure}
\centering
\scalebox{0.7}{\input{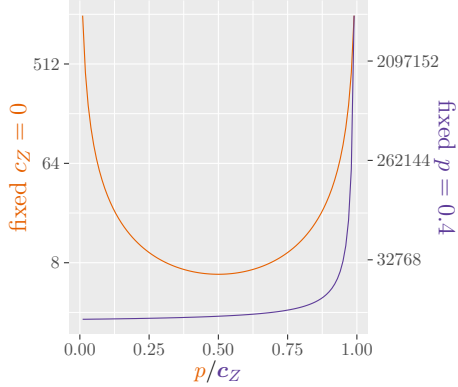}}
\caption{The function $M_{1}(\conZ,0)$ in (2.15) in dependence of $\conZ$ and the success probability $p$.}
\label{fig:DGP2}
\vspace{-0.5cm}
\end{figure}

\section{Sparse estimation for latent VAR processes} 
\label{se:AppVAR}
In this section, we suppose that $\{Z_{t}\}_{t \in \ZZ}$ follows a causal VAR model of order $\VO$ (VAR($\VO$)); that is,
\begin{equation} 
\label{def:VAR}
Z_{t} = \sum_{u=1}^{\VO} \TM_{u} Z_{t-u} + \varepsilon_{t},
\hspace{0.2cm}
t \in \ZZ,
\end{equation}
for some $\TM_{u} \in \RR^{\Dim \times \Dim}$ and white noise series $\{ \varepsilon_{t}\}_{t \in \ZZ}$ characterized by
\begin{equation} 
\label{eq:Sigmaepsilon}
\E[ \varepsilon_{t}] = 0, 
\hspace{0.2cm}
\E[\varepsilon_{t} \varepsilon_{t}'] = \Sigma_{\varepsilon},
\hspace{0.2cm}
\E[ \varepsilon_{s} \varepsilon_{t}'] = 0
\hspace{0.2cm}
\text{ for }
s \neq t.
\end{equation}
We assume that the VAR($\VO$) process is causal; that is, 
\begin{equation} 
\label{eq:StabilityConditionVAR}
\det(\TM(z)) \neq 0,
\hspace{0.2cm}
\text{ for } 
\hspace{0.2cm} 
|z| \leq 1, z \in \CC
\hspace{0.2cm}
\text{ with } 
\hspace{0.2cm} 
\TM(z) = I_{\Dim} - \TM_{1}z - \dots - \TM_{\VO}z^{p}.
\end{equation}
Our goal here is to estimate the transition matrices $\TM_{1}, \dots, \TM_{\VO}$ in \eqref{def:VAR} sparsely in a possibly high-dimensional regime.

While our main result in Proposition \ref{prop:cons} is proven in generality without imposing any assumptions on the observed series, our results here require a couple of more assumptions. 

\begin{VARassump} \label{Ass:V1}
The function $G$ in \eqref{eq:subo} satisfies $G: \RR^{\Dim} \to [a,b]^{\Dim}$.
\end{VARassump}

\begin{VARassump} \label{Ass:V2}
The parameters in $\theta_{i}$ satisfy $\theta_{i} = \E [X_{i,t}]$ which allows us to estimate $\theta_{i}$ via $\widehat{\theta}_{i} = \frac{1}{T} \sum_{t = 1}^{T} X_{i,t}$.
\end{VARassump}

These assumptions are satisfied for all discrete distributions having a finite support and whose population mean coincides with the unknown parameter characterizing the respective distribution. Examples are Bernoulli, binomial and hypergeometric distributions. In contrast, it excludes discrete distributions with an infinite support set like such as Poisson. We refer the reader to the second part of Section \ref{se:VARdisc} for a detailed discussion and potential extensions.

In the two subsequent sections, estimation of the coefficient matrices of $\{ Z_t \}$ (Section \ref{sec:estimationprocedure}) and their estimator's theoretical properties (Section \ref{theoreticalproperties}) are presented. We conclude with a discussion of our results and the required assumptions (Section \ref{se:VARdisc}) and an illustration of the convergence rate (Section \ref{Illustration of estimation error}).

\subsection{Estimation procedure} 
\label{sec:estimationprocedure}
To estimate the VAR coefficients sparsely, we adopt a procedure proposed by \cite{basu2015regularized}. However, their procedure needs to be modified so as to base inferences on the observations. 

The VAR($\VO$) model can be written in a linear models form as
\begin{equation} 
\label{eq:VARlinearform}
\begin{pmatrix}
Z_{\VO+1}' \\
\vdots \\
Z_{T}'
\end{pmatrix}
=
\begin{pmatrix}
Z_{\VO}' & \dots & Z_{1}' \\
\vdots & \ddots & \vdots \\
Z_{T-1}' & \dots & Z_{T-\VO}'
\end{pmatrix}
\begin{pmatrix}
\TM_{1}' \\
\vdots \\
\TM_{\VO}'
\end{pmatrix}
+
\begin{pmatrix}
\varepsilon_{\VO+1}' \\
\vdots \\
\varepsilon_{T}'
\end{pmatrix}
\hspace{0.2cm}
\text{ or } 
\hspace{0.2cm} 
\mathcal{Y}_{Z} = \mathcal{X}_{Z} B_{0} + \mathcal{E}.
\end{equation}
A vectorized version of \eqref{eq:VARlinearform} is then seen to be
\begin{align}
\vecop(\mathcal{Y}_{Z}) 
&= \vecop( \mathcal{X}_{Z} B_{0}) + \vecop(\mathcal{E}) \nonumber
\\&= (I_{\Dim} \otimes \mathcal{X}_{Z})\vecop( B_{0}) + \vecop(\mathcal{E}), 
\label{al:VARlinear}
\end{align}
or
\begin{equation}
Y = Z\beta_{0} + E, 
\label{al:vecVAR}
\end{equation}
where $Y \in \RR^{N \Dim }$ with $N = T-p$, $Z \in \RR^{N \Dim \times q}$ with $q = \VO \Dim^2$, $\beta_{0} \in \RR^{q}$, and $E \in \RR^{N \Dim}$. To estimate $\TM_{1}, \dots, \TM_{\VO}$, we simply estimate $\beta_{0} \in \RR^{q}$ in \eqref{al:vecVAR}.  To impose sparsity on the transition matrices, we state the following assumption on the true vector $\beta_{0}$.

\begin{VARassump} \label{ass:sparsity}
Assume that $\beta_{0}$ is an $s$-sparse vector; that is, $\| \beta_{0} \|_{0} = \sum_{u =1}^{\VO} \| \vecop( \TM_{u}) \|_{0} = s$.
\end{VARassump}

Following \cite{basu2015regularized}, we define a LASSO-type estimator for $\beta_{0}$ by
\begin{equation} 
\label{eq:estimatorforbeta}
\widehat{\beta} = \argmin_{\beta \in \RR^{q}} \Big( -2\beta' \widehat{\gamma} + \beta' \widehat{\Gamma} \beta
+ \lambda_{N} \| \beta \|_{1} \Big)
\end{equation}
with
\begin{equation} 
\label{eq:gammahatGammahat}
\begin{gathered}
\widehat{\gamma}
= \vecop( \widehat{\bm{\gamma}}_{Z} )
= \vecop ( \widehat{g}( \widehat{\bm{\gamma}}_{X} ) ), \\
\widehat{\Gamma} 
= I_{\Dim} \otimes \acmfZhat
= I_{\Dim} \otimes \widehat{g}( \acmfXhat ),
\end{gathered}
\end{equation}
where $\bm{\gamma}_{Z} = (\Gamma_{Z}(1)', \dots, \Gamma_{Z}(\VO)')' $, $\acmfZ = (\Gamma_{Z}(r-s))_{r,s = 1,\dots, \VO}$, and their estimated counterparts $\widehat{\bm{\gamma}}_{Z}$ and $\acmfZhat$ are defined analogously. Furthermore, for the observed series $\{X_{t}\}$, we set
$\widehat{\bm{\gamma}}_{X} = N^{-1} \mathcal{X}_{X}' \mathcal{Y}_{X}$ and 
$\acmfXhat = \frac{1}{N} \mathcal{X}_{X}' \mathcal{X}_{X}$, where $\mathcal{Y}_{X}$ and $\mathcal{X}_{X}$ are quantities analogous to $\mathcal{Y}_{Z}$ and $\mathcal{X}_{Z}$ above formed by replacing $\{ Z_{t} \}_{t=1,\dots,T}$ with $\{X_{t}\}_{t=1,\dots,T}$ in \eqref{eq:VARlinearform}.

In contrast to \cite{basu2015regularized}, the population quantities $\gamma$ and $\Gamma$ in \eqref{eq:gammahatGammahat} cannot be estimated through the VAR series $\{ Z_{t} \}$, which is unobserved.  Instead, we need to estimate $\gamma$ and $\Gamma$ from $\{ X_{t} \}_{t=1, \dots, T}$.

\subsection{Theoretical properties} 
\label{theoreticalproperties}
The theoretical properties of $\widehat{\beta}$ in \eqref{eq:estimatorforbeta} are derived here.  We state consistency results in a possibly high-dimensional regime, allowing $\Dim$ and $T$ to go to infinity. 

As suggested in \cite{loh2012high} and \cite{basu2015regularized}, we first establish consistency under a restricted eigenvalue condition and a deviation bound.  Subsequently, we verify that these conditions are satisfied by $\{ Z_{t} \}$. Section \ref{se:VARdisc} clarifies how our proofs differ from those in \cite{loh2012high} and \cite{basu2015regularized}.

\textbf{Restricted eigenvalue:}
A symmetric matrix $\widehat{\Gamma} \in \RR^{q \times q}$ satisfies the restricted eigenvalue condition with curvature $\alpha >0$ and tolerance $\tau >0$ if
\begin{equation} 
\label{eq:restrictedEV}
x' \widehat{\Gamma} x \geq \alpha \| x \|^2 - \tau \| x \|_{1}^{2}
\hspace{0.2cm}
\text{ for all }
\hspace{0.2cm} x \in \RR^{q}.
\end{equation}
We write $\widehat{\Gamma} \sim RE(\alpha,\tau)$ for short.

\textbf{Deviation bound:}
There exists a deterministic function $\mathcal{Q}(\beta_{0})$ such that
\begin{equation} 
\label{eq:deviationbound}
\| \widehat{\gamma} - \widehat{\Gamma} \beta_{0} \|_{\max}
\leq
\mathcal{Q}(\beta_{0}) \sqrt{\frac{\log(q)}{N}}.
\end{equation}

The following proposition ensures consistent estimation of the coefficients for a latent VAR($\VO$) model obeying the restricted eigenvalue condition and deviation bounds.
The statement is effectively the same as Proposition 4.1 in \cite{basu2015regularized} for observed VAR models but using the estimators \eqref{eq:gammahatGammahat} instead of those for an observed series.
\begin{proposition} 
\label{prop:consistentbeta}
Suppose that $\widehat{\Gamma}$ in \eqref{eq:gammahatGammahat} satisfies the restricted eigenvalue condition \eqref{eq:restrictedEV} (we write $\widehat{\Gamma} \sim RE(\alpha,\tau))$ with $s \tau \leq \alpha/32$ and that $(\widehat{\Gamma}, \widehat{\gamma} )$ in \eqref{eq:gammahatGammahat} satisfies the deviation bound in \eqref{eq:deviationbound}.  Then, for any $\lambda_{N} \geq 4 \mathcal{Q}(\beta_{0}) \sqrt{\frac{\log(q)}{N}}$, 
\begin{equation*}
\begin{gathered}
\| \widehat{\beta} - \beta_{0} \|_{1} \leq 64 s \frac{\lambda_{N} }{\alpha},
\hspace{0.2cm}
\| \widehat{\beta} - \beta_{0} \| \leq 16 \sqrt{s} \frac{\lambda_{N} }{\alpha},
\\
(\widehat{\beta} - \beta_{0})' \widehat{\Gamma} (\widehat{\beta} - \beta_{0}) \leq 128 s \frac{\lambda^2_{N} }{\alpha}.
\end{gathered}
\end{equation*}
\end{proposition}

The following two lemmas provide sufficient conditions for when the restricted eigenvalue condition \eqref{eq:restrictedEV} and deviation bound \eqref{eq:deviationbound} hold for a latent VAR series.  To state the lemmas, recall the causality assumption in \eqref{eq:StabilityConditionVAR} on the VAR($\VO$) model and set $\mu_{\max}(\mathcal{A}) := \max_{|z| = 1} \lambda_{\max}(\mathcal{A}^{*}(z)\mathcal{A}(z))$, with $\mathcal{A}(z) = I_{\Dim} - \sum_{j =1}^{p} A_{j} z^j$.

\begin{lemma}[Verifying restricted eigenvalue] 
\label{le:RE}
Suppose that Assumptions \ref{as:bound for entries}--\ref{ass:finitemoments3} hold and that the latent process follows a causal VAR model. Then, there are constants $c_{1}, c_{2} > 0$ such that for all $N \succsim \max\{ c_{0}(s) \nu^{-1}, \nu^{-2}, 1 \} \max\{ s\log(\Dim \VO), \log( \Dim K) \}$, with probability at least $1 - c_{1} \exp\left( -c_{2}N \min\{ 1, \nu^2, \nu/c_{0}(s) \} \right)$, 
\begin{equation*}
\widehat{\Gamma} \sim RE(\alpha,\tau),
\end{equation*}
where 
\begin{equation} \label{eq:alphataunu}
\begin{aligned}
\alpha = \frac{\lambda_{\min}(\Sigma_{\varepsilon})}{2 \mu_{\max}(\mathcal{A})},
\hspace{0.2cm}
\tau = \alpha \max\{ c_{0}(s) \nu^{-1}, \nu^{-2},1\} \frac{\log( \Dim \VO)}{N},
\hspace{0.2cm}
\nu = \frac{\lambda_{\min}(\Sigma_{\varepsilon})}{54 \mu_{\max}(\mathcal{A}) Q(\acmfZ) c_{0}(s)}
\end{aligned}
\end{equation}
with $Q(\acmfZ)$ defined as in Proposition \ref{prop:cons} and $c_{0}(s) =s$.
\end{lemma}

\begin{lemma}[Verifying deviation bound] 
\label{le:DB}
Suppose that Assumptions \ref{as:bound for entries}--\ref{ass:finitemoments3} hold and that the latent process is a causal VAR. Then, there are constants $c_{1}, c_{2} > 0$ such that for all $N \succsim \max\{ c_{0}(s)\nu^{-1}, \nu^{-2}, 1 \} s\log(\Dim \VO)$, with probability at least $1 - c_{1} \exp\left( -c_{2}N \min\{ 1, \nu^2, \nu/c_{0}(s) \} \right)$, 
\begin{equation*}
\| \widehat{\gamma} - \widehat{\Gamma} \beta_{0} \|_{\max}
\leq
\sqrt{\frac{\log(q)}{N}} \mathcal{Q}(\beta_{0}),
\end{equation*}
where
\begin{equation*}
\mathcal{Q}(\beta_{0}) = Q(\acmfZ) c_{0}(s)
\end{equation*}
with $Q(\acmfZ)$ as in Proposition \ref{prop:cons}, $B_{0}$ in \eqref{eq:VARlinearform}, $e_{q,i}$ denotes the $i$th basis vector of $\RR^{q}$ and $c_{0}(s) \leq s$.
\end{lemma}

\subsection{Discussion}\label{se:VARdisc} A discussion on our results, including a comparison to existing literature and potential relaxations of our assumptions, is now provided.

\textit{Comparison to \cite{basu2015regularized}:}
The statement of Lemma \ref{le:RE} is analogous to the one of Proposition 4.2 in \cite{basu2015regularized}. The proof is very similar and leads to almost the same relation between the sample size $N = T-p$ and the dimension $d$. More precisely, we get $N \succsim \max\{ c_{0}(s)\nu^{-1}, \nu^{-2}, 1 \} \max\{ s\log(\Dim \VO), \log( \Dim K) \}$ and Proposition 4.2 in \cite{basu2015regularized} states that $N \succsim \max\{ \nu^{-2}, 1 \} \max\{ s\log(\Dim \VO)\}$. Since the number of unknown parameters $K$ in the marginal distributions is relatively small, it is expected that $\max\{ s\log(\Dim \VO), \log( \Dim K) \} = s\log(\Dim \VO)$. Furthermore, $c_{0}(s)\nu^{-1}$ is due a second order approximation used in the proofs and will also vanish for certain choices of $\nu$.

The statement of Lemma \ref{le:DB} is analogous to the one of Proposition 4.3 in \cite{basu2015regularized}. This said, the deviation bound is significantly different in our setting and results in a different rate between the sample size $N=T-p$ and the dimension $d$. \cite{basu2015regularized} require $N \succsim \max\{ \nu^{-2}, 1 \} 2\log(\Dim \VO)$, while our relation $N \succsim \max\{ c_{0}(s)\nu^{-1}, \nu^{-2}, 1 \} s\log(\Dim \VO)$ includes the sparsity parameter $s$. However, this relation only impacts the assumptions on the verification of the deviation bound in Lemma \ref{le:DB}. Proposition \ref{prop:consistentbeta} is however not impacted and requires the same assumptions on $N$ and $d$ as Proposition 4.1 in \cite{basu2015regularized}. To compare our proof of Lemma \ref{le:DB} with that of Proposition 4.3 in \cite{basu2015regularized}, assume that $\{Z_{t}\}$ is observed. Then, estimation in \eqref{eq:estimatorforbeta} is done through 
\begin{equation} 
\label{eq:BASUetalgammahatGammahat}
\begin{gathered}
\widehat{\gamma}
= \vecop( \widehat{\bm{\gamma}}_{Z} )
= \vecop ( \mathcal{X}_{Z}' \mathcal{Y}_{Z} ), \\
\widehat{\Gamma} 
= I_{\Dim} \otimes \acmfZhat
= I_{\Dim} \otimes \mathcal{X}_{Z}'\mathcal{X}_{Z}/N.
\end{gathered}
\end{equation}
The estimators in \eqref{eq:BASUetalgammahatGammahat} simply replace those in \eqref{eq:gammahatGammahat}. \cite{basu2015regularized} use the fact that $ \widehat{\gamma} - \widehat{\Gamma} \beta_{0} = (I_{\Dim} \otimes \mathcal{X}_{Z}') \vecop(E)/N = \vecop( \mathcal{X}_{Z}' E)/N$, with $E$ in \eqref{al:vecVAR}, to extract a concentration bound on $ \| \mathcal{X}_{Z}' E/N \|_{\max}$.  Since $\{ Z_t \}$ is unobserved in our setting, we need to reduce the issue to inference of autocovariance matrices to use our main result Proposition \ref{prop:cons}.  Concluding, our procedure affects the convergence rate for the deviation bound but not for the main consistency result.  As a byproduct, our procedure offers an alternative proof for the deviation bound in \cite{basu2015regularized}. 

\textit{Discussion of assumptions part I:}
Our analysis of the latent VAR setup is restricted to bounded $G_{i}$ in \eqref{eq:subo} as formalized in Assumption \ref{Ass:V1}. Our proof route will verify Assumptions \ref{ass:CI-VAR} and \ref{ass:CItheta} so that Corollary \ref{cor:CI} can be applied. While Proposition \ref{prop:cons} allows us to phrase concentration bounds on the autocovariance matrices of the latent process through those of the observed one, it remains to verify Assumption \ref{ass:CI-VAR} for an estimator of the autocovariances of the observed process.  Since the observed process is a function $G$ of a Gaussian vector series, bounded $G_{i}$s permit use of Hoeffding's inequality for Markov chains \citep{fan2021hoeffding}.   Most existing concentration bounds for functionals have been developed under Lipschitz continuity; a general result for our setting requires results for the Hermite polynomials in \eqref{eq:Hermiteexpansion}. The case for unbounded functions of Gaussian random variables is more challenging and has been studied only in special cases. Here, \cite{adamczak2015exponential} develop concentration bounds for polynomials of certain degree. However, their bounds are difficult to compute explicitly since they rely on a generalization of the Frobenius norm for multi-indexed matrices. \cite{adamczak2015concentration} generalized \cite{fan2021hoeffding} for unbounded functions, but require much more restrictive conditions than a Markov chain.

Assumption \ref{Ass:V2} is chosen for simplicity. The purpose of this section is to illustrate the usefulness of our main result Proposition \ref{prop:cons}. In principle, we should be able to prove concentration results for other estimators for the parameters of the marginal distribution. For example for the variance using the methods of moments.

\textit{Discussion of assumptions part II:} We add here a discussion on Assumption \ref{ass:CI-VAR} and $c_{0}(s)$ therein.
If the causal VAR($\VO$) series $\{ Z_t \}$ is observed, Assumption \ref{ass:CI-VAR} is effectively satisfied by Proposition 2.4 in \cite{basu2015regularized}. 
In Proposition 2.4 of \cite{basu2015regularized}, $c_{0}(s) = 2\pi \mathcal{M}(f_{Z}, s)$ with
\begin{equation*}
\mathcal{M}(f_{Z}, s) := \max_{S \subset \{1, \dots, q\}, |S| \leq s}  \esssup_{\lambda \in [-\pi, \pi]} \| f_{Z(S)}(\lambda) \|,
\end{equation*}
where $f_{Z(S)}$ describes the spectral density of the subprocess $\{ Z(S) \} = \{ \widetilde{Z}_{i,t} ~|~ i \in S \}_{t\in \ZZ}$ with 
$\widetilde{Z}_{t} = (Z'_{t}, Z'_{t-1}, \dots, Z'_{t-\VO+1})'$. 
In Section \ref{Illustration of estimation error} below, we illustrate our convergence rate and compare it with the one for observed VAR models.


\subsection{Illustration of estimation error} \label{Illustration of estimation error}
In this section, we illustrate our theoretical results on latent VAR estimation in a numerical study. We demonstrate how the estimation error of our estimator for the latent VAR \eqref{eq:estimatorforbeta} scales with the sample size $T$ and dimension $d$. 
We simulate $\Dim$-dimensional count series with marginal Bernoulli distributions and latent Gaussian VAR(1) process. The success probabilities were randomly sampled from $p_{i} \in (0.4,0.7)$, $i=1,\dots, \Dim$. For different values of $\Dim$ ($\Dim=5,10,15,20, 25, 30$, or in terms of the number of parameters, $\Dim^2 = 25, 100,225,400, 625, 900$, resp.), we generated sparse coefficient matrices $\Psi$ ($\beta_{0} = \vecop(\Psi))$ with sparsity $s$ ($s=(13, 28, 43, 58, 73, 88)$, resp.), that is, the number of non-zero entries; see Assumption \ref{ass:sparsity}. Then, we applied \eqref{eq:estimatorforbeta} with tuning parameter $\lambda_{N} = \sqrt{\log(d^2)/N}$ on samples of sizes $T\in \{200, 300, 400, 500, 1000, 2000\}$. The $\ell_{2}$-error of estimation $\| \widehat{\beta} - \beta_{0} \|^2$ is plotted in the first row of Figure \ref{fig1}. The left panel displays the errors for different values of $\Dim$, plotted against the sample sizes $T$. As expected, the errors increase with the dimension. The right panel displays the estimation errors against the rescaled sample size $(T-1)/(s \log (\Dim^2))$. Note that this rate is suggested by our theoretical result in Proposition \ref{prop:consistentbeta}. That is, the result suggests that $\| \widehat{\beta} - \beta_{0} \|^2$ is proportional to $s \lambda^2_{N} = s \log (\Dim^2)/(T-1) = 1/((T-1)/(s \log (\Dim^2)))$ and hence is proportional to $1/x$ if plotted against $x= (T-1)/(s \log (\Dim^2))$.

As discussed in Section 4.3, it is not expected to get a better rate when the link function is known. 
The estimation of $\theta$ to get estimates of the link function should only impact the rate when $K$ (the number of unknown parameters) is large. To illustrate that, the second row in Figure \ref{fig1} displays plots for the same setting as described above but for known link function. As one can see, there is no significant difference between the plots of the first and second rows.

As also discussed in Section 4.3, our theoretical results claim that we get the same convergence rate as when we observe a process which follows a VAR model and estimate the transition matrices directly from the data. Those results are proved in \cite{basu2015regularized}. For reference, the third row of Figure \ref{fig1}, displays the results for the observed process following the VAR(1) models in our simulation study. As one can see, the estimation error does scale well with the same rate as for the latent models, the difference being in the scale of the vertical axes where the observed process naturally has smaller estimation errors.

\begin{figure}
\centering
\scalebox{.7}{\input{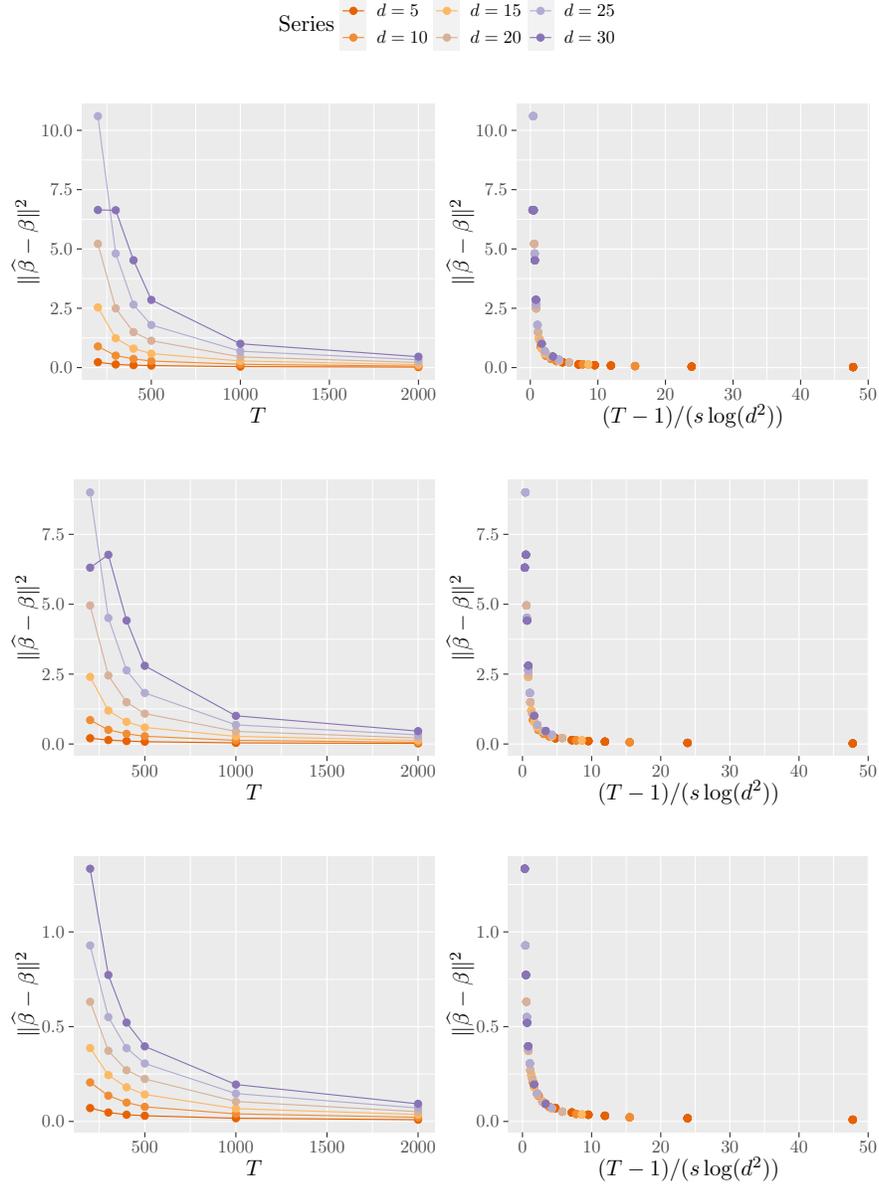}}
\caption{Estimation error of the LASSO $\| \widehat{\beta} - \beta\|^2$ plotted against $T$ (left column) and the theoretical rate $(T-1)/(s \log(d^2))$ (right column) for three different settings:
\textit{First row:} $\Dim$-dimensional count series with marginal Bernoulli distributions and latent Gaussian VAR(1) process.
\textit{Second row:} $\Dim$-dimensional count series with marginal Bernoulli distributions and latent Gaussian VAR(1) process with known link function.
\textit{Third row:} $\Dim$-dimensional Gaussian VAR(1).}
\label{fig1}
\end{figure}

\section{Conclusions} \label{se:conclusion}
This work considered a possibly high-dimensional count time series model whose correlation structure is determined through the correlation of an underlying latent Gaussian process. We derived a relation between consistent estimation of the autocorrelation matrices of the latent process and the autocovariance matrices of the observed process.

Several theoretical challenges needed to be addressed to ensure consistent estimation of the latent model.  These include estimation of the link function based on unknown CDF parameters, the non-differentiability of the link function around unity, and the issue of high-dimensionality; see Section \ref{se:roadmap}. Assuming that the latent process follows a VAR model, our results ensure consistent estimation of the transition matrices of a VAR series at the same rate as for an observable VAR series. 

While we illustrated our results on the example of a latent VAR model, our main result can be used for other models. For instance, it is conceivable to assume that the latent process follows a dynamic factor model where the factors follow a stationary VAR structure. Consistency results would require concentration bounds on functionals of a dynamic factor model. A similar question concerns how to derive consistency for possibly unbounded functions of the latent Gaussian process as discussed in Section \ref{se:VARdisc}.

Further questions include extensions to non-stationary models, particularly those involving covariates, estimation of related model parameters of the latent process like the VAR order, and extensions to spatial settings.

\appendix

\section{Proofs of the main results} \label{se:proofs:main}
We give a short roadmap of the proof of Proposition \ref{prop:cons} (Section \ref{se:roadmap}). Subsequently, we restate Proposition \ref{prop:cons} in terms of explicit constants which are omitted in the statement (Section \ref{se:constants}) and continue then with the detailed proofs of our main results (Section \ref{se:proof}).

\subsection{Roadmap} \label{se:roadmap}

We give here a roadmap of the proof of Proposition \ref{prop:cons} to emphasize the overall structure and some of the main tools. Recall that
\begin{equation*}
\acmfZ = g(\acmfX)
\hspace{0.2cm}
\text{ with }
\hspace{0.2cm}
g = \ell^{-1}.
\end{equation*}

In contrast to the link function $\ell$, its first derivative $\ell'$ admits a known explicit representation which is stated in Proposition \ref{prop:prop2.1}. We would like to take advantage of this representation, and identify the following three issues to deal with.

First, the derivative $\ell'(u)$ is defined only on the open interval $u \in (-1,1)$. Therefore, we have to deal with the cases $u=1$ and $u=-1$ separately. Assumption \ref{as:bound for entries} excludes both cases for all off-diagonal elements of $\acmfZ$. However, all diagonal elements of $\bm{\Gamma}_{Z}$ satisfy $\bm{\Gamma}_{Z,rr} = \Gamma_{Z,ii}(0) = 1$ since we assume that the latent process is standard Gaussian (has variance one).
The general strategy will be to apply a first-order Taylor approximation to the diagonal elements and a second-order Taylor expansion to the off-diagonals. 
This procedure will be applied to a range of different functions related to $\ell$. One of those functions is $ g = \ell^{-1}$. In order to exclude all diagonal elements and to treat those separately, we introduce the notation
\begin{equation} \label{eq:gdot}
g_{\bullet,rs}(v) = 
\begin{cases}
1, &\hspace{0.2cm} \text{ for all } r, s \text{ such that } \bm{\Gamma}_{Z,rs} = \Gamma_{Z,ii}(0),\\
g_{rs}(v), &\hspace{0.2cm} \text{ else.}
\end{cases}
\end{equation}
This notation will also be used for functions other than $g$ but refers to the same edges $r,s$ to be equal to one. 

The second and third issues need a little more context. Our goal is to bound the distance 
$\|\widehat{g}(\acmfXhat) - g(\acmfX)\|_{s}$ by $\|\acmfXhat - \acmfX\|_{s}$ involving the autocovariances of the observed process. An application of the mean value theorem is not sufficient. The function $g$ itself is estimated through the unknown parameters $\theta$ of the marginal distributions. Second, we address this latter issue by controlling the error we make by estimating $\theta$; the procedure is explained in more detail in Remark \ref{re:intersectingwithevent}.

Third, consider the simpler problem of proving the bound $\| g(\acmfXhat) - g(\acmfX)\|_{s} \leq C \|\acmfXhat - \acmfX\|_{s}$ and note that applying the mean value theorem componentwise yields $\| g(\acmfXhat) - g(\acmfX)\|_{s} \leq \| g'(\Sigma) \odot (\acmfXhat - \acmfX)\|_{s}$ for some $\Sigma$. Since $g'(\Sigma) $ is not necessarily positive semidefinite, one can show that the optimal $C$ such that
$\| g(\acmfXhat) - g(\acmfX)\|_{s} \leq C \max_{i =1, \dots, \Dim} |g_{ii}'(\Sigma)| \|\acmfXhat - \acmfX\|_{s}$ is of order $\sqrt{2s}$. This would weaken our bounds. 
For this reason, we will work with a second-order Taylor expansion; see also Remark \ref{re:Wegkamp} and \cite{wegkamp2016adaptive}. 

The following tree diagram gives an idea of how the subsequent sections contribute to the different proof steps.
Each node of the graph below represents a distance which needs to be controlled with high probability. With each layer of the diagram, we reduce the problem further, up to the point where we only rely on the observed process.
The edges represent proof steps which are justified in subsequent sections. Note that we omit the consideration of the diagonal elements, those are studied in Section \ref{se:diagonal_elements}.

In the fourth and fifth layer of the diagram, the distances have a superscript $r =1,2$. This is due to the use of second-order Taylor expansions applied to different functions. 

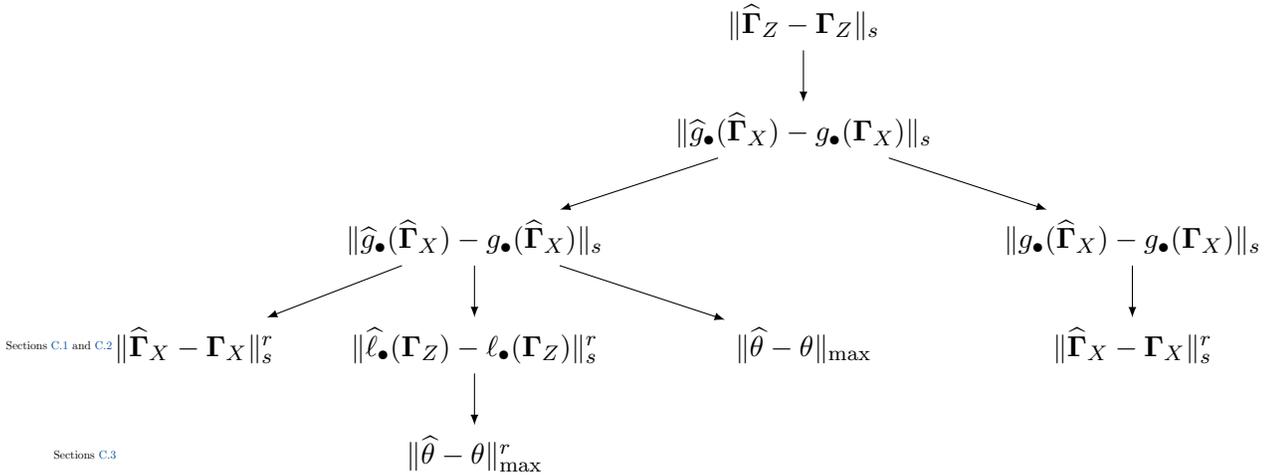
\hspace{-2cm}
\begin{tikzpicture}[scale=0.4, transform shape]
\matrix[matrix of math nodes, column sep=20pt, row sep=20pt] (mat)
{
\phantom{1}				&								&	\|\acmfZhat - \acmfZ \|_{s}								& 		\\
\phantom{2}				&								&	\|\widehat{g}_{\bullet}(\acmfXhat) - g_{\bullet}(\acmfX)\|_{s} 	& 		\\ 
\phantom{3}				&	\|\widehat{g}_{\bullet}(\acmfXhat) - 
							g_{\bullet}(\acmfXhat)\|_{s}		&													&	\|g_{\bullet}(\acmfXhat) -
																												g_{\bullet}(\acmfX)\|_{s}	\\
\|\acmfXhat - \acmfX\|^{r}_{s}	&	\|\widehat{\ell}_{\bullet}(\acmfZ) - 
							\ell_{\bullet}( \acmfZ)\|^{r}_{s}		&	\| \widehat{\theta} - \theta\|_{\max}						&	\|\acmfXhat - \acmfX\|^{r}_{s}	\\
\phantom{5}				&	\|\widehat{\theta} - \theta\|^{r}_{\max}	&													&		\\
};

    \draw[->,>=latex] (mat-1-3) -- (mat-2-3);
    \draw[->,>=latex] (mat-2-3) -- (mat-3-2);
    \draw[->,>=latex] (mat-2-3) -- (mat-3-4);
    \draw[->,>=latex] (mat-3-2) -- (mat-4-1);
    \draw[->,>=latex] (mat-3-2) -- (mat-4-2);
    \draw[->,>=latex] (mat-3-2) -- (mat-4-3);
    \draw[->,>=latex] (mat-3-4) -- (mat-4-4);
    \draw[->,>=latex] (mat-4-2) -- (mat-5-2);

\node[anchor=east] at ([xshift=-70pt]mat-4-1) {Sections \ref{se:inverted_link_function} and \ref{se:oneoverellprime} \ };
\node[anchor=east] at ([xshift=-70pt]mat-5-1) {Sections \ref{se:link_function}};

\end{tikzpicture}

Note that the statement of Proposition \ref{prop:cons} involves a series of parameters $\delta,\widetilde{\delta},\varepsilon,\widetilde{\varepsilon}$. While $\delta$ is supposed to control the difference between $\acmfZhat$ and $\acmfZ$, the remaining ones only appear on the right hand side of the relation \eqref{eq:prop:cons}. The remaining parameters $\widetilde{\delta},\varepsilon,\widetilde{\varepsilon}$ ensure respectively that $\|\acmfXhat - \acmfX\|_{s}$, $\|\widehat{\theta} - \theta\|_{\max}$ and $\|\widehat{\ell}_{\bullet}(\acmfZ) - \ell_{\bullet}( \acmfZ)\|_{s}$ are small.   
See also Remarks \ref{re:differentiablityanddeltatilde} and \ref{re:intersectingwithevent} below for further discussions.

\subsection{Statements of results in Section \ref{se:concentration} with explicit constants}
\label{se:constants}
We provide here the explicit constants which enter our main result Proposition \ref{prop:cons}. For completeness, we restate the proposition.
\begin{proposition} 
\label{prop:consAPP}
Suppose that Assumptions \ref{as:bound for entries}--\ref{ass:finitemoments3} hold. Then, for any $\delta, \widetilde{\delta}, \varepsilon, \widetilde{\varepsilon} >0$, 
\begin{equation} 
\label{eq:prop:consAPP}
\begin{aligned}
\Prob\left[ \| \widehat{\bm{\Gamma}}_{Z} - \bm{\Gamma}_{Z} \|_{s} > Q(\acmfZ) \delta \right]
&\precsim
\Prob[ \| \acmfXhat - \acmfX \|_{s} > \delta \wedge \widetilde{\delta} ] 
+
\Prob[ \| \acmfXhat - \acmfX \|_{s}^{2} > \delta  ]
\\&\hspace{0.2cm}+
\Prob [  \| \widehat{\theta} - \theta \|_{\maxF} > \delta \wedge \varepsilon \wedge \widetilde{\varepsilon} ]
+
\Prob[
\| \widehat{\theta} - \theta \|^2_{\maxF} > \delta ]
\end{aligned}
\end{equation}
with $$Q(\acmfZ) := Q(\widetilde{\delta}, \varepsilon, \widetilde{\varepsilon}, \acmfZ)  = 4\max\{ D(\acmfZ),
4R(\acmfZ), 2U(\acmfZ), T(\acmfZ) \} \max\{ S^2(\acmfZ), 1\}$$ and the quantities $D,R,S,T,U$ defined below.
\end{proposition}

The quantities entering \eqref{eq:prop:consAPP} are:
\begin{equation} 
\label{eq:quantitiesmainresult}
\begin{aligned}
D(\acmfZ) := D(\varepsilon, \acmfZ) &= M_{1}^{\frac{1}{2}} (1/2,\varepsilon) 2 \max\{ 3 \Delta(\varepsilon), 1 \},
\\
R(\acmfZ) := R(\varepsilon, \acmfZ) &= \Big( 8 \pi M_{1}(0,\varepsilon) + 24 \pi \frac{1}{(1-\conZ^2)^2} M_{2}(\conZ, \varepsilon) \Big) \| \acmfZ \|_{s},
\\
S(\acmfZ) := S(\varepsilon, \acmfZ) &= \frac{12}{(1-\conZ^2)^{\frac{7}{2}}} 
M(\conZ,\varepsilon) \mu(\conZ,\varepsilon) \| \acmfZ \|_{s},
\\
T(\acmfZ) := T(\varepsilon, \widetilde{\delta}, \acmfZ) &= \frac{6}{1-\bm{c}(\widetilde{\delta})^2}  M_{1}(\bm{c}(\widetilde{\delta}),\varepsilon) M_{2}(\bm{c}(\widetilde{\delta}),\varepsilon),
\\
U(\acmfZ) := U(\varepsilon, \widetilde{\varepsilon}, \acmfZ) &= T(\varepsilon, \max\{ S^2(\acmfZ), 1\} \widetilde{\varepsilon}, \acmfZ), 
\end{aligned}
\end{equation}
for some constant $\bm{c}(\widetilde{\delta}) \in (0,1)$, which depends on $\widetilde{\delta}$; here, $\Delta$, $M, \mu$ and $M_{1}, M_{2}$ are defined in \eqref{eq:diagonal_d(epsilon)}, \eqref{eq:mumoments}, and \eqref{eq:mumoments2}.

Proposition \ref{prop:cons}/\ref{prop:consAPP} is a consequence of Lemmas \ref{le:main1APP} and \ref{le:main2APP} below, which respectively cover the cases when $g$ is estimated or known.
\begin{lemma} 
\label{le:main1APP}
Suppose that Assumptions \ref{as:bound for entries}--\ref{ass:finitemoments3} hold. Then, for any $\delta,\widetilde{\delta},\varepsilon,\widetilde{\varepsilon}>0$,
\begin{align*}
&
\Prob\left[ \| \widehat{g}_{\bullet}(\acmfXhat) - g_{\bullet}(\acmfXhat) \|_{s} > Q_{1}(\acmfZ) \delta \right]
\\&\hspace{1cm}\precsim
\Prob\left[ \| \acmfXhat - \acmfX \|_{s} > \delta \wedge \widetilde{\delta} \right]
+ 
\Prob\left[ \| \acmfXhat - \acmfX \|_{s}^{2} > \delta \right]
\notag
\\&\hspace{2cm}+
\Prob\left[ \| \widehat{\theta} - \theta \|_{\maxF} > \delta \wedge \varepsilon \wedge \widetilde{\varepsilon} \right]
+
\Prob\left[
\| \widehat{\theta} - \theta \|^2_{\maxF} > \delta \right],
\end{align*}
where $$Q_{1}(\acmfZ) := Q_{1}(\widetilde{\delta}, \varepsilon, \widetilde{\varepsilon}, \acmfZ) = 4\max\{ 4R(\acmfZ), 2U(\acmfZ), T(\acmfZ) \} \max\{ S^2(\acmfZ), 1 \}$$ with $R, S, T$, and $U$ as in \eqref{eq:quantitiesmainresult}.
\end{lemma}

\begin{lemma} 
\label{le:main2APP}
Suppose that Assumptions \ref{as:bound for entries}--\ref{ass:finitemoments3} hold. Then, for any $\delta, \widetilde{\delta} >0$,
\begin{align*}
\Prob \left[ \| g_{\bullet}(\acmfXhat) - g_{\bullet}(\acmfX) \|_{s} > Q_{2}(\acmfZ) \delta \right]
&\precsim
\Prob \left[ \| \acmfXhat - \acmfX \|_{s} > \delta \wedge \widetilde{\delta} \right]
\\&\hspace{1cm}+
\Prob \left[ \| \acmfXhat - \acmfX \|_{s}^{2} > \delta \right],
\end{align*}
where $Q_{2}(\acmfZ) := Q_{2}(\widetilde{\delta},\acmfZ) = \max\{2 R(0, \acmfZ), T(0,  \widetilde{\delta},\acmfZ) \}$ with $R, T$ as in \eqref{eq:quantitiesmainresult}.
\end{lemma}

In the upcoming proof sections, we refer to the statements in Section \ref{se:concentration} as opposed to those in the current Section \ref{se:constants} but we certainly use the constants in \eqref{eq:quantitiesmainresult}.

\subsection{Proofs} \label{se:proof}
\begin{proof}[Proof of Proposition \ref{prop:cons}/\ref{prop:consAPP}]
By treating the diagonal and off-diagonal elements of $\acmfZ$ separately, we get
\begin{align}
&
\Prob\left[ \| \acmfZhat - \acmfZ \|_{s} > Q(\acmfZ) \delta \right]
\nonumber
\\&=
\Prob\left[ \| \widehat{g}(\acmfXhat) - g(\acmfX) \|_{s} > Q(\acmfZ) \delta \right] 
\nonumber
\\&\leq 
\Prob\left[ \| \widehat{g}_{\bullet}(\acmfXhat) - g_{\bullet}(\acmfX) \|_{s} > Q(\acmfZ) \frac{\delta}{2} \right]
\notag
\\&\hspace{1cm}+
\Prob\left[ \max_{i=1,\dots,\Dim}| \widehat{ g }_{ii}(\widehat{\Gamma}_{X,ii}(0)) - g_{ii}(\Gamma_{X,ii}(0))| > D(\acmfZ) \delta \right]
\label{al:ststststststs1}
\\&\precsim
\Prob[ \| \acmfXhat - \acmfX \|_{s} > \delta \wedge \widetilde{\delta} ]
+ 
\Prob[ \| \acmfXhat - \acmfX \|_{s}^{2} > \delta ] 
\nonumber
\\&\hspace{1cm}+
\Prob [  \| \widehat{\theta} - \theta \|_{\maxF} > \delta \wedge \widetilde{\varepsilon} \wedge \varepsilon ]
+
\Prob[
\| \widehat{\theta} - \theta \|^2_{\maxF} > \delta ].
\label{al:ststststststs2}
\end{align}
To obtain \eqref{al:ststststststs1}, we used the triangle inequality and the fact that for the diagonal elements of any matrix $A \in \RR^{\Dim \times \Dim}$, $\| I_{\Lag} \otimes \diag(a_{11},\dots, a_{\Dim\Dim}) \|_{s} \leq \max_{i =1, \dots, \Dim} |a_{ii}|$. 
To obtain \eqref{al:ststststststs2}, the second summand is bounded by Lemma \ref{le:diagonalelements}. The rest of the proof bounds the first summand of \eqref{al:ststststststs1}. We have
\begin{align}
&
\Prob\left[ \| \widehat{g}_{\bullet}(\acmfXhat) - g_{\bullet}(\acmfX) \|_{s} > Q(\acmfZ) \frac{\delta}{2} \right] 
\nonumber
\\&\leq 
\Prob\left[ \| \widehat{g}_{\bullet}(\acmfXhat) - g_{\bullet}(\acmfXhat) \|_{s} > Q(\acmfZ) \frac{\delta}{4} \right]
\nonumber
\\&\hspace{1cm}+ 
\Prob\left[ \| g_{\bullet}(\acmfXhat) - g_{\bullet}(\acmfX) \|_{s} > Q(\acmfZ) \frac{\delta}{4} \right] 
\nonumber
\\&\leq 
\Prob\left[ \| \widehat{g}_{\bullet}(\acmfXhat) - g_{\bullet}(\acmfXhat) \|_{s} > Q_{1}(\acmfZ) \delta \right]
\nonumber
\\&\hspace{1cm}+ 
\Prob\left[ \| g_{\bullet}(\acmfXhat) - g_{\bullet}(\acmfX) \|_{s} > Q_{2}(\acmfZ) \delta \right] 
\label{al:twosummandssep}
\\& \precsim
\Prob[ \| \acmfXhat - \acmfX \|_{s} > \delta \wedge \widetilde{\delta} ] 
+
\Prob[ \| \acmfXhat - \acmfX \|_{s}^{2} > \delta  ] \notag
\\&\hspace{1cm}+
\Prob [  \| \widehat{\theta} - \theta \|_{\maxF} > \delta \wedge \varepsilon \wedge \widetilde{\varepsilon} ]
+
\Prob [  \| \widehat{\theta} - \theta \|^2_{\maxF} > \delta ],
\nonumber
\end{align}
where \eqref{al:twosummandssep} follows from
\begin{equation*} 
\begin{aligned}
&
Q(\acmfZ) 
\\&=
4\max\{ D(\acmfZ),
4R(\acmfZ), 2U(\acmfZ), T(\acmfZ) \} \max\{ S^2(\acmfZ), 1\}
\\& \geq
4\max\{ 
4R(\varepsilon, \acmfZ), 2U(\varepsilon, \widetilde{\varepsilon}, \acmfZ), T(\varepsilon, \widetilde{\delta}, \acmfZ) \} \max\{ S^2(\acmfZ), 1\}
\\& =
4\max\{ 
4R(\varepsilon, \acmfZ), 2U(\varepsilon, \widetilde{\varepsilon}, \acmfZ), T(\varepsilon, \widetilde{\delta}, \acmfZ), 2 R(0,\acmfZ), T(0, \widetilde{\delta}, \acmfZ) \} 
\\&\hspace{2cm} \times \max\{ S^2(\acmfZ), 1\}
\\& =
4 \max\{ Q_{1}(\acmfZ), Q_{2}(\acmfZ)\}
\end{aligned}
\end{equation*}
with $Q(\acmfZ) \geq 4 \max\{ Q_{1}(\acmfZ), Q_{2}(\acmfZ)\}$ and $Q_{1}(\acmfZ), Q_{2}(\acmfZ)$ defined in Lemmas \ref{le:main1APP} and \ref{le:main2APP}. The two summands in \eqref{al:twosummandssep} can then be bounded by the results in Lemmas \ref{le:main1APP} and \ref{le:main2APP}, respectively.
\end{proof}

\begin{proof}[Proof of Corollary \ref{cor:CI}]
The corollary is a consequence of Proposition \ref{prop:cons}. Indeed, set $\nu = \delta \wedge \widetilde{\delta} \wedge \varepsilon \wedge \widetilde{\varepsilon}$. Then, 
\begin{align}
&
\Prob\left[ \| \acmfZhat - \acmfZ \|_{s} > Q(\acmfZ) c_{0}(s) \delta \right]
\nonumber
\\&\precsim
\Prob[ \| \acmfXhat - \acmfX \|_{s} > c_{0}(s) \nu ] 
+
\Prob[ \| \acmfXhat - \acmfX \|_{s}^{2} > c_{0}(s) \nu ]
\nonumber
\\&\hspace{1cm}+
\Prob [  \| \widehat{\theta} - \theta \|_{\maxF} > \nu ]
+
\Prob [  \| \widehat{\theta} - \theta \|^2_{\maxF} > \nu ]
\label{al:proofcoroal1}
\\&\leq
c_{1,1} \exp\Big( -c_{1,2} N \min\left\{ 1, \nu^2, \sqrt{ \nu / c_{0}(s) }, \nu / c_{0}(s) \right\} 
\nonumber
\\&\hspace{1cm}+
2s \min\{ \log(\Lag \Dim), \log(21 e \ \Lag \Dim /2s) \} \Big)
\nonumber
\\&\hspace{2cm}+
c_{2,1} \Dim K \exp\left(-c_{2,2} T \min\{\nu, \nu^{2}\} \right)
\label{al:proofcoroal2}
\\&\leq
c_{1,1} \exp\left( -c_{1,2}N \min\left\{ 1, \nu^2, \sqrt{ \nu / c_{0}(s) } \right\}+ 2s \log(\Lag \Dim) \right)
\nonumber
\\&\hspace{1cm}+
c_{2,1} \Dim K \exp\left(-c_{2,2} T \min\{1, \nu^{2}\} \right),
\label{al:proofcoroal3}
\end{align}
where \eqref{al:proofcoroal1} follows from Proposition \ref{prop:cons} and since $c_{0}(s) \geq 1$ in Assumption \ref{ass:CI-VAR}, \eqref{al:proofcoroal2} is discussed in more detail below, and \eqref{al:proofcoroal3} is due to
\begin{equation*}
\min\left\{ \nu, \nu^2, \sqrt{\frac{\nu}{c_{0}(s)}}, \frac{\nu}{c_{0}(s)} \right\}
\geq
\min\left\{ 1, \nu^2, \frac{\nu}{c_{0}(s)} \right\}.
\end{equation*}
Turning back to relation \eqref{al:proofcoroal2}, the last two probabilities in \eqref{al:proofcoroal1} are bounded by Assumption \ref{ass:CItheta}.
The first two probabilities in \eqref{al:proofcoroal1} can be bounded from the following observation. By Assumption \ref{ass:CI-VAR} and Lemma F.2 in the supplementary material of \cite{basu2015regularized},
\begin{align*}
\Prob\left[ \| \acmfXhat - \acmfX \|_{s} > c_{0}(s) \nu \right]
&\leq 
c_{1}\exp\big( -c_{2}N \min\{ \nu, \nu^2 \} 
\\ &\hspace{2cm}+
2s \min\{ \log(\Lag \Dim), \log(21 e \ \Lag \Dim /2s) \} \big),
\\
\Prob\left[ \| \acmfXhat - \acmfX \|_{s}^2 > c_{0}(s) \nu \right]
&\leq 
c_{1}\exp\Bigg( -c_{2}N \min\Bigg\{ \sqrt{\frac{\nu}{c_{0}(s)}}, \frac{\nu}{c_{0}(s)} \Bigg\} 
\\ &\hspace{2cm}+ 
2s \min\{ \log(\Lag \Dim), \log(21 e \ \Lag \Dim /2s) \} \Bigg).
\end{align*}
Lemma F.2 in the supplementary material of \cite{basu2015regularized} requires Assumption \ref{ass:CI-VAR} to be true for any vector $v$ with $\| v \| \leq 1$. A close look into the proof reveals that it is sufficient to have a result for any vector $v \in \mathcal{K}(2s)$. The proof uses a discretization argument to approximate the set $\mathcal{K}(2s)$. More precisely, the authors construct an $\varepsilon$-net to approximate $\mathcal{K}(2s)$ following Definition 3.1 in \cite{Vershynin2009} and the proofs therein utilizing the concept of $\varepsilon$-nets. Following Lemma 3.5 in \cite{Vershynin2009}, there is an $\varepsilon$-net which is a subset of $\mathcal{K}(2s)$. The proof of Lemma F.2 needs Assumption \ref{ass:CI-VAR} to be satisfied for any vector in the $\varepsilon$-net and therefore elements of $\mathcal{K}(2s)$.
\end{proof}

\begin{proof}[Proof of Lemma \ref{le:main1APP}]
For shortness, we set $\delta^{*} = \max\{ S^2(\acmfZ), 1\} \delta$ and $\varepsilon^{*} = \max\{ S^2(\acmfZ), 1\} \widetilde{\varepsilon}$ for some $\widetilde{\varepsilon} >0$. With explanations given below, we have
\begin{align}
&
\Prob\left[ \| \widehat{g}_{\bullet}(\acmfXhat) - g_{\bullet}(\acmfXhat) \|_{s} > Q_{1}(\acmfZ) \delta \right]
\nonumber
\\&\leq
\Prob\left[ \norm{\widehat{g}_{\bullet}(\acmfX) - g_{\bullet}(\acmfX) } > Q_{1}(\acmfZ) \frac{\delta}{4} \right]
\nonumber
\\& \hspace{1cm}+ 
\Prob\left[ \norm{(\widehat{g}'_{\bullet}(\acmfX) - g'_{\bullet}(\acmfX)) \odot ( \acmfXhat - \acmfX )} > Q_{1}(\acmfZ) \frac{\delta}{4} \right]
\nonumber
\\& \hspace{2cm}+ 
\Prob\left[ \frac{1}{2} \norm{g''_{\bullet}(\Sigma) \odot ( \acmfXhat - \acmfX ) ^{\odot 2} } > Q_{1}(\acmfZ) \frac{\delta}{4} \right]
\nonumber
\\& \hspace{3cm}+ 
\Prob\left[ \frac{1}{2} \norm{\widehat{g}''_{\bullet}(\Sigma) \odot ( \acmfXhat - \acmfX ) ^{\odot 2} } > Q_{1}(\acmfZ) \frac{\delta}{4} \right]
\label{al:le:main1:eg1}
\\&\leq
\Prob\left[ \norm{\widehat{g}_{\bullet}(\acmfX) - g_{\bullet}(\acmfX) } > q_{1}(\varepsilon, \varepsilon^{*}, \acmfZ) \delta^{*} \right]
\nonumber
\\& \hspace{1cm}+ 
\Prob\left[ \norm{(\widehat{g}'_{\bullet}(\acmfX) - g'_{\bullet}(\acmfX)) \odot ( \acmfXhat - \acmfX )} > q_{2}(\varepsilon, \varepsilon^{*}, \acmfZ) \delta^{*} \right]
\nonumber
\\& \hspace{2cm}+ 
\Prob\left[ \norm{g''_{\bullet}(\Sigma) \odot ( \acmfXhat - \acmfX ) ^{\odot 2} } > 
q_{3}(\widetilde{\delta}, \acmfZ) \delta \right]
\nonumber
\\& \hspace{3cm}+ 
\Prob\left[ \norm{\widehat{g}''_{\bullet}(\Sigma) \odot ( \acmfXhat - \acmfX ) ^{\odot 2} } > 
q_{4}(\varepsilon, \widetilde{\delta}, \acmfZ) \delta \right]
\label{al:le:main1:eg1.1}
\\&\precsim
\Prob[ \| \widehat{\ell}_{\bullet}(\acmfZ) - \ell_{\bullet}( \acmfZ) \|_{s} > \delta^{*} \wedge \varepsilon^{*} ]
+
\Prob[ \| \widehat{\ell}_{\bullet}(\acmfZ) - \ell_{\bullet}(\acmfZ) \|_{s}^{2} > \delta^{*} ] \nonumber
\\&\hspace{1cm}+
\Prob[ \| \acmfXhat - \acmfX \|_{s} > \delta \wedge \widetilde{\delta} ]
+ 
\Prob[ \| \acmfXhat - \acmfX \|_{s}^{2} > \delta ] \notag
\\&\hspace{1cm}+
\Prob[ \| \widehat{\ell}_{\bullet}(\acmfZ) - \ell_{\bullet}(\acmfZ) \|_{s} > \varepsilon^{*} ]
+
\Prob[ \| \widehat{\ell}_{\bullet}(\acmfZ) - \ell_{\bullet}(\acmfZ) \|_{s}^{2} > \delta^{*} ]
\notag
\\&\hspace{2cm}+
\Prob [  \| \widehat{\theta} - \theta \|_{\maxF} > \varepsilon ]
\nonumber
\\& \hspace{1cm}+ 
\Prob\left[ \norm{g''_{\bullet}(\Sigma) \odot ( \acmfXhat - \acmfX ) ^{\odot 2} } > q_{3}(\widetilde{\delta}, \acmfZ) \delta \right]
\notag
\\&\hspace{2cm}+
\Prob\left[ \norm{\widehat{g}''_{\bullet}(\Sigma) \odot ( \acmfXhat - \acmfX ) ^{\odot 2} } > q_{4}(\varepsilon, \widetilde{\delta}, \acmfZ) \delta \right]
\label{al:le:main1:eg1.2}
\\&\precsim
\Prob[ \| \widehat{\ell}_{\bullet}(\acmfZ) - \ell_{\bullet}( \acmfZ) \|_{s} > \delta^{*} \wedge \varepsilon^{*} ]
+
\Prob[ \| \widehat{\ell}_{\bullet}(\acmfZ) - \ell_{\bullet}(\acmfZ) \|_{s}^{2} > \delta^{*} ] \nonumber
\\&\hspace{1cm}+
\Prob[ \| \acmfXhat - \acmfX \|_{s} > \delta \wedge \widetilde{\delta} ]
+ 
\Prob[ \| \acmfXhat - \acmfX \|_{s}^{2} > \delta ] 
\notag
\\&\hspace{1cm}+
\Prob[ \| \widehat{\ell}_{\bullet}(\acmfZ) - \ell_{\bullet}(\acmfZ) \|_{s} > \varepsilon^{*} ]
+
\Prob[ \| \widehat{\ell}_{\bullet}(\acmfZ) - \ell_{\bullet}(\acmfZ) \|_{s}^{2} > \delta^{*} ]
\notag
\\&\hspace{2cm}+
\Prob [  \| \widehat{\theta} - \theta \|_{\maxF} > \varepsilon ]
\label{al:le:main1:eg2}
\\&\precsim
\Prob[ \| \widehat{\ell}_{\bullet}(\acmfZ) - \ell_{\bullet}( \acmfZ) \|_{s} > \delta^{*} \wedge \varepsilon^{*} ]
+
\Prob[ \| \widehat{\ell}_{\bullet}(\acmfZ) - \ell_{\bullet}(\acmfZ) \|_{s}^{2} > \delta^{*} ] 
\nonumber
\\&\hspace{1cm}+
\Prob[ \| \acmfXhat - \acmfX \|_{s} > \delta \wedge \widetilde{\delta} ]
+ 
\Prob[ \| \acmfXhat - \acmfX \|_{s}^{2} > \delta ]
\notag
\\&\hspace{2cm}+
\Prob [  \| \widehat{\theta} - \theta \|_{\maxF} > \varepsilon ].
\label{al:le:main1:eg3}
\end{align}
The bound \eqref{al:le:main1:eg1} follows by applying the second-order Taylor expansion to the function $x \mapsto \widehat{g}(x) - g(x)$ around the true covariance matrix, where $\Sigma$ is such that $\left| \Sigma - \acmfX \right| < \big| \Sigma - \acmfXhat \big|$ (entry-wise); see also Remark \ref{re:differentiablityanddeltatilde}.
To bound the probabilities in \eqref{al:le:main1:eg1} further, we aim to use Lemmas \ref{le:ghatSigmagSigma}, \ref{le:ghatSigmagprimegSigmaprime}, \ref{le:detgprimeprimeSigma} and \ref{le:estgprimeprimeSigma}. The four lemmas respectively introduce the constants $q_{1}, q_{2}, q_{3}$ and $q_{4}$. In the following, we argue why $Q_{1}(\acmfZ) $ is always larger than either one of them. This will give us \eqref{al:le:main1:eg1.1}. Indeed, note that
\begin{align}
&
Q_{1}(\acmfZ) 
\nonumber
\\&=
4\max\{ 4R(\acmfZ), 2U(\acmfZ), T(\acmfZ) \} \max\{ S^2(\acmfZ), 1 \} \nonumber
\\&=
4\max\{ 
4R(\varepsilon, \acmfZ), 2U(\varepsilon, \widetilde{\varepsilon}, \acmfZ), T(\varepsilon, \widetilde{\delta}, \acmfZ) \} \max\{ S^2(\acmfZ), 1\}
\nonumber
\\&=
4\max\{ 
4R(\varepsilon, \acmfZ), 2T(\varepsilon, \max\{ S^2(\acmfZ), 1\} \widetilde{\varepsilon}, \acmfZ), T(\varepsilon, \widetilde{\delta}, \acmfZ) \} \max\{ S^2(\acmfZ), 1\}
\nonumber
\\&=
4\max\{ 
4R(\varepsilon, \acmfZ), 2T(\varepsilon, \varepsilon^{*}, \acmfZ), T(\varepsilon, \widetilde{\delta}, \acmfZ) \} \max\{ S^2(\acmfZ), 1\}
\nonumber
\\&= 
4\max\{ 
q_{2}(\varepsilon, \varepsilon^{*}, \acmfZ), T(\varepsilon, \widetilde{\delta}, \acmfZ) \} \max\{ S^2(\acmfZ), 1\}
\label{al:yyy2}
\\&=
4\max\{ 
q_{1}(\varepsilon, \varepsilon^{*}, \acmfZ), q_{2}(\varepsilon, \varepsilon^{*}, \acmfZ), T(0, \widetilde{\delta}, \acmfZ), T(\varepsilon, \widetilde{\delta}, \acmfZ) \} \max\{ S^2(\acmfZ), 1\}
\label{al:yyy3}
\\&=
4\max\{ 
q_{1}(\varepsilon, \varepsilon^{*}, \acmfZ), q_{2}(\varepsilon, \varepsilon^{*}, \acmfZ), q_{3}(\widetilde{\delta}, \acmfZ), q_{4}(\varepsilon, \widetilde{\delta}, \acmfZ) \} \max\{ S^2(\acmfZ), 1\},
\label{al:yyy4}
\end{align}
where \eqref{al:yyy2} follows since $q_{2}(\varepsilon, \varepsilon^{*}, \acmfZ) := \max\{4 R(\varepsilon, \acmfZ), 2 T(\varepsilon, \varepsilon^{*}, \acmfZ)\}$ as defined in Lemma \ref{le:ghatSigmagprimegSigmaprime}. The inequality \eqref{al:yyy3} is due to the relation $q_{2} = 2q_{1}$ and $T(0, \widetilde{\delta}, \acmfZ) \leq T(\varepsilon, \widetilde{\delta}, \acmfZ)$. Finally, \eqref{al:yyy4} follows since $q_{3}(\widetilde{\delta}, \acmfZ) = T(0, \widetilde{\delta}, \acmfZ)$ 
and $q_{4}(\varepsilon, \widetilde{\delta}, \acmfZ) = T(\varepsilon, \widetilde{\delta}, \acmfZ)$ as stated in Lemmas \ref{le:detgprimeprimeSigma} and \ref{le:estgprimeprimeSigma}.
Then, \eqref{al:le:main1:eg1.2} can be inferred by applying Lemmas \ref{le:ghatSigmagSigma} and \ref{le:ghatSigmagprimegSigmaprime} with $\delta = \delta^{*}$ and $\widetilde{\delta} = \varepsilon^{*}$. The inequality \eqref{al:le:main1:eg2} follows by Lemmas \ref{le:detgprimeprimeSigma} and \ref{le:estgprimeprimeSigma}.
Finally, note that by Lemma \ref{le:ellhatSigmaellSigma},
\begin{align*}
\Prob\left[
\| \widehat{\ell}_{\bullet}(\acmfZ) - \ell_{\bullet}(\acmfZ) \|_{s} > S(\acmfZ) \delta \right]
&\leq
\Prob[
\| \widehat{\theta} - \theta \|_{\maxF} > \delta \wedge \varepsilon ],
\\
\Prob\left[
\| \widehat{\ell}_{\bullet}(\acmfZ) - \ell_{\bullet}(\acmfZ) \|^2_{s} > S^2(\acmfZ) \delta \right]
&\leq
\Prob[
\| \widehat{\theta} - \theta \|^2_{\maxF} > \delta ]
+
\Prob[
\| \widehat{\theta} - \theta \|_{\maxF} > \varepsilon ].
\end{align*}
Hence, using \eqref{al:le:main1:eg3}, we can infer
\begin{align}
&
\Prob\left[ \| \widehat{g}_{\bullet}(\acmfXhat) - g_{\bullet}(\acmfXhat) \|_{s} > Q_{1}(\acmfZ) \delta \right]
\nonumber
\\&\precsim
\Prob\left[ \| \widehat{\ell}_{\bullet}(\acmfZ) - \ell_{\bullet}( \acmfZ) \|_{s} > S(\acmfZ) (\delta \wedge \widetilde{\varepsilon}) \right]
+
\Prob[ \| \widehat{\ell}_{\bullet}(\acmfZ) - \ell_{\bullet}(\acmfZ) \|_{s}^{2} > S^2(\acmfZ) \delta ] \nonumber
\\&\hspace{1cm}+
\Prob[ \| \acmfXhat - \acmfX \|_{s} > \delta \wedge \widetilde{\delta} ]
+ 
\Prob[ \| \acmfXhat - \acmfX \|_{s}^{2} > \delta ] 
+
\Prob [  \| \widehat{\theta} - \theta \|_{\maxF} > \varepsilon ]
\nonumber
\\&\precsim
\Prob[ \| \acmfXhat - \acmfX \|_{s} > \delta \wedge \widetilde{\delta} ]
+ 
\Prob[ \| \acmfXhat - \acmfX \|_{s}^{2} > \delta ] 
\notag
\\&\hspace{1cm}+
\Prob [  \| \widehat{\theta} - \theta \|_{\maxF} > \delta \wedge \widetilde{\varepsilon} \wedge \varepsilon ]
+
\Prob[
\| \widehat{\theta} - \theta \|^2_{\maxF} > \delta ].
\nonumber
\end{align}
\end{proof}

\begin{remark} \label{re:differentiablityanddeltatilde}
In light of our extensive use of second-order Taylor approximations applied to the link function $\ell$ and its inverse $g$, we pause here to discuss some differentiability issues.
Note first that as a power series with absolutely summable coefficients, the function $\ell_{ij}(u)$ is differentiable infinitely many times for $u \in (-1,1)$. An expression for its first derivative is given in Proposition \ref{prop:prop2.1}. The inverse function $g_{ij}$ of $\ell_{ij}$ is defined on $(\ell_{ij}(-1),\ell_{ij}(1))$ and is differentiable infinitely many times on this interval, since the same holds for $\ell_{ij}$ on $(-1,1)$.
As an example for differentiability requirements in the proof of our results, we discuss \eqref{al:le:main1:eg1}, where we applied a second-order Taylor approximation as
\begin{equation*}
g_{\bullet}(\acmfXhat) =  g_{\bullet}(\acmfX) + g'_{\bullet}(\acmfX) \odot (\acmfXhat - \acmfX ) + \frac{1}{2} g''_{\bullet}(\Sigma) \odot (\acmfXhat - \acmfX )^{\odot 2 }
\end{equation*}
for some $\Sigma$ such that $\left| \Sigma - \acmfX \right| < \big| \Sigma - \acmfXhat \big|$. Strictly speaking, this requires twice differentiability of $g$ on $\left| \Sigma - \acmfX \right| < \big| \Sigma - \acmfXhat \big|$. However, $g_{ij}$ is twice differentiable only on $(\ell_{ij}(-1),\ell_{ij}(1))$ but the estimator $\acmfXhat$ can certainly take values outside of this interval. We assume implicitly for now that $\acmfXhat$ is close enough to the true $\acmfX$ to ensure differentiability. In Lemmas \ref{le:detgprimeprimeSigma} and \ref{le:estgprimeprimeSigma}, we will address this issue by intersecting with the event $\{ \| \acmfXhat - \acmfX \|_{s} \leq \widetilde{\delta} \}$, whenever we aim to bound the second-order terms.
\end{remark}

\begin{proof}[Proof of Lemma \ref{le:main2APP}]
By the second-order Taylor approximation of each component of $g$ around $\acmfX$ and subsequent application of the triangle inequality, we get, for some $\Sigma$ such that $\left| \Sigma - \acmfX \right| < \big| \Sigma - \acmfXhat \big|$,
\begin{align}
&
\norm{ g_{\bullet}(\acmfXhat) - g_{\bullet}(\acmfX)}
\nonumber
\\&\leq
\norm{ g^{\prime}_{\bullet}(\acmfX) \odot (\acmfXhat - \acmfX)}
+
\frac{1}{2}
\norm{ g^{\prime \prime}_{\bullet}(\Sigma) \odot (\acmfXhat - \acmfX) ^{\odot 2} } \nonumber
\\&=
\norm{ (\ell^{\prime})_{\bullet}^{\odot (-1)}(\acmfZ) \odot (\acmfXhat - \acmfX)}
+
\frac{1}{2}
\norm{ g^{\prime \prime}_{\bullet}(\Sigma) \odot (\acmfXhat - \acmfX) ^{\odot 2} } 
\label{al:hh0}
\\&\leq
R(0,\acmfZ) \norm{\acmfXhat - \acmfX}
+
\frac{1}{2}
\norm{ g^{\prime \prime}_{\bullet}(\Sigma) \odot (\acmfXhat - \acmfX) ^{\odot 2} }. \label{al:hh1}
\end{align}
The equality \eqref{al:hh0} follows since $g^{\prime}(\acmfX) = (\ell^{-1})^{\prime}( \ell(\acmfZ)) = (\ell^{\prime})^{\odot (-1)}(\acmfZ)$; see Proposition \ref{prop:prop2.1} for an explicit representation of $\ell^{\prime}$.
Furthermore, relation \eqref{al:hh1} follows by Lemma \ref{le:derivofg}.
While the first term in \eqref{al:hh1} is already what appears in the bound of Lemma \ref{le:main2APP}, the second term needs to be considered further. We thus have
\begin{align}
&
\Prob\left[ 
\norm{ g_{\bullet}(\acmfXhat) - g_{\bullet}(\acmfX)} > Q_{2}(\acmfZ) \delta \right]
\nonumber
\\&\leq
\Prob\left[ R(0,\acmfZ) \| \acmfXhat - \acmfX \|_{s} > Q_{2}(\acmfZ) \frac{\delta}{2} \right]
\nonumber
\\&\hspace{1cm}+
\Prob\left[ \norm{ g^{\prime \prime}_{\bullet}(\Sigma) \odot (\acmfXhat - \acmfX)^{ \odot 2} } > Q_{2}(\acmfZ) \delta \right]
\nonumber
\\&\leq
\Prob\left[ \| \acmfXhat - \acmfX \|_{s} > \delta \right]
+
\Prob\left[ \norm{ g^{\prime \prime}_{\bullet}(\Sigma) \odot (\acmfXhat - \acmfX) }^{2} > 
q_{3}(\widetilde{\delta}, \acmfZ) \delta \right]
\label{al:hh3}
\\&\leq
\Prob\left[ \| \acmfXhat - \acmfX \|_{s} > \delta \right]
+
\Prob\left[ \| \acmfXhat - \acmfX \|_{s}^{2} > \delta \right]
+
\Prob\left[ \| \acmfXhat - \acmfX \|_{s} > \widetilde{\delta} \right]
\label{al:hh4}
\\&\precsim
\Prob\left[ \| \acmfXhat - \acmfX \|_{s} > \delta \wedge \widetilde{\delta} \right]
+
\Prob\left[ \| \acmfXhat - \acmfX \|_{s}^{2} > \delta \right],
\nonumber
\end{align}
where \eqref{al:hh3} follows since 
$$
Q_{2}(\acmfZ) 
= 
\max\{2 R(0,\acmfZ), T(0, \widetilde{\delta}, \acmfZ) \}
=
\max\{2 R(0,\acmfZ), q_{3}(\widetilde{\delta}, \acmfZ) \}
$$
and \eqref{al:hh4} is a consequence of applying Lemma \ref{le:detgprimeprimeSigma}.
\end{proof}

\section{Case of latent VAR processes}

\subsection{Proofs of results in Section \ref{se:AppVAR}}
In this section, we provide all proofs concerning the transition matrix estimation of the latent VAR($\VO$) process.

The following proof of Proposition \ref{prop:consistentbeta} is exactly the same as the proof of Proposition 4.1 in \cite{basu2015regularized} and is only included for completeness. The actual contributions below are the proofs of Lemmas \ref{le:RE} and \ref{le:DB}. Those lemmas show that the restricted eigenvalue condition and the deviation bound can be verified for the latent process based on the observed count series.
\begin{proof}[Proof of Proposition \ref{prop:consistentbeta}]
Recall $\widehat{\Gamma} = I_{\Dim} \otimes \acmfZhat$ from \eqref{eq:gammahatGammahat}.
Since $\widehat{\beta}$ minimizes the objective function, we get that, for all $\beta$,
\begin{align*}
-2\widehat{\beta}' \widehat{\gamma} + \widehat{\beta}' \widehat{\Gamma} \widehat{\beta}
+ \lambda_{N} \| \widehat{\beta} \|_{1}
\leq
-2\beta_{0}' \widehat{\gamma} + \beta_{0}' \widehat{\Gamma} \beta_{0}
+ \lambda_{N} \| \beta_{0} \|_{1}.
\end{align*}
The above inequality reduces to
\begin{align*}
v' \widehat{\Gamma} v 
\leq
2v' (\widehat{\gamma} - \widehat{\Gamma} \beta_{0}) + \lambda_{N} ( \| \beta_{0} \|_{1} - \| \beta_{0} + v \|_{1} ),
\end{align*}
where $v = \widehat{\beta} - \beta_{0}$. 
By the restricted eigenvalue condition
\begin{equation} \label{al:lower}
v' \widehat{\Gamma} v \geq \alpha \| v \|^2 - \tau(N,q) \| v \|^2_{1} \geq (\alpha - 16 s \tau(N,q)) \|v\|^{2} \geq \frac{\alpha}{2} \|v \|^{2},
\end{equation}
where the second last inequality is due to $\| v \|_{1} \leq 4 \| v_{S} \|_{1} \leq 4\sqrt{s} \| v \|$ which follows by Cauchy-Schwarz inequality. The last inequality in \eqref{al:lower} follows by the assumed relationship between $\tau(N,q)$ and $\alpha$.

Set $S = \supp\{\beta_{0}\}$, such that $\beta_{0,j} = 0 $ for all $j \in S^{c}$, where $S^{c}$ denotes the complement of $S$. We write $\beta_{0,S}$ for the corresponding non-zero entries of $\beta_{0}$. 
By \eqref{al:lower}, $v' \widehat{\Gamma} v \geq 0$ for all $v \in \RR^{q}$. Then,
\begin{align}
0
\leq
v' \widehat{\Gamma} v 
&\leq
2v' (\widehat{\gamma} - \widehat{\Gamma} \beta_{0}) + \lambda_{N} ( \| \beta_{0} \|_{1} - \| \beta_{0} + v \|_{1} )
\nonumber
\\&\leq
2v' (\widehat{\gamma} - \widehat{\Gamma} \beta_{0}) + \lambda_{N} ( \| \beta_{0,S} \|_{1} - \| \beta_{0,S} + v_{S} \|_{1} - \| \beta_{0,S^{c}} + v_{S^{c}} \|_{1} )
\nonumber
\\&\leq
2 \| v \|_{1} \| \widehat{\gamma} - \widehat{\Gamma} \beta_{0} \|_{\max} + \lambda_{N} ( \| v_{S} \|_{1} - \| v_{S^{c}} \|_{1} )
\nonumber
\\&\leq
\frac{\lambda_{N}}{2} \| v \|_{1}+ \lambda_{N} ( \| v_{S} \|_{1} - \| v_{S^{c}} \|_{1} )
\nonumber
\\&\leq
\frac{3\lambda_{N}}{2} \| v_{S} \|_{1} - \frac{\lambda_{N}}{2} \| v_{S^{c}} \|_{1}
\label{al:upper4}
\end{align}
since $\| \widehat{\gamma} - \widehat{\Gamma} \beta_{0} \|_{\max} \leq \mathcal{Q}(\beta_{0}) \sqrt{\frac{\log(q)}{N}} $ by the deviation bound and since we suppose that $\lambda_{N} \geq 4 \mathcal{Q}(\beta_{0}) \sqrt{\frac{\log(q)}{N}}$.
Then, \eqref{al:upper4} ensures that $\| v_{S^{c}} \|_{1} \leq 3 \| v_{S} \|_{1}$ and hence $\| v \|_{1} \leq 4 \| v_{S} \|_{1} \leq 4\sqrt{s} \| v \|$.

Due to \eqref{al:lower}, the upper and lower bounds on \eqref{al:upper4} and \eqref{al:lower} lead
\begin{equation*}
\begin{gathered}
\| v \| \leq 16 \sqrt{s} \frac{\lambda_{N} }{\alpha},
\hspace{0,2cm}
\| v \|_{1} \leq 4 \sqrt{s} \lambda_{N} \| v\| \leq 64 s \frac{\lambda_{N} }{\alpha},
\hspace{0,2cm}
v' \widehat{\Gamma} v \leq 2 \lambda_{N} \| v\|_{1} \leq 128 s \frac{\lambda^2_{N} }{\alpha}.
\end{gathered}
\end{equation*}
\end{proof}

\begin{proof}[Proof of Lemma \ref{le:RE}]
(\textit{Restricted Eigenvalue condition})
By Lemma B.1 in the supplementary material of \cite{basu2015regularized}, we have
$$\widehat{\Gamma} \sim RE(\alpha,\tau)
\hspace{0.2cm} \text{ if } \hspace{0.2cm} 
\widehat{\bm{\Gamma}}_{Z} \sim RE(\alpha,\tau).$$
For this reason, it is sufficient to prove the restricted eigenvalue condition for $\widehat{\bm{\Gamma}}_{Z}$.

Under Assumptions \ref{ass:CI-VAR} and \ref{ass:CItheta}, we apply Corollary \ref{cor:CI} with $\delta = \widetilde{\delta}  = \varepsilon = \widetilde{\varepsilon} = \nu$ so that
\begin{align*}
&
\Prob\left[ \| \acmfZhat - \acmfZ \|_{s} > Q(\acmfZ) c_{0}(s) \nu \right]
\nonumber
\\&\leq
c_{1,1} \exp\left( -c_{1,2}N \min\{ 1, \nu^2 \} + 2s \log(\Dim \VO) \right)
+
c_{2,1} \Dim K \exp\left(-c_{2,2} T \min\{ 1, \nu^2 \} \right)
\\&\leq
c_{1} \exp\left( -c_{1,2}N \min\{ 1, \nu^2 \} + 2s \log(\Dim \VO) -c_{2,2} T \min\{ 1, \nu^2 \} + \log( \Dim K) \right)
\\&\leq
c_{1} \exp\left( -c_{2}N \min\{ 1, \nu^2 \} \right)
\end{align*}
since $N = T - \VO \leq T$ and $N \succsim \max\{ c_{0}(s)/\nu, \nu^{-2}, 1 \} \max\{ s\log(\Dim \VO), \log( \Dim K) \}$. Then, 
\begin{align} \label{al:inewithhighprobRE}
\left| v' (\acmfZhat - \acmfZ) v \right| \leq Q(\acmfZ) c_{0}(s) \nu
=
\frac{\lambda_{\min}(\Sigma_{\varepsilon})}{54 \mu_{\max}(\mathcal{A})}
\hspace{0.2cm}
\text{ for all }
\hspace{0.2cm}
v \in \RR^{\Dim \VO}
\end{align}
with probability at least $1-c_{1} \exp\left( -c_{2}N \min\{ 1, \nu^2, \nu/c_{0}(s) \} \right)$ and choosing 
$$\nu = \frac{\lambda_{\min}(\Sigma_{\varepsilon})}{54 \mu_{\max}(\mathcal{A}) Q(\acmfZ) c_{0}(s)}.$$ 
From \eqref{al:inewithhighprobRE} and by applying Lemma 12 in the supplement of \cite{loh2012high} 
we infer
\begin{equation*}
v' \acmfZhat v
\geq
v' \acmfZ v
-
\frac{\lambda_{\min}(\Sigma_{\varepsilon})}{2 \mu_{\max}(\mathcal{A})} ( \| v \|^{2}_{2} + \frac{1}{s} \| v \|^2_{1}).
\end{equation*}
Then,
\begin{align}
v' \acmfZhat v
&\geq
\frac{\lambda_{\min}(\Sigma_{\varepsilon})}{2 \mu_{\max}(\mathcal{A})} \| v \|^{2}_{2}
-
\frac{\lambda_{\min}(\Sigma_{\varepsilon})}{2 \mu_{\max}(\mathcal{A})} \frac{1}{s} \| v \|^2_{1}
\label{al:ejcnjecnejncejk1}
\\&\geq
\alpha \| v \|^{2}_{2}
-
\alpha \frac{1}{N} \max\{\nu^{-2},1 \}  4 \log(\Dim \VO) \| v \|^2_{1}
=
\alpha \| v \|^{2}_{2}
-
\tau \| v \|^2_{1}.
\label{al:ejcnjecnejncejk2}
\end{align}
The bound \eqref{al:ejcnjecnejncejk1} follows since $v' \acmfZ v \geq \frac{\lambda_{\min}(\Sigma_{\varepsilon})}{\mu_{\max}(\mathcal{A})} \| v \|_{2}^{2}$ by Proposition 2.3 and relation (4.1) in \cite{basu2015regularized}.
With $s = \lceil N \min\{1, \nu^2, \nu/c_{0}(s) \} / 4 \log(\Dim \VO) \rceil$ and $\alpha$ and $\tau$ in \eqref{eq:alphataunu}, we infer \eqref{al:ejcnjecnejncejk2}.
\end{proof}

\begin{proof}[Proof of Lemma \ref{le:DB}]
\textit{(Deviation bound)}
Note that
\begin{align} \label{eq:q1}
\| \widehat{\gamma} - \widehat{\Gamma} \beta_{0} \|_{\max}
\leq
\| ( \widehat{\Gamma} - \Gamma ) \beta_{0} \|_{\max}
+
\| \widehat{\gamma} - \gamma \|_{\max},
\end{align}
since $\gamma - \Gamma \beta_{0} = 0$, where $\gamma$ and $\Gamma$ are the population quantities of $\widehat{\gamma}$ and $\widehat{\Gamma}$. Let $e_{q,i}$ denote the $i$th basis vector of $\RR^{q}$. Then, considering both summands in \eqref{eq:q1} separately, we get for the first summand,
\begin{align}
\| ( \widehat{\Gamma} - \Gamma ) \beta_{0} \|_{\max}
&=
\max_{i=1,\dots,q} | e_{q,i}'( I_{\Dim} \otimes \widehat{\bm{\Gamma}}_{Z} - \Gamma )\beta_{0} | \nonumber
\\&=
\max_{i=1,\dots,q} | e_{q,i}'( I_{\Dim} \otimes ( \widehat{\bm{\Gamma}}_{Z} - \bm{\Gamma}_{Z} ) ) \beta_{0} | \nonumber
\\&=
\max_{i=1,\dots,q} | e_{q,i}' \vecop( ( \widehat{\bm{\Gamma}}_{Z} - \bm{\Gamma}_{Z} ) B_{0}) | \label{al:DBsumand1-1}
\\&=
\max_{i=1,\dots,\VO \Dim; j = 1,\dots, \Dim} | e_{\VO \Dim,i}' ( \widehat{\bm{\Gamma}}_{Z} - \bm{\Gamma}_{Z} ) B_{0} e_{\Dim,j} |, \label{al:DBsumand1-2}
\end{align}
where $\beta_{0} = \vecop(B_{0})$ with $B_{0}$ in \eqref{eq:VARlinearform} and \eqref{al:DBsumand1-1} is due to Theorem 2, Section 4 in \cite{MagnusNeudecker}. 
Note that $\| B_{0} e_{\Dim,j} \| <1$ since $\| B_{0} e_{\Dim,j} \|  = \| B_{0} e_{\Dim,j} \| / \| e_{\Dim,j} \| \leq \sup_{v }  \| B_{0}v\| / \| v \| < 1$.
By using \eqref{al:DBsumand1-2} and with further explanations given below, for 
$\mu_{1} = 3 Q(\acmfZ) c_{0}(s) \delta$ with $v_{j} = B_{0} e_{\Dim,j}$, 
\begin{align}
&
\Prob\left[ \| ( \widehat{\Gamma} - \Gamma ) \beta_{0} \|_{\max} > \mu_{1} \right]
\nonumber
\\&=
\Prob\left[ \max_{i=1,\dots,\VO \Dim; j = 1,\dots, \Dim} | e_{\VO \Dim,i}' ( \widehat{\bm{\Gamma}}_{Z} - \bm{\Gamma}_{Z} ) B_{0} e_{\Dim,j} | > \mu_{1} \right] 
\nonumber
\\&\leq
\sum_{i,j=1}^{\VO \Dim} \Prob\left[ | e_{\VO \Dim, i}' ( \widehat{\bm{\Gamma}}_{Z} - \bm{\Gamma}_{Z} ) v_{j} | > \mu_{1} \right] 
\label{al:DBsumand1-2.3}
\\&\leq
\sum_{i,j=1}^{\VO \Dim} 
\Bigg( \Prob\left[ | e_{\VO \Dim, i}' ( \widehat{\bm{\Gamma}}_{Z} - \bm{\Gamma}_{Z} ) e_{\VO \Dim, i} | > \frac{2}{3}\mu_{1} \right] 
+
\Prob\left[ | v_j' ( \widehat{\bm{\Gamma}}_{Z} - \bm{\Gamma}_{Z} ) v_{j} | > \frac{2}{3}\mu_{1} \right] 
\nonumber
\\&\hspace{1cm}+
\Prob\left[ | (e_{\VO \Dim, i} + v_j)' ( \widehat{\bm{\Gamma}}_{Z} - \bm{\Gamma}_{Z} ) (e_{\VO \Dim, i} + v_j) | > \frac{2}{3}\mu_{1} \right] \Bigg)
\label{al:DBsumand1-2.4}
\\&\leq
(\VO \Dim)^2 3
\Prob\left[ \| \widehat{\bm{\Gamma}}_{Z} - \bm{\Gamma}_{Z} \|_s > \frac{1}{3}\mu_{1} \right]
\nonumber
\\&\leq
c_{1,1} \exp\left( -c_{1,2}N \min\left\{ 1, \nu^2, \nu / c_{0}(s) \right\} + 2\log(\VO \Dim) \right).
\label{al:DBsumand1-2.5}
\end{align}
We used a union bound to infer \eqref{al:DBsumand1-2.3}. 
The relation $2 |v'Aw| \leq |v'Av| + |w'Aw| + |(v+w)'A(v+w)|$ implies \eqref{al:DBsumand1-2.4}.
Under Assumptions \ref{ass:CI-VAR} and \ref{ass:CItheta}, Corollary \ref{cor:CI} and
\eqref{eq:consequenceCor3.1} with $N \succsim \max\{ c_{0}(s)\nu^{-1}, \nu^{-2}, 1 \} \max\{ s\log(\Lag \Dim), \log( \Dim K) \}$ and $N = T - \VO \leq T$ give \eqref{al:DBsumand1-2.5}.

For the second summand in \eqref{eq:q1}, recall from \eqref{eq:gammahatGammahat} that $\gamma = \vecop( \bm{\gamma}_{Z} )$ with $\bm{\gamma}_{Z} = (\Gamma_{Z}(1)',\dots,\Gamma_{Z}(\VO)')' $. We further introduce 
$\bm{\Gamma}^{p+1}_{Z} = ( \Gamma_{Z}(r-s) )_{r,s=1,\dots,(\VO+1)} $
and
$\widehat{\bm{\Gamma}}^{p+1}_{Z} = ( \widehat{\Gamma}_{Z}(r-s) )_{r,s=1,\dots,(\VO+1)} $.
\begin{align}
&
\| \widehat{\gamma} - \gamma \|_{\max} \notag
\\&=
\| \vecop( \widehat{\bm{\gamma}}_{Z}) - \vecop( \bm{\gamma}_{Z} ) \|_{\max} 
\nonumber
\\&=
\| (e'_{(\VO+1),1} \otimes I_{\Dim}) ( \widehat{\bm{\Gamma}}^{p+1}_{Z} - \bm{\Gamma}^{p+1}_{Z} ) ( (e_{(\VO+1) \Dim,2}, \dots, e_{(\VO+1),(\VO+1)}) \otimes I_{\Dim}) \|_{\max} 
\nonumber
\\&=
\max_{i=1,\dots,\Dim;j=1,\dots, \Dim \VO} | e'_{\Dim,i} (e'_{(\VO+1),1} \otimes I_{\Dim}) ( \widehat{\bm{\Gamma}}^{p+1}_{Z} 
- \bm{\Gamma}^{p+1}_{Z} ) 
\nonumber
\\ &\hspace{3cm}\times 
( (e_{(\VO+1) \Dim,2}, \dots, e_{(\VO+1),(\VO+1)}) \otimes I_{\Dim}) e_{\Dim \VO,j}  |
\nonumber
\\&=
\max_{i=1,\dots,\Dim;j=1,\dots, \Dim \VO} | w'_{1,i} ( \widehat{\bm{\Gamma}}^{p+1}_{Z}
- \bm{\Gamma}^{p+1}_{Z} ) w_{2,j}  | ,
\label{al:DBsumand2-4}
\end{align}
where $w_{1,i}' = e'_{\Dim,i} (e'_{(\VO+1),1} \otimes I_{\Dim})$ and $w_{2,j} = ( (e_{(\VO+1) \Dim,2}, \dots, e_{(\VO+1),(\VO+1)}) \otimes I_{\Dim}) e_{\Dim \VO,j}$.
In particular, using \eqref{al:DBsumand2-4} and with further explanations given below, for $\mu_{2} = 3 Q(\acmfZ) c_{0}(s)\delta$,
\begin{align}
&
\Prob\left[ \| \widehat{\gamma} - \gamma \|_{\max} > \mu_{2} \right] 
\notag
\\&=
\Prob\left[ \max_{i=1,\dots,\Dim;j=1,\dots, \Dim \VO} | w'_{1,i} ( \widehat{\bm{\Gamma}}^{p+1}_{Z}  - \bm{\Gamma}^{p+1}_{Z} ) w_{2,j}| > \mu_{2} \right] 
\nonumber
\\&\leq
\sum_{i,j=1}^{\Dim \VO}
\Prob\left[ | w'_{1,i} ( \widehat{\bm{\Gamma}}^{p+1}_{Z}  - \bm{\Gamma}^{p+1}_{Z} ) w_{2,j}| > \mu_{2} \right] 
\label{al:DBsumand2-2.3}
\\&\leq
\sum_{i,j=1}^{\Dim \VO}
\Bigg(
\Prob\left[ | w'_{1,i} ( \widehat{\bm{\Gamma}}^{p+1}_{Z}  - \bm{\Gamma}^{p+1}_{Z} ) w_{1,i}| > \frac{2}{3}\mu_{2} \right] 
\nonumber\\&\hspace{1cm}+
\Prob\left[ | w'_{2,j} ( \widehat{\bm{\Gamma}}^{p+1}_{Z}  - \bm{\Gamma}^{p+1}_{Z} ) w_{2,j}| > \frac{2}{3}\mu_{2} \right] 
\nonumber\\&\hspace{1cm}+
\Prob\left[ | (w_{1,i} + w_{2,j})' ( \widehat{\bm{\Gamma}}^{p+1}_{Z}  - \bm{\Gamma}^{p+1}_{Z} ) (w_{1,i} + w_{2,j}) | > \frac{2}{3}\mu_{2} \right] 
\Bigg)
\label{al:DBsumand2-2.4}
\\&\leq
3 (\Dim \VO)^2
\Prob\left[ \| \widehat{\bm{\Gamma}}^{p+1}_{Z}  - \bm{\Gamma}^{p+1}_{Z} \|_s > \frac{1}{3} \mu_{2} \right] 
\nonumber
\\&\leq
c_{1} \exp\left( -c_{2}N \min\{ 1, \nu^2, \nu/c_{0}(s) \} + 2 \log(\VO \Dim) \right).
\label{al:DBsumand2-2.6}
\end{align}
We used a union bound to infer \eqref{al:DBsumand2-2.3}. 
The relation $2 |v'Aw| \leq |v'Av| + |w'Aw| + |(v+w)'A(v+w)|$ implies \eqref{al:DBsumand2-2.4}.
Under Assumptions \ref{ass:CI-VAR} and \ref{ass:CItheta}, Corollary \ref{cor:CI} 
and
\eqref{eq:consequenceCor3.1} with $N \succsim \max\{ c_{0}(s)\nu^{-1}, \nu^{-2}, 1 \} \max\{ s\log(\Lag \Dim), \log( \Dim K) \}$ gives \eqref{al:DBsumand2-2.6}.
Note that $N = T - \VO \leq T$ and $\log((\VO+1) \Dim) = \log(1 + 1/\VO) + \log(\VO \Dim) \leq 2\log(\VO \Dim)$.

Combining \eqref{al:DBsumand1-2.5} and \eqref{al:DBsumand2-2.6}, there are constants $c_{1}, c_{2} >0$ such that
\begin{align*}
\Prob\left[ \| \widehat{\gamma} - \widehat{\Gamma} \beta_{0} \|_{\max} > 6 Q(\acmfZ) c_{0}(s) \nu \right]
&\leq c_{1} \exp\left( -c_{2}N \min\{ 1, \nu^2, \nu/c_{0}(s) \} \right).
\end{align*}
Choosing $\nu = \sqrt{ \frac{\log(q)}{N}}$, we get
\begin{equation*}
\| \widehat{\gamma} - \widehat{\Gamma} \beta_{0} \|_{\max} \leq 6 Q(\acmfZ) c_{0}(s) \sqrt{ \frac{\log(q)}{N}}
\end{equation*}
with high probability.
\end{proof}

\subsection{Verification of Assumptions \ref{ass:CI-VAR} and \ref{ass:CItheta}}
In this section, we verify Assumptions \ref{ass:CI-VAR} and \ref{ass:CItheta} in certain cases. More specifically, we show that Assumption \ref{ass:CI-VAR} is satisfied with $c_{0}(s)=s$ and for VAR($\VO$) models whenever the function $G$ is bounded. We refer to Section \ref{se:VARdisc} for discussions on $c_{0}(s)=s$ and boundedness of $G$.
We prove that Assumption \ref{ass:CItheta} is satisfied whenever the unknown CDF parameters $\theta_{i}$ can be estimated through the mean of the observed process. Examples include Bernoulli, binomial and negative hypergeometric distributions.

Write a VAR($\VO$) model as a $\VO \Dim$-dimensional VAR($1$) model, that is, 
\begin{equation} \label{eq:writtenasVAR1}
\begin{pmatrix}
Z_{t} \\
Z_{t-1} \\
\vdots \\
Z_{t-p+1}
\end{pmatrix}
=
\begin{pmatrix}
\TM_{1}	&	\cdots	& \TM_{\VO-1}	& \TM_{\VO} \\
I_{d}		&	\cdots	& 0			&	0 \\
\vdots	&	\ddots	&	\vdots	&	\vdots \\
0		&	\cdots	&	I_{d}		&	0\\
\end{pmatrix}
\begin{pmatrix}
Z_{t-1} \\
\vdots \\
Z_{t-p}
\end{pmatrix}
+
\begin{pmatrix}
\varepsilon_{t} \\
0 \\
\vdots \\
0
\end{pmatrix}
\hspace{0.2cm}
\text{ or }
\hspace{0.2cm}
Y_{t} = A Y_{t-1} + \widetilde{\varepsilon}_{t}.
\end{equation}
A VAR($1$) model in \eqref{eq:writtenasVAR1} is known to satisfy the Markov property and is also geometrically ergodic under assumption \eqref{eq:StabilityConditionVAR}; see p.\ 944 in \cite{an1996geometrical}. Under geometric ergodicity, Theorem 2.1 in \cite{roberts1997geometric} implies that there is a spectral gap $\lambda$ with $1-\lambda>0$.

For the verification of both Assumptions \ref{ass:CI-VAR} and \ref{ass:CItheta}, we will use the following concentration inequality for bounded functions of general-state-space Markov chains derived in \cite{fan2021hoeffding}. The result is expressed in terms of $\lambda_{r}$, the rightmost value of the spectrum $[-\lambda, \lambda]$. We refer to $1-\lambda_{r}$ as the right spectral gap of the Markov chain.

\begin{theorem}[Theorem 3 in \cite{fan2021hoeffding}] \label{theorem:CIMarkovchain}
Let $\{ Y_{t} \}_{t \geq 1}$ be a Markov chain on $\mathcal{X}$ with right spectral gap $1-\lambda_{r} >0$. For any $\varepsilon > 0$ and bounded function $f: \mathcal{X} \to [a,b]$, 
\begin{equation*}
\Prob \left[ \left| \frac{1}{T} \sum_{t =1}^{T} f(Y_{t}) - \frac{1}{T} \sum_{t =1}^{T} \E[f(Y_{t})] \right| > \varepsilon \right]
\leq
2\exp\left(- \frac{1-\max\{0,\lambda_{r}\}}{1+\max\{0,\lambda_{r}\}} \frac{T \varepsilon^{2}}{ (b-a)^{2}/2} \right).
\end{equation*}
\end{theorem}
We start with the verification of Assumption \ref{ass:CItheta} since it is slightly simpler. 

\textit{Verification of Assumption \ref{ass:CItheta}:}
First, consider the expected value and note that
\begin{equation*}
\E[\widehat{\theta}_{i}]
= 
\E\left[ \frac{1}{T}
\sum_{t =1}^{T} X_{i,t} \right]
=
\theta_{i}.
\end{equation*}
Then, 
\begin{align}
\Prob[ \max_{i =1, \dots, \Dim} | \widehat{\theta}_{i} - \theta_{i} | > \varepsilon ]
&
=
\Prob[ \max_{i =1, \dots, \Dim} | \widehat{\theta}_{i} - \E[\widehat{\theta}_{i}] | > \varepsilon ]
\nonumber
\\&
\leq
\Dim \Prob[ | \widehat{\theta}_{i} - \E[\widehat{\theta}_{i}] | > \varepsilon ] \nonumber
\\&
\leq
\Dim \Prob\left[ \left| \frac{1}{T} \sum_{t =1}^{T} X_{t,i} - \E[ \widehat{\theta}_{i} ] \right| > \varepsilon \right] \nonumber
\\&
=
\Dim \Prob\left[ \left| \frac{1}{T} \sum_{t =1}^{T} f(Y_{t}) - \frac{1}{T} \sum_{t =1}^{T} \E[f(Y_{t})] \right| > \varepsilon \right] \label{al:CIthetaineq1}
\\&
\leq
2 \Dim
\exp\left(- \frac{1-\max\{0,\lambda_{r}\}}{1+\max\{0,\lambda_{r}\}} \frac{T \varepsilon^{2}}{ 2 b^2} \right), \label{al:CIthetaineq2}
\end{align}
where \eqref{al:CIthetaineq1} follows by choosing the function $f$ as $f: y \mapsto (e_{\VO,i} \otimes e_{\Dim,i})' G(y)$. Then, it remains to verify that $f$ is bounded which follows since $f(y) = (e_{\VO,i} \otimes e_{\Dim,i})' G(y) \leq b (e_{\VO,i} \otimes e_{\Dim,i})' j_{\Dim \VO} = b$, where $j_{\Dim}$ denotes a $\Dim \VO$-dimensional column vector with all entries equal to one. Finally, \eqref{al:CIthetaineq2} is a consequence of applying Theorem \ref{theorem:CIMarkovchain}.

\textit{Verification of Assumption \ref{ass:CI-VAR}:}
We consider here the centered random variables $\widetilde{X}_{t} = X_{t} - \E X_{t}$ to estimate $\acmfX$.
First, note that the expected value of $\acmfXhat$ can be calculated as 
\begin{align}
&
\E \acmfXhat
\nonumber
\\&=
\frac{1}{N} \E \mathcal{X}_{X}' \mathcal{X}_{X}
=
\left(\frac{1}{N} \E \sum_{t=p}^{T-1} \widetilde{X}_{t-r+1} \widetilde{X}_{t-s+1}' \right)_{r,s =1,\dots, p} \nonumber
\\&=
\left( \frac{1}{N} \sum_{t=p}^{T-1} ( \E ( G(Z_{t-r+1}) G(Z_{t-s+1})' ) - \E G(Z_{t-r+1}) \E G(Z_{t-s+1})' )\right)_{r,s =1,\dots, p}
\nonumber
\\&=
\Bigg( \frac{1}{N} \sum_{t=p}^{T-1} ( \E ( G_{m}(Z_{m,t-r+1}) G_{n}(Z_{n,t-s+1})) 
\nonumber
\\&\hspace{1cm}
- \E G_{m}(Z_{m,t-r+1}) \E G_{n}(Z_{n,t-s+1}) ) _{m,n =1, \dots, \Dim} \Bigg)_{r,s =1,\dots, p}
\nonumber
\\&=
\left( \sum_{k = 1}^{\infty} \Big( \frac{c_{m,k}c_{n,k}}{k! } \Gamma_{Z,mn}(r-s)^k \Big)_{m,n =1, \dots, \Dim} \right)_{r,s =1,\dots, p}
= \acmfX,
\label{al:qwqwqw1.3}
\end{align}
where the first relation in \eqref{al:qwqwqw1.3} follows by applying the Hermite expansion \eqref{eq:Hermiteexpansion} to get
\begin{align}
&
\frac{1}{N} \sum_{t=p}^{T-1} \E G_{m}(Z_{m,t-r+1}) G_{n}(Z_{n,t-s+1})
\nonumber
\\&=
\frac{1}{N} \sum_{t=p}^{T-1} \sum_{k,l = 0}^{\infty} \frac{c_{m,k}c_{n,l}}{k! l!}
\E H_{k}(Z_{m,t-r+1}) H_{l}(Z_{n,t-s+1})
\nonumber
\\&=
\frac{1}{N} \sum_{t=p}^{T-1} \sum_{k = 0}^{\infty} \frac{c_{m,k}c_{n,k}}{k! } \Gamma_{Z,mn}(r-s)^k
\label{al:qwqwqw2}
\\&=
\sum_{k = 0}^{\infty} \frac{c_{m,k}c_{n,k}}{k! } \Gamma_{Z,mn}(r-s)^k,
\nonumber
\end{align}
where \eqref{al:qwqwqw2} follows by relation (5.1.4) in \cite{pipiras2017long}.
We now aim to apply Theorem \ref{theorem:CIMarkovchain}. Write
\begin{align}
\Prob[ | v' ( \acmfXhat - \acmfX ) v | > s \delta ]
&
=
\Prob[ | v' ( \acmfXhat - \E \acmfXhat ) v | > s \delta ] \nonumber
\\&
=
\Prob\left[ \Big| v' \Big( \frac{1}{N} \mathcal{X}_{X}'\mathcal{X}_{X} - \E \frac{1}{N} \mathcal{X}_{X}'\mathcal{X}_{X} \Big) v \Big| > s \delta \right] \nonumber
\\&
=
\Prob\left[ \Big| \frac{1}{N} \sum_{t =p}^{T-1} f(Y_{t}) - \frac{1}{N} \sum_{t =p}^{T-1} \E[f(Y_{t})] \Big| > s \delta \right] \label{al:CIGammaXineq1}
\\&
\leq
2 \exp\left(- \frac{1-\max\{0,\lambda_{r}\}}{1+\max\{0,\lambda_{r}\}} \frac{N \delta^{2}}{ 8 b^4} \right). \label{al:CIGammaXineq2}
\end{align}
The function $f$ in \eqref{al:CIGammaXineq1} is characterized below and satisfies $|f(y)| \leq b^2 4s$. Finally, \eqref{al:CIGammaXineq2} is a consequence of applying Theorem \ref{theorem:CIMarkovchain}.

To find the function $f$ in \eqref{al:CIGammaXineq1}, set $v = \vecop( [v_{1} : \dots : v_{\VO} ])$ with $v_{r} \in \RR^{\Dim}$ and $\widetilde{G}(Y_{t}) = (G(Z_{t})' - \E G(Z_{t})', \dots, G(Z_{t-p+1})' - \E G(Z_{t-p+1})' )'$. Then, the function $f$ in \eqref{al:CIGammaXineq1} can be determined through the following calculations:
\begin{align}
&
v' \mathcal{X}'_{X} \mathcal{X}_{X} v
\nonumber
\\&=
v' [ \widetilde{G}( Y_{p} ) : \widetilde{G}( Y_{p+1} ) : \cdots : \widetilde{G}( Y_{T-1} ) ] [ \widetilde{G}( Y_{p} ) : \widetilde{G}( Y_{p+1} ) : \cdots : \widetilde{G}( Y_{T-1} ) ]' v \nonumber
\\&=
v' \left( \sum_{t = p}^{T-1} (e'_{p,r} \otimes I_{d} ) \widetilde{G}(Y_{t} ) \widetilde{G}(Y_{t})' (e'_{p,s} \otimes I_{d})' \right)_{r,s =1, \dots, p} v \nonumber
\\&=
\sum_{t = p}^{T-1} \sum_{r,s =1}^{\VO} v'_{r} (e'_{p,r} \otimes I_{d} ) \widetilde{G}(Y_{t} ) \widetilde{G}(Y_{t})' (e'_{p,s} \otimes I_{d})' v_{s} \nonumber
=
\sum_{t = p}^{T-1} f(Y_{t})
\end{align}
with $f: y \mapsto \sum_{r,s =1}^{\VO} v'_{r} (e'_{p,r} \otimes I_{d} ) \widetilde{G}( y ) \widetilde{G}(y)' (e'_{p,s} \otimes I_{d})' v_{s}$. 
Then, it remains to verify that $f$ is bounded. Denote $J_{\Dim}$ as a $\Dim \times \Dim$-matrix with all entries equal to one and $j_{\Dim}$ as a $\Dim$-dimensional column vector with all entries equal to one. Then, with explanations given below,
\begin{align}
| f(y) |
&= 
\bigg|
\sum_{r,s =1}^{\VO} v'_{r} (e'_{\VO,r} \otimes I_{\Dim} ) \widetilde{G}( y ) \widetilde{G}(y)' (e'_{\VO,s} \otimes I_{\Dim})' v_{s} 
\bigg|
\nonumber
\\&\leq
4b^{2}
\sum_{r,s =1}^{\VO} |v'_{r}| (e'_{\VO,r} \otimes I_{\Dim} ) J_{\Dim \VO} (e'_{\VO,s} \otimes I_{\Dim})' |v_{s}| 
\label{al:fbounded1}
\\& =
4b^{2}
\sum_{r,s =1}^{\VO} |v'_{r}| (e'_{\VO,r} \otimes I_{\Dim} ) ( J_{\VO} \otimes  J_{\Dim}) (e_{\VO,s} \otimes I_{\Dim}) |v_{s}|
\nonumber
\\& =
4b^{2}
\sum_{r,s =1}^{\VO} |v'_{r}| (e'_{\VO,r} \otimes I_{\Dim} ) (j_{\VO} \otimes J_{\Dim}) |v_{s}|
\label{al:fbounded3}
\\& =
4b^{2}
\sum_{r,s =1}^{\VO} |v'_{r}| (1 \otimes J_{\Dim}) |v_{s}|
\label{al:fbounded4}
\\& =
4b^{2}
\sum_{r,s =1}^{\VO} \sum_{i =1}^{\Dim} |v_{r,i}| \sum_{j =1}^{\Dim} |v_{s,j}|
\nonumber
\\& \leq
2 b^{2}
\sum_{r,s =1}^{\VO} \sum_{i, j =1}^{\Dim} (v^{2}_{r,i} + v^{2}_{s,j} )
=
b^{2} 4s,
\label{al:fbounded6}
\end{align}
where \eqref{al:fbounded1} follows since we assume that $|G_{i}(y_{i}) - \E G_{i}(y_{i}) | \leq 2b$, \eqref{al:fbounded3} and \eqref{al:fbounded4} use a Kronecker product property in equation (4) on page 32 in \cite{MagnusNeudecker}; and for \eqref{al:fbounded6} note that $\| v \| =1$.

\section{Technical lemmas and their proofs} \label{se:TLandProofs}
Our technical lemmas required to prove our main results are separated into results on the inverse link function and its derivatives
(Section \ref{se:inverted_link_function}), the reciprocal of the first derivative of the link function (Section \ref{se:oneoverellprime}) and  bounds on the link function itself and its derivatives (Section \ref{se:link_function}). Finally, we consider the diagonal elements separately (Section \ref{se:diagonal_elements}).

\subsection{Inverse link function and its derivatives} \label{se:inverted_link_function}

This section provides high probability bounds for expressions of the form
\begin{equation} \label{eq:TLIstructure}
\norm{ (\widehat{g}^{(a)}_{\bullet}(\bm{\Sigma}) - g^{(a)}_{\bullet}(\bm{\Sigma})) \odot (\acmfXhat - \acmfX)^{a} }, 
\hspace{0.2cm}
a \in \{0,1, 2\},
\end{equation}
where the function $g$ is the inverse of the link function $\ell$, $g_{\bullet}$ is defined in \eqref{eq:gdot} and $\bm{\Sigma}$ is such that $\bm{\Sigma} = \acmfX$ or 
$\left| \bm{\Sigma} - \acmfX \right| < \big| \acmfXhat - \acmfX \big|$.
Note that $g^{(a)}$ diverges from our previous notation and denotes the function $g$ itself ($a=0$) and its first and second derivatives ($a=1,2$).
The following lemma covers the case $a=0$ in \eqref{eq:TLIstructure}.

\begin{lemma} \label{le:ghatSigmagSigma}
Suppose Assumptions \ref{as:bound for entries}--\ref{ass:finitemoments3}. Then, for any $\delta, \widetilde{\delta}, \varepsilon >0$, 
\begin{align*}
&
\Prob[\norm{\widehat{g}_{\bullet}(\acmfX) - g_{\bullet}(\acmfX) } > q_{1}(\acmfZ) \delta]
\\& \precsim
\Prob\left[ \| \widehat{\ell}_{\bullet}(\acmfZ) - \ell_{\bullet}( \acmfZ) \|_{s} > \delta \wedge \widetilde{\delta} \right]
+
\Prob[\| \widehat{\ell}_{\bullet}(\acmfZ) - \ell_{\bullet}(\acmfZ) \|_{s}^{2} > \delta ]
\\&\hspace{1cm}+
\Prob[ \| \widehat{\theta} - \theta \|_{\maxF} > \varepsilon ]
\end{align*}
with $q_{1}(\acmfZ) := q_{1}(\varepsilon, \widetilde{\delta}, \acmfZ) = \max\{ 2R(\varepsilon, \acmfZ), T(\varepsilon, \widetilde{\delta}, \acmfZ) \} $.
\end{lemma}

\begin{proof}
Noting that $g(\acmfX) = \ell^{-1}(\ell(\acmfZ)) = \acmfZ = \widehat{\ell}^{-1}(\widehat{\ell}(\acmfZ))$, we get
\begin{align}
&
\norm{\widehat{g}_{\bullet}(\acmfX) - g_{\bullet}(\acmfX) } \nonumber
\\&=
\norm{\widehat{\ell}^{-1}_{\bullet}(\ell(\acmfZ)) - \ell^{-1}_{\bullet}(\ell(\acmfZ)) }
=
\norm{\widehat{\ell}^{-1}_{\bullet}(\ell(\acmfZ)) - \widehat{\ell}^{-1}_{\bullet}(\widehat{\ell}(\acmfZ)) } \nonumber
\\&\leq
\norm{(\widehat{\ell}^{-1}_{\bullet})' (\widehat{\ell}(\acmfZ)) \odot ( \ell_{\bullet}(\acmfZ) - \widehat{\ell}_{\bullet}(\acmfZ)) }
\notag
\\&\hspace{1cm}+
\frac{1}{2} \norm{(\widehat{\ell}^{-1}_{\bullet})'' (\Sigma) \odot ( \ell_{\bullet}(\acmfZ) - \widehat{\ell}_{\bullet}(\acmfZ)) ^{ \odot 2} } \label{al:a4}
\\&=
\norm{ (\widehat{\ell}^{\prime})^{\odot (-1)}_{\bullet}(\acmfZ) \odot ( \ell_{\bullet}(\acmfZ) - \widehat{\ell}_{\bullet}(\acmfZ)) }
\notag
\\&\hspace{1cm}+
\frac{1}{2} \norm{(\widehat{\ell}^{-1}_{\bullet})'' (\Sigma) \odot ( \ell_{\bullet}(\acmfZ) - \widehat{\ell}_{\bullet}(\acmfZ)) ^{ \odot 2} }, \label{al:a5}
\end{align}
where \eqref{al:a4} follows by the second-order Taylor expansion of $\widehat{\ell}^{-1}$ around $\widehat{\ell}(\acmfZ)$ for some $\Sigma$ such that 
$\left| \Sigma - \ell(\acmfZ) \right| < \big| \widehat{\ell}(\acmfZ) - \ell(\acmfZ) \big|$. 
For the equality \eqref{al:a5}, we use \eqref{eq:firstderivellinv}. We further bound the probabilities of the two terms in \eqref{al:a5} separately.
\begin{align}
&
\Prob\left[
\norm{\widehat{g}_{\bullet}(\acmfX) - g_{\bullet}(\acmfX) }
> q_{1}(\acmfZ) \delta \right]
\nonumber
\\&\leq
\Prob\left[
\norm{ (\widehat{\ell}^{\prime})^{\odot (-1)}_{\bullet}(\acmfZ) \odot (\widehat{\ell}_{\bullet}(\acmfZ) - \ell_{\bullet}( \acmfZ))}
> q_{1}(\acmfZ) \frac{\delta}{2} \right]
\nonumber
\\&\hspace{1cm}+
\Prob \left[ \frac{1}{2} \norm{ \widehat{g}''_{\bullet}(\Sigma) \odot ( \widehat{\ell}_{\bullet}(\acmfZ) - \ell_{\bullet}(\acmfZ) )^{ \odot 2} } > q_{1}(\acmfZ) \frac{\delta}{2} \right]
\nonumber
\\&\leq
\Prob\left[
\norm{ (\widehat{\ell}^{\prime})^{\odot (-1)}_{\bullet}(\acmfZ) \odot (\widehat{\ell}_{\bullet}(\acmfZ) - \ell_{\bullet}( \acmfZ))}
> R(\varepsilon, \acmfZ) \delta \right]
\nonumber
\\&\hspace{1cm}+
\Prob \left[\norm{ \widehat{g}''_{\bullet}(\Sigma) \odot ( \widehat{\ell}_{\bullet}(\acmfZ) - \ell_{\bullet}(\acmfZ) )^{ \odot 2} } > T(\varepsilon, \widetilde{\delta}, \acmfZ) \delta \right]
\label{al:blablablabla}
\\&\precsim
\Prob\left[ \norm{ \widehat{\ell}_{\bullet}(\acmfZ) - \ell_{\bullet}( \acmfZ) } > \delta \wedge \widetilde{\delta} \right]
+
\Prob[\| \widehat{\ell}_{\bullet}(\acmfZ) - \ell_{\bullet}(\acmfZ) \|_{s}^{2} > \delta ]
\nonumber
\\&\hspace{1cm}+
\Prob[ \| \widehat{\theta} - \theta \|_{\maxF} > \varepsilon ],
\label{al:blablablabla1}
\end{align}
where \eqref{al:blablablabla} follows since $q_{1}(\acmfZ) = \max\{ 2R(\varepsilon, \acmfZ), T(\varepsilon, \widetilde{\delta}, \acmfZ) \} $, and \eqref{al:blablablabla1} follows from Lemmas \ref{le:derivofgHATtimesellhatell} and \ref{le:estgprimeprimeellofSigma} since $R(\acmfZ) := R(\varepsilon, \acmfZ)$ and $q_{4}(\acmfZ) = T(\varepsilon, \widetilde{\delta}, \acmfZ)$.
\end{proof}

The following lemma concerns the case $a=1$ in \eqref{eq:TLIstructure}.

\begin{lemma} \label{le:ghatSigmagprimegSigmaprime}
Suppose Assumptions \ref{as:bound for entries}--\ref{ass:finitemoments3}. Then, for any $\delta, \widetilde{\delta}, \varepsilon >0$, 
\begin{equation}
\begin{aligned}
&\Prob \left[ \norm{(\widehat{g}'_{\bullet}(\acmfX) - g'_{\bullet}(\acmfX)) \odot ( \acmfXhat - \acmfX )} > q_{2}(\acmfZ) \delta \right]
\\&\precsim
\Prob[\| \acmfXhat - \acmfX \|_{s} > \delta \wedge \widetilde{\delta} ]
+ 
\Prob[\| \acmfXhat - \acmfX \|_{s}^{2} > \delta ] 
\notag
\\&\hspace{1cm}+
\Prob[\| \widehat{\ell}_{\bullet}(\acmfZ) - \ell_{\bullet}(\acmfZ) \|_{s} > \widetilde{\delta} ]
+
\Prob[\| \widehat{\ell}_{\bullet}(\acmfZ) - \ell_{\bullet}(\acmfZ) \|_{s}^{2} > \delta ]
\notag
\\&\hspace{1cm}+
\Prob[ \| \widehat{\theta} - \theta \|_{\maxF} > \varepsilon ]
\end{aligned}
\end{equation}
with $q_{2}(\acmfZ) := q_{2}(\varepsilon, \widetilde{\delta}, \acmfZ) = \max\{ 4R(\varepsilon, \acmfZ), 2T(\varepsilon, \widetilde{\delta}, \acmfZ) \}$.
\end{lemma}

\begin{proof}
With explanations given below, we bound the quantity of interest as follows:
\begin{align}
&\norm{(\widehat{g}'_{\bullet}(\acmfX) - g'_{\bullet}(\acmfX)) \odot ( \acmfXhat - \acmfX )} \nonumber
\\&=
\norm{ (\widehat{g}'_{\bullet}(\ell(\acmfZ)) - g'_{\bullet}(\ell(\acmfZ)))
\odot 
( \acmfXhat - \acmfX )
} \nonumber
\\&\leq
\norm{
(\widehat{g}'_{\bullet}(\ell(\acmfZ)) - \widehat{g}'_{\bullet}(\widehat{\ell}(\acmfZ)))
\odot 
( \acmfXhat - \acmfX )
} 
\notag
\\&\hspace{1cm}+
\norm{
(\widehat{g}'_{\bullet}(\widehat{\ell}(\acmfZ)) - g'_{\bullet}(\ell(\acmfZ)))
\odot 
( \acmfXhat - \acmfX )
} \nonumber
\\&=
\norm{ \widehat{g}''_{\bullet}(\Sigma)
\odot
(\widehat{\ell}_{\bullet}(\acmfZ) - \ell_{\bullet}(\acmfZ) ) \odot ( \acmfXhat - \acmfX ) } 
\notag
\\&\hspace{1cm}+
\norm{
\left(
(\widehat{\ell}^{\prime})^{\odot (-1)}_{\bullet}(\acmfZ) - (\ell^{\prime})^{\odot (-1)}_{\bullet}(\acmfZ) \right)
\odot ( \acmfXhat - \acmfX )} \label{al:mvtlastal-1}
\\&\leq
\frac{1}{2} \norm{ |\widehat{g}''_{\bullet}(\Sigma)|
\odot
(\widehat{\ell}_{\bullet}(\acmfZ) - \ell_{\bullet}(\acmfZ) ) ^{\odot 2} }
+
\frac{1}{2} \norm{ | \widehat{g}''_{\bullet}(\Sigma) |
\odot
( \acmfXhat - \acmfX ) ^{\odot 2}} \notag
\\&\hspace{1cm}+
\norm{
(\widehat{\ell}^{\prime})^{\odot (-1)}_{\bullet}(\acmfZ)
\odot ( \acmfXhat - \acmfX )} 
+
\norm{
(\ell^{\prime})^{\odot (-1)}_{\bullet}(\acmfZ)
\odot ( \acmfXhat - \acmfX )} , \label{al:mvtlastal-2}
\end{align}
where \eqref{al:mvtlastal-1} follows by the mean value theorem for some $\Sigma$ such that 
$\left| \Sigma - \ell(\acmfZ) \right| < \big| \widehat{\ell}(\acmfZ) - \ell(\acmfZ) \big|$ and the last line \eqref{al:mvtlastal-2} is a consequence of $| A \odot B | \leq \frac{1}{2} (A^{\odot 2} + B^{\odot 2})$ and \eqref{le:matrixprops2-1} in Lemma \ref{le:matrixprops2}. The four summands in \eqref{al:mvtlastal-2} can be handled through subsequent Lemmas as follows
\begin{align}
&
\Prob \left[ \norm{(\widehat{g}'_{\bullet}(\acmfX) - g'_{\bullet}(\acmfX)) \odot ( \acmfXhat - \acmfX )} > q_{2}(\acmfZ) \delta \right]
\nonumber
\\&\leq
\Prob \left[ 
\norm{ | \widehat{g}''_{\bullet}(\Sigma) |
\odot
(\widehat{\ell}_{\bullet}(\acmfZ) - \ell_{\bullet}(\acmfZ) ) ^{\odot 2} } > q_{2}(\acmfZ) \delta/2 \right]
\nonumber
\\&\hspace{1cm}+
\Prob \left[ \norm{ | \widehat{g}''_{\bullet}(\Sigma) |
\odot
( \acmfXhat - \acmfX ) ^{\odot 2}} > q_{2}(\acmfZ)\delta/2 \right]
\nonumber
\\&\hspace{2cm}+
\Prob \left[ 
\norm{
(\ell^{\prime})^{\odot (-1)}_{\bullet}(\acmfZ)
\odot ( \acmfXhat - \acmfX )} > q_{2}(\acmfZ) \delta/4 \right]
\nonumber
\\&\hspace{3cm}+
\Prob \left[ 
\norm{ (\widehat{\ell}^{\prime})^{\odot (-1)}_{\bullet}(\acmfZ) \odot (\acmfXhat - \acmfX)}
> q_{2}(\acmfZ) \delta/4 \right]
\nonumber
\\&\leq
\Prob \left[ 
\norm{ | \widehat{g}''_{\bullet}(\Sigma) |
\odot
(\widehat{\ell}_{\bullet}(\acmfZ) - \ell_{\bullet}(\acmfZ) ) ^{\odot 2} } >  T(\varepsilon, \widetilde{\delta}, \acmfZ) \delta \right]
\nonumber
\\&\hspace{1cm}+
\Prob \left[ \norm{ | \widehat{g}''_{\bullet}(\Sigma) |
\odot
( \acmfXhat - \acmfX ) ^{\odot 2}} > T(\varepsilon, \widetilde{\delta}, \acmfZ) \delta \right]
\nonumber
\\&\hspace{2cm}+
\Prob\left[ \norm{ (\widehat{\ell}^{\prime})^{\odot (-1)}_{\bullet}(\acmfZ)
\odot ( \acmfXhat - \acmfX ) } > R(\varepsilon, \acmfZ) \delta \right]
\nonumber
\\&\hspace{3cm}+
\Prob\left[ \norm{ (\ell^{\prime})^{\odot (-1)}_{\bullet}(\acmfZ) \odot (\acmfXhat - \acmfX)} > R(0,\acmfZ) \delta \right]
\label{al:tytytytyty1}
\\&\precsim
\Prob[\| \widehat{\ell}_{\bullet}(\acmfZ) - \ell_{\bullet}(\acmfZ) \|_{s}^{2} > \delta ]
+ 
\Prob[\| \acmfXhat - \acmfX \|_{s}^{2} > \delta ] \notag
\\&\hspace{1cm}+
\Prob[\| \acmfXhat - \acmfX \|_{s} > \widetilde{\delta} ]
+
\Prob[\| \widehat{\ell}_{\bullet}(\acmfZ) - \ell_{\bullet}(\acmfZ) \|_{s} > \widetilde{\delta} ]
\notag
\\&\hspace{2cm}+ 
\Prob[\| \acmfXhat - \acmfX \|_{s} > \delta ]
+ 
\Prob[ \| \widehat{\theta} - \theta \|_{\maxF} > \varepsilon ]
 \label{al:tytytytyty2} 
\\&\precsim
\Prob[\| \acmfXhat - \acmfX \|_{s} > \delta \wedge \widetilde{\delta} ]
+ 
\Prob[\| \acmfXhat - \acmfX \|_{s}^{2} > \delta ) \notag
\\&\hspace{1cm}+
\Prob[\| \widehat{\ell}_{\bullet}(\acmfZ) - \ell_{\bullet}(\acmfZ) \|_{s} > \widetilde{\delta} ]
+
\Prob[\| \widehat{\ell}_{\bullet}(\acmfZ) - \ell_{\bullet}(\acmfZ) \|_{s}^{2} > \delta ]
\notag
\\&\hspace{2cm}+
\Prob[ \| \widehat{\theta} - \theta \|_{\maxF} > \varepsilon ]
\nonumber
\end{align}
where \eqref{al:tytytytyty1} follows since
\begin{align*}
q_{2}(\acmfZ) 
=
\max\{ 4R(\varepsilon, \acmfZ), 2T(\varepsilon, \widetilde{\delta}, \acmfZ) \}
=
\max\{ 4R(0,\acmfZ), 4R(\varepsilon, \acmfZ), 2T(\varepsilon, \widetilde{\delta}, \acmfZ) \}
\end{align*}
with $R, T$ as in \eqref{eq:quantitiesmainresult}.
The relation \eqref{al:tytytytyty2} follows from Lemmas \ref{le:estgprimeprimeellofSigma}
\ref{le:estgprimeprimeSigma}, \ref{le:derivofgHAT} and \ref{le:derivofg} since $R(\acmfZ) := R(\varepsilon, \acmfZ)$ and $q_{4}(\acmfZ) = T(\varepsilon, \widetilde{\delta}, \acmfZ)$. 
\end{proof}

We proceed with finding bounds on \eqref{eq:TLIstructure} with $a=2$ and reduce the problem to expressions of the form
\begin{align*}
\norm{ g^{\prime \prime}_{\bullet}(\Sigma) \odot (\acmfXhat - \acmfX)^{ \odot 2} }.
\end{align*}
In particular, we distinguish the cases when $g$ is known (Lemma \ref{le:detgprimeprimeSigma}) and when $g$ is estimated (Lemma \ref{le:estgprimeprimeSigma}).

\begin{lemma} \label{le:detgprimeprimeSigma}
Suppose Assumptions \ref{as:bound for entries}--\ref{ass:finitemoments3} and let $\Sigma$ be such that $\left| \Sigma - \acmfX \right| < \big| \acmfXhat - \acmfX \big|$. Then, for any $\delta, \widetilde{\delta} >0 $, 
\begin{align*}
\Prob\left[\norm{ g_{\bullet}^{\prime \prime}(\Sigma) \odot (\acmfXhat - \acmfX)^{ \odot 2} } > q_{3}(\acmfZ) \delta \right]
&\leq
\Prob\left[\| \acmfXhat - \acmfX \|_{s}^{2} > \delta \right]
\\&\hspace{1cm}+
\Prob[ \| \acmfXhat - \acmfX \|_{s} > \widetilde{\delta} ]
\end{align*}
with 
\begin{equation*}
q_{3}(\acmfZ) := q_{3}(\widetilde{\delta}, \acmfZ) = \frac{6}{1-\bm{c}(\widetilde{\delta})^2} M_{1}(\bm{c}(\widetilde{\delta}),0) M_{2}(\bm{c}(\widetilde{\delta}),0), 
\end{equation*}
where $M_{1}, M_{2}$ are as in \eqref{eq:mumoments2}.
\end{lemma}

\begin{proof}
With explanations given below, we get
\begin{align}
&
\Prob\left[\norm{ g^{\prime \prime}_{\bullet}(\Sigma) \odot (\acmfXhat - \acmfX)^{ \odot 2} } > q_{3}(\acmfZ) \delta \right]\nonumber
\\ & =
\Prob \bigg[
\left\{ \norm{ g^{\prime \prime}_{\bullet}(\Sigma) \odot (\acmfXhat - \acmfX)^{ \odot 2} } > q_{3}(\acmfZ) \delta \right\} 
\notag
\\&\hspace{1cm}\cap
\Big(\{ \| \acmfXhat - \acmfX \|_{s} \leq \widetilde{\delta} \} \cup \{ \| \acmfXhat - \acmfX \|_{s} > \widetilde{\delta} \} \Big)
\bigg] \nonumber
\\ & \leq
\Prob \bigg[
\left\{ \norm{ g^{\prime \prime}_{\bullet}(\Sigma) \odot (\acmfXhat - \acmfX)^{ \odot 2} } > q_{3}(\acmfZ) \delta \right\} \cap
\{ \| \acmfXhat - \acmfX \|_{\max} \leq 4 \widetilde{\delta} \} \bigg] \nonumber
\\&\hspace{1cm}+
\Prob \bigg[ \| \acmfXhat - \acmfX \|_{s} > \widetilde{\delta} 
\bigg]
\label{eq:ineq0}
\\ & \leq
\Prob\left[ \sup_{\sigma \in \Omega(\widetilde{\delta})} | g^{\prime \prime}_{\bullet}( \sigma ) | \| \acmfXhat - \acmfX \|^{2}_{s} > q_{3}(\acmfZ) \delta \right]
+ 
\Prob[ \| \acmfXhat - \acmfX \|_{s} > \widetilde{\delta} ]
\label{eq:ineq1}
&
\\ & \leq
\Prob[ \| \acmfXhat - \acmfX \|_{s}^{2} > \delta ]
+ 
\Prob[ \| \acmfXhat - \acmfX \|_{s} > \widetilde{\delta} ] \label{eq:ineq2}.
\end{align}
The bound \eqref{eq:ineq0} follows since 
\begin{equation} \label{ine:maxsupnorm}
\begin{aligned}
2 \| A \|_{\max} 
&= \max_{i,j =1, \dots, \Dim} 2 |e_{\Dim,i}' A e_{\Dim,j}| 
\\&\leq 
\max_{i =1, \dots, \Dim} |e_{\Dim,i}' A e_{\Dim,i}| +
\max_{j =1, \dots, \Dim} |e_{\Dim,j}' A e_{\Dim,j}| 
\notag
\\&\hspace{1cm}+
2 \max_{i,j =1, \dots, \Dim} \left| \left( \frac{e_{\Dim,i} + e_{\Dim,j}}{ \sqrt{2} } \right) ' A \left( \frac{e_{\Dim,i} + e_{\Dim,j}}{ \sqrt{2} } \right) \right| 
\\&\leq 
4 \sup_{v \in \mathcal{K}(2s)} |v'Av| = 4 \| A \|_{s}
\end{aligned}
\end{equation}
for a matrix $A \in \RR^{\Dim \times \Dim}$. 
For \eqref{eq:ineq1}, set $\Omega(\widetilde{\delta}) = \{ \widetilde{\sigma}_{X, ij} ~|~ \max_{i,j =1,\dots, \Dim} | \widetilde{\sigma}_{X, ij} - \sigma_{X, ij} | \leq 2 \widetilde{\delta} \}$. 

Note that $g''(x) = f(\ell^{-1}(x))$ as stated in \eqref{eq:secndderivellinv} in the appendix. In order to apply Lemma \ref{le:gprimeprimemonotone}, we need to find a bound on $\ell^{-1}(\sigma)$ uniformly over all $ \sigma \in \Omega(\widetilde{\delta})$.

The function $\ell$ is strictly increasing as a consequence of Proposition \ref{prop:prop2.1}. For strictly increasing functions, its inverse $\ell^{-1}$ exists and is also strictly increasing. 
Recall that by Assumption \ref{as:bound for entries} there is a a constant $\conZ \in (0,1)$ such that $|\Gamma_{Z,ij}(h) | < \conZ < 1$. Then, $-\bm{c}_{Z} < \Gamma_{Z,ij}(h) < \bm{c}_{Z}$ and $\ell_{ij}(-\conZ) < \ell_{ij}(\Gamma_{Z,ij}(h) ) = \Gamma_{X,ij}(h) < \ell_{ij}(\conZ) < \ell_{ij}(1)$.
Since $\max_{i,j =1,\dots, \Dim} | \widetilde{\sigma}_{X, ij} - \sigma_{X, ij} | \leq 2 \widetilde{\delta}$ we further argue that for $\sigma \in \Omega(\widetilde{\delta})$,
\begin{equation*}
\ell^{-1}_{ij}(\sigma)
\leq
\ell^{-1}_{ij}( \Gamma_{X,ij}(h) + 2 \widetilde{\delta}) 
<
\ell^{-1}_{ij}(\ell_{ij} (1) ) = 
1.
\end{equation*}
Therefore, for $\widetilde{\delta}$ small, there is a constant $\bm{c}(\widetilde{\delta}) < 1$ such that the assumptions in Lemma \ref{le:gprimeprimemonotone} are satisfied and therefore $ \sup_{\sigma \in \Omega(\widetilde{\delta})} | g^{\prime \prime}_{\bullet}( \sigma ) | \leq q_{3}(\acmfZ)$.
\end{proof}

The following lemma is the analogue of Lemma \ref{le:detgprimeprimeSigma} for the estimated counterpart $\widehat{g}''$.
\begin{lemma} \label{le:estgprimeprimeSigma}
Suppose Assumptions \ref{as:bound for entries}--\ref{ass:finitemoments3} and let $\Sigma$ be such that $\left| \Sigma - \acmfX \right| < \big| \acmfXhat - \acmfX \big|$. Then, for any $\delta, \widetilde{\delta}, \varepsilon >0 $, 
\begin{equation*}
\begin{aligned}
&
\Prob\left[\norm{ \widehat{g}''_{\bullet}(\Sigma) \odot (\acmfXhat - \acmfX)^{ \odot 2} } > 
q_{4}(\acmfZ) \delta \right]
\\&\leq
\Prob[ \| \acmfXhat - \acmfX \|_{s} > \widetilde{\delta} ]
+ 
\Prob[ \| \acmfXhat - \acmfX \|_{s}^{2} > \delta ]
\\&\hspace{1cm}+ 
\Prob[  \| \widehat{\theta} - \theta \|_{\maxF} > \varepsilon ]
\end{aligned}
\end{equation*}
with
\begin{equation} \label{eq:q41Sigma}
q_{4}(\acmfZ) := q_{4}(\varepsilon, \widetilde{\delta}, \acmfZ) = \frac{6}{1-\bm{c}(\widetilde{\delta})^2}  M_{1}(\bm{c}(\widetilde{\delta}),\varepsilon) M_{2}(\bm{c}(\widetilde{\delta}),\varepsilon), 
\end{equation}
where $M_{1}, M_{2}$ are as in \eqref{eq:mumoments2}.
\end{lemma}

\begin{proof}[Proof]
The proof consists of two parts in order to bound $\widehat{g}''(\Sigma)$ across all elements. The matrix $\Sigma$ satisfies $\left| \Sigma - \acmfX \right| < \big| \acmfXhat - \acmfX \big|$ but $g''$ is generally not bounded on the whole interval $(-1,1)$. Therefore, we need to control how much $\Sigma$ deviates from the true $\acmfX$ (Step 1). Furthermore, $\widehat{g}''$ depends on $\widehat{\theta}$ and needs to be bounded across all possible values of $\theta$ (Step 2).  

\textit{Step 1:} With explanations given below,
\begin{align}
&
\Prob\left[\norm{ \widehat{g}^{\prime \prime}_{\bullet}(\Sigma) \odot ( \acmfXhat - \acmfX )^{ \odot 2} } > 
q_{4}(\acmfZ) \delta \right] \nonumber
\\ & =
\Prob \bigg[
\left\{ \norm{ \widehat{g}^{\prime \prime}_{\bullet}(\Sigma) \odot ( \acmfXhat - \acmfX )^{ \odot 2} } > 
q_{4}(\acmfZ) \delta \right\} 
\notag
\\&\hspace{1cm}\cap
\Big(\{ \| \acmfXhat - \acmfX \|_{s} \leq \widetilde{\delta} \} \cup \{ \| \acmfXhat - \acmfX \|_{s} > \widetilde{\delta} \} \Big)
\bigg] 
\nonumber
\\ & \leq
\Prob \Bigg[
\Bigg\{ \norm{ \sup_{\sigma \in \Omega(\widetilde{\delta})} | \widehat{g}^{\prime \prime}_{\bullet}(\Sigma) | \odot ( \acmfXhat - \acmfX )^{ \odot 2} } > 
q_{4}(\acmfZ) \delta \Bigg\} 
\nonumber
\\ &\hspace{1cm} \cap
\{ \| \acmfXhat - \acmfX \|_{\max} \leq 4 \widetilde{\delta} \} \Bigg]
+
\Prob [ \| \acmfXhat - \acmfX \|_{s} > \widetilde{\delta} 
] \label{eq:ineq01}
\\ & \leq
\Prob\Bigg[
\frac{6}{1-\bm{c}(\widetilde{\delta})^2} \widehat{M}_{1}(\bm{c}(\widetilde{\delta}),0) \widehat{M}_{2}(\bm{c}(\widetilde{\delta}),0) \| \acmfXhat - \acmfX \|^{2}_{s} > q_{4}(\acmfZ) \delta \Bigg]
\notag
\\&\hspace{1cm}+ 
\Prob[ \| \acmfXhat - \acmfX \|_{s} > \widetilde{\delta} ]
\label{eq:ineq21}
\\ & \leq
\Prob\left[
\widehat{q}_{4}(0, \widetilde{\delta}, \acmfZ) \| \acmfXhat - \acmfX \|^{2}_{s} > q_{4}(\acmfZ) \delta \right]
+ 
\Prob[ \| \acmfXhat - \acmfX \|_{s} > \widetilde{\delta} ],
\nonumber
\end{align}
where $\widehat{q}_{4}(0, \widetilde{\delta}, \acmfZ)$ denotes the estimated counterpart of $q_{4}(0, \widetilde{\delta}, \acmfZ)$ in \eqref{eq:q41Sigma}.
The bound \eqref{eq:ineq01} follows from \eqref{ine:maxsupnorm}. Then, applying the same strategy as to get from \eqref{eq:ineq1} to \eqref{eq:ineq2} but for an estimator $\widehat{g}$ of $g$, we get \eqref{eq:ineq21}.

\textit{Step 2:}
For the second step, we argue similarly, 
\begin{align}
&
\Prob\left[\widehat{q}_{4}(0, \widetilde{\delta}, \acmfZ) \| \acmfXhat - \acmfX \|_{s}^{2} > 
q_{4}(\acmfZ) \delta \right]
\nonumber
\\& =
\Prob \bigg[
\bigg\{ \widehat{q}_{4}(0, \widetilde{\delta}, \acmfZ) \| \acmfXhat - \acmfX \|_{s}^{2} > 
q_{4}(\acmfZ) \delta \bigg\}
\nonumber
\\&\hspace{1cm} \cap
\Big(\{ \| \widehat{\theta} - \theta \|_{\maxF} \leq \varepsilon \} \cup \{ \| \widehat{\theta} - \theta \|_{\maxF} > \varepsilon \} \Big)
\bigg] 
\nonumber
\\ & \leq
\Prob \bigg[
\bigg\{ 
\widehat{q}_{4}(0, \widetilde{\delta}, \acmfZ)  \| \acmfXhat - \acmfX \|_{s}^{2} > q_{4}(\acmfZ) \delta \bigg\} \cap
\{ \| \widehat{\theta} - \theta \|_{\max} \leq \varepsilon \} \bigg]
\nonumber
\\&\hspace{1cm}+
\Prob[  \| \widehat{\theta} - \theta \|_{\maxF} > \varepsilon ]
\nonumber
%
\\ & \leq
\Prob \bigg[
\sup_{\theta_{i} \in \Theta(\varepsilon)}
q_{4}(0, \widetilde{\delta}, \acmfZ) \| \acmfXhat - \acmfX \|_{s}^{2} > q_{4}(\acmfZ) \delta \bigg]
\nonumber
\\&\hspace{1cm}+
\Prob[  \| \widehat{\theta} - \theta \|_{\maxF} > \varepsilon ]
\nonumber
\\ & \leq
\Prob[ \| \acmfXhat - \acmfX \|_{s}^{2} > \delta ]
+ 
\Prob[  \| \widehat{\theta} - \theta \|_{\maxF} > \varepsilon ],
\nonumber
\end{align}
since $\sup_{\theta_{i} \in \Theta(\varepsilon)}
q_{4}(0, \widetilde{\delta}, \acmfZ) = q_{4}(\varepsilon, \widetilde{\delta}, \acmfZ)$.
\end{proof}

We also state the result in terms of $\widehat{\ell}(\acmfZ) - \ell(\acmfZ)$ instead of $\acmfXhat - \acmfX$ and omit its proof.
\begin{lemma} \label{le:estgprimeprimeellofSigma}
Suppose Assumptions \ref{as:bound for entries}--\ref{ass:finitemoments3} and let $\Sigma$ be such that $\left| \Sigma - \ell(\acmfZ) \right| < \big| \widehat{\ell}(\acmfZ) - \ell(\acmfZ) \big|$. Then, for any $\delta, \widetilde{\delta}, \varepsilon >0 $, 
\begin{equation*}
\begin{aligned}
&
\Prob\left[\norm{ \widehat{g}''_{\bullet}(\Sigma) \odot ( \widehat{\ell}_{\bullet}(\acmfZ) - \ell_{\bullet}(\acmfZ) )^{ \odot 2} } > q_{4}(\acmfZ) \delta \right]
\\&\leq
\Prob[ \| \widehat{\ell}_{\bullet}(\acmfZ) - \ell_{\bullet}(\acmfZ) \|_{s} > \widetilde{\delta} ]
+
\Prob[ \| \widehat{\ell}_{\bullet}(\acmfZ) - \ell_{\bullet}(\acmfZ) \|_{s}^{2} > \delta ]
\\&\hspace{1cm}+
\Prob[  \| \widehat{\theta} - \theta \|_{\maxF} > \varepsilon ]
\end{aligned}
\end{equation*}
with $q_{4}(\acmfZ)$ as in \eqref{eq:q41Sigma}.
\end{lemma}

\subsection{Reciprocal of the first derivative of link function} \label{se:oneoverellprime}
This section concerns high probability bounds for expressions of the form
\begin{align} \label{eq:genellinvSigma}
\norm{ (\ell^{\prime})_{\bullet}^{\odot (-1)}(\acmfZ) \odot (\acmfXhat - \acmfX)}.
\end{align}
We study the cases when $\ell$ is known (Section \ref{se:determ}) and when $\ell$ is estimated (Section \ref{se:estimated}).

\subsubsection{Case of known link function} \label{se:determ}
In this section, we focus on \eqref{eq:genellinvSigma} given that the link function $\ell$ is known.

\begin{lemma} \label{le:derivofg}
Suppose Assumptions \ref{as:bound for entries}--\ref{ass:finitemoments3}. Then, 
\begin{align}
\norm{ (\ell^{\prime})_{\bullet}^{\odot (-1)}(\acmfZ) \odot (\acmfXhat - \acmfX)}
\leq
R(0,\acmfZ) \norm{\acmfXhat - \acmfX}
\end{align}
with 
\begin{equation*}
R(0,\acmfZ) = \Big(
24 \pi \frac{1}{(1-\conZ^2)^2}
M_{2}(\conZ,0)
+ 8 \pi M_{1}(0,0) \Big) \norm{ \acmfZ }
\end{equation*}
and $M_{1}, M_{2}$ are as in \eqref{eq:mumoments2}. 
\end{lemma}

\begin{proof}
With explanations given below,
\begin{align}
&
\norm{ (\ell^{\prime})_{\bullet}^{\odot (-1)}(\acmfZ) \odot (\acmfXhat - \acmfX)}
\nonumber
\\&=
2\pi \norm{ (J_{\Dim \Lag} - \acmfZ^{\odot 2} )^{\odot \frac{1}{2}} \odot Z_{\bullet}(\acmfZ^{\odot 2}) \odot (\acmfXhat - \acmfX)} \label{eq:al1_2}
\\&=
2\pi \norm{ (J_{\Dim \Lag} - \acmfZ^{\odot 2} )^{\odot \frac{1}{2}} \odot M \odot Z_{\bullet}(\acmfZ^{\odot 2}) \odot (\acmfXhat - \acmfX)} \label{eq:al1_2.1}
\\&\leq
2\pi \sum_{k=0}^{\infty} \left| \binom{1/2}{k} \right| \norm{ \acmfZ^{\odot 2k} \odot M \odot Z_{\bullet}(\acmfZ^{\odot 2}) \odot (\acmfXhat - \acmfX)} \label{eq:al1_3}
\\&\leq
4 \pi \norm{ M \odot Z_{\bullet}(\acmfZ^{\odot 2}) \odot (\acmfXhat - \acmfX)}
\label{eq:al1_4}.
\end{align}
The relation \eqref{eq:al1_2} is rewritten in terms of the function $Z$ defined in \eqref{eq:def:smallzcapitalZ} and the matrix $J_{\Dim}$ denotes a $\Dim \times \Dim$-matrix with all entries equal to one. 
For the equality \eqref{eq:al1_2.1}, note that $Z_{\bullet}(\acmfZ^{\odot 2})$ is zero on the diagonals. Then, we may include a matrix $M$ being a $0-1$ matrix with the diagonal entries equal to zero. In the next step, we will take advantage of this construction by replacing $Z_{\bullet}(\acmfZ^{\odot 2})$ by a sum of different functions having either zero or non-zero diagonals.
For \eqref{eq:al1_3}, we applied the Taylor series expansion $\sqrt{1-y} = \sum_{k=0}^{\infty} \binom{1/2}{k} y^{k}$. For \eqref{eq:al1_4}, note that $\sum_{k=0}^{\infty} \big| \binom{1/2}{k} \big| = 2$; see p.\ 1206 in \cite{wegkamp2016adaptive}.
Furthermore, since $\acmfZ$ is positive semidefinite, so is $\acmfZ^{\odot 2k}$ due to \eqref{eq:HornJohn4}. An application of \eqref{le:matrixprops1-1} proves \eqref{eq:al1_4} since $\Gamma_{Z,ii}(0) = 1$ for all $ i =1, \dots, \Dim$.

We continue bounding the expression in \eqref{eq:al1_4}. Note that due to the definition of $M$ all diagonal elements are zero. Therefore, with $z$ defined in \eqref{eq:def:smallzcapitalZ},
\begin{align}
&
M \odot Z_{\bullet}(\acmfZ^{\odot 2}) \odot (\acmfXhat - \acmfX)
\nonumber
\\&= 
M \odot
\Big(
(Z_{\bullet}(\acmfZ^{\odot 2}) - Z_{\bullet}(0)) + (Z_{\bullet}(0) - z_{\bullet}(0)) + z_{\bullet}(0) \Big) \odot (\acmfXhat - \acmfX)
\nonumber
\\&= 
M \odot
\Big(
(Z_{\bullet}(\acmfZ^{\odot 2}) - Z_{\bullet}(0)) + (Z_{\bullet}(0) - z_{\bullet}(0)) + z(0) \Big) \odot (\acmfXhat - \acmfX)
\nonumber
\\&= 
\Big( (Z_{\bullet}(\acmfZ^{\odot 2}) - Z_{\bullet}(0)) + M \odot (z(\acmfZ) - z(0)) + M \odot z(0) \Big) \odot (\acmfXhat - \acmfX),
\label{al:awawawawawawawaw2}
\end{align}
where \eqref{al:awawawawawawawaw2} follows since $z(\acmfZ) = Z(0)$. Combining \eqref{eq:al1_4} and \eqref{al:awawawawawawawaw2}, with explanations given below, 
\begin{align}
&
\norm{ (\ell^{\prime})_{\bullet}^{\odot (-1)}(\acmfZ) \odot (\acmfXhat - \acmfX)} \nonumber
\\&\leq
4 \pi
\Big(
\norm{ (Z_{\bullet}(\acmfZ^{\odot 2}) - Z_{\bullet}(0)) \odot (\acmfXhat - \acmfX)}
\notag
\\&\hspace{1cm}+
\norm{ M \odot (z(\acmfZ) - z(0)) \odot (\acmfXhat - \acmfX)} 
\nonumber
\\
&\hspace{2cm}
+
\norm{ M \odot z(0) \odot (\acmfXhat - \acmfX)} \Big)
\label{al:yy2}
\\&\leq
4 \pi
\Bigg(
\frac{1}{(1-\conZ^2)^2}
M_{2}(\conZ,0) + 
2 M_{2}(0,0)
\nonumber 
\\&\hspace{2cm}+
2 M_{2}(\conZ,0) + 2M_{1}(0,0) \Bigg) \norm{ \acmfZ } \norm{ \acmfXhat - \acmfX } 
\label{al:yy3}
\\&\leq
\Big(24 \pi \frac{1}{(1-\conZ^2)^2}
M_{2}(\conZ,0)
+ 8 \pi M_{1}(0,0) \Big) \norm{ \acmfZ } \norm{ \acmfXhat - \acmfX } \nonumber,
\end{align}
where \eqref{al:yy3} follows by Lemmas \ref{le:hzerozero}, \ref{le:hzerosigmaminushzerozero} and \ref{le:hsigmasigmaminushzerosigma}.
\end{proof}

\begin{remark} \label{re:Wegkamp}
We pause here to note that \eqref{eq:al1_2}--\eqref{eq:al1_4} are borrowed from \cite{wegkamp2016adaptive} and have served as an inspiration for the rest of the proofs. 
\cite{wegkamp2016adaptive} consider a semi-parametric elliptical copula model. Similarly to the link function $\ell$, the entries of the copula correlation matrix ($\Sigma$) relate to the entries of the Kendall’s tau matrix ($T$) via the formula $\Sigma = \sin\left(\frac{\pi}{2}T\right)$. Then, $\sin'\left(\frac{\pi}{2}T\right) = \cos\left(\frac{\pi}{2}T\right) = (J_{d}-\sin^{\odot 2}\left(\frac{\pi}{2}T\right))^{\odot \frac{1}{2}} = (J_{d}-\Sigma^{\odot 2})^{\odot \frac{1}{2}}$ and 
Lemma 4.3. in \cite{wegkamp2016adaptive} provides a result similar to Lemma \ref{le:derivofg}, relating consistent estimation of the copula correlation matrix to that of Kendall's tau matrix.
In their scenario, the cosine function, however, does not depend on any unknown parameters which need to be estimated. Furthermore, the function and its derivative are bounded on the whole interval $(-1,1)$.
\end{remark}

\begin{remark}
One can use a different approach than the one pursued in the proof of Lemma \ref{le:derivofg} to deal with $(\ell')^{\odot -1}(\acmfZ) \odot (\acmfXhat - \acmfX)$. An alternative is to apply a second-order Taylor expansion around zero componentwise so that $(\ell')^{\odot -1}(\acmfZ) = (\ell')^{\odot -1}(0) + ((\ell')^{\odot -1})'(0) \odot \acmfZ + \frac{1}{2} ((\ell')^{\odot -1})''(\Sigma) \odot \acmfZ^{\odot 2}$ for some $\Sigma$. The challenges here are to show that $((\ell')^{\odot -1})'(0)$ is positive semidefinite and to bound $((\ell')^{\odot -1})''(\Sigma)$ which involves the third derivative of the link function. It seems like the quantities arising in our approach in the proof of Lemma \ref{le:derivofg} are slightly simpler to handle.
\end{remark}

\begin{lemma} \label{le:hzerozero}
Suppose Assumptions \ref{as:bound for entries}--\ref{ass:finitemoments3}. For $z$ in \eqref{eq:def:smallzcapitalZ} and $M_{1}$ as in \eqref{eq:mumoments2}, 
\begin{equation}
\norm{ M \odot z(0) \odot (\acmfXhat - \acmfX)}
\leq
2 M_{1}(0,0)
\norm{ \acmfXhat - \acmfX }.
\end{equation}
\end{lemma}

\begin{proof}
With further explanations given below, 
\begin{align}
&
\norm{ M \odot z(0) \odot (\acmfXhat - \acmfX)}
\nonumber
\\&\leq
\max_{i =1, \dots, \Dim \Lag} (\sqrt{MM'})_{ii}
\norm{z(0) \odot (\acmfXhat - \acmfX)} 
\label{al:ffffkkkkk1}
\\&\leq
2
\norm{z(0) \odot (\acmfXhat - \acmfX)} \label{al:ffffkkkkk2}
\\&\leq
2 M_{1}(0,0)
\big\| \acmfXhat - \acmfX \big\|_{s}, \label{al:ffffkkkkk3}
\end{align}
where \eqref{al:ffffkkkkk1} follows by \eqref{le:matrixprops1-3}. For \eqref{al:ffffkkkkk2}, note that with $n = \Dim \Lag$,
$\sqrt{MM'} = \sqrt{M_{1}} = M_{2}$ with 
$M_{1,ij} = ( n -1)\mathds{1}_{\{i =j\}} + (n -2)\mathds{1}_{\{i \neq j\}}$ and $M_{2,ij} = 2\frac{ n -1}{n}\mathds{1}_{\{i =j\}} + \frac{ n -2}{n} \mathds{1}_{\{i \neq j\}}$ (which we leave as an exercise), so that $\max_{i =1, \dots, n} (\sqrt{MM'})_{ii} = 2\frac{ n -1}{n} \leq 2$. 
The relation \eqref{al:ffffkkkkk3} follows by positive semidefiniteness of $z(0) $ (proven below) and application of \eqref{le:matrixprops1-1}.

We prove that $z(0) $ is positive semidefnite by expressing it as a vector product. The matrix $z(0)$ is a $\Dim \Lag \times \Dim \Lag$-dimensional block matrix consisting of blocks of the form
\begin{equation}
\begin{aligned}
&
\left( \sum_{n_{0},n_{1} = 0}^{\infty} \exp\left(-\frac{1}{2} (Q_{i,n_{0}}^{2} + Q_{j,n_{1}}^{2})\right) \right)^{\odot -1}_{i,j=1,\dots, \Dim}
\\
&= \left( \sum_{n_{0} = 0}^{\infty} \exp\left(-\frac{1}{2} Q_{i,n_{0}}^{2} \right) \sum_{n_{1} = 0}^{\infty} \exp\left(-\frac{1}{2} Q_{j,n_{1}}^{2} \right) \right)^{\odot -1}_{i,j=1,\dots, \Dim}
= Q Q'
\succcurlyeq 0
\end{aligned}
\end{equation}
with 
\begin{equation}
Q' = 
\left(
\sum_{n = 0}^{\infty} \exp\left(-\frac{1}{2} Q_{1,n}^{2} \right), \dots, \sum_{n = 0}^{\infty} \exp\left(-\frac{1}{2} Q_{\Dim,n}^{2}
\right)
\right)^{\odot -1}.
\end{equation}
Due to the block structure $ z(0) = (QQ')_{r,s = 1, \dots, L}$ is also positive semidefinite. 
\end{proof}

\begin{lemma} \label{le:hzerosigmaminushzerozero}
Suppose Assumptions \ref{as:bound for entries}--\ref{ass:finitemoments3}. For $z$ in \eqref{eq:def:smallzcapitalZ} and $M_{2}$ as in \eqref{eq:mumoments2},
\begin{equation}
\begin{aligned}
&
\norm{ M \odot (z(\acmfZ) - z(0)) \odot (\acmfXhat - \acmfX)}
\\&\leq
2 ( M_{2}(0,0) + M_{2}(\conZ,0) ) 
\| \acmfZ \|_{s} \norm{\acmfXhat - \acmfX}.
\end{aligned}
\end{equation}
\end{lemma}

\begin{proof}
The second-order Taylor approximation of $z(y)$ around zero, applied componentwise in \eqref{al:xyz1} below, gives, for some $| \Sigma | < | \acmfZ |$,
\begin{align}
&
\norm{ M \odot (z(\acmfZ) - z(0)) \odot (\acmfXhat - \acmfX)} \nonumber
\\&\leq
\norm{ M \odot z^{\prime}(0) \odot \acmfZ \odot (\acmfXhat - \acmfX)}
+
\frac{1}{2} \norm{ z_{\bullet}^{\prime\prime}(\Sigma) \odot \acmfZ^{ \odot 2} \odot (\acmfXhat - \acmfX)} \label{al:xyz1}
\\&\leq
2 \norm{ z^{\prime}(0) \odot \acmfZ }
\norm{ \acmfXhat - \acmfX }
+
\frac{1}{2} \norm{ z_{\bullet}^{\prime\prime}(\Sigma) \odot \acmfZ^{ \odot 2} } \norm{ \acmfXhat - \acmfX }
\label{al:xyz2}
\\&\leq
2 M_{2}(0,0) \| \acmfZ \|_{s}
\norm{ \acmfXhat - \acmfX }
+
2 M_{2}(\conZ,0) \| \acmfZ \|_{s} \norm{\acmfXhat - \acmfX},
\label{al:xyz3}
\end{align}
where \eqref{al:xyz2} follows by \eqref{le:matrixprops1-2} and by the same arguments as \eqref{al:ffffkkkkk1} and \eqref{al:ffffkkkkk2} in the proof of Lemma \ref{le:hzerozero} to handle the matrix $M$. The first summand in \eqref{al:xyz3} is a consequence of \eqref{le:matrixprops1-1} since $(-z'(0))$ is positive semidefinite and its diagonals are bounded by $M_{2}(0,0)$; see  Lemma \ref{le:zprimezero}.
The second summand in \eqref{al:xyz3} follows by Lemma \ref{le:zprimeprime} and $\| \acmfZ^{ \odot 2 } \|_{s} \leq \| \acmfZ \|_{s} $ which is satisfied due to \eqref{le:matrixprops1-1}, positive semidefiniteness of $\acmfZ$ and $\Gamma_{Z,ii}(0) = 1$ for all $ i =1, \dots, \Dim$.
\end{proof}

\begin{lemma} \label{le:hsigmasigmaminushzerosigma}
Suppose Assumptions \ref{as:bound for entries}--\ref{ass:finitemoments3}. For $Z$ in \eqref{eq:def:smallzcapitalZ} and $M_{2}$ as in \eqref{eq:mumoments2},
\begin{equation}
\begin{aligned}
&
\norm{ (Z_{\bullet}(\acmfZ^{ \odot 2} ) - Z_{\bullet}(0)) \odot (\acmfXhat - \acmfX)}
\\&\leq
\frac{1}{(1-\conZ^2)^2}
M_{2}(\conZ,0)
\norm{ \acmfZ }
\norm{ \acmfXhat - \acmfX }.
\end{aligned}
\end{equation}
\end{lemma}

\begin{proof}
By the mean value theorem there is a $\Sigma$ such that $|\Sigma| < \acmfZ^{\odot 2}$ and
\begin{align}
&
\norm{ (Z_{\bullet}(\acmfZ^{ \odot 2} ) - Z_{\bullet}(0)) \odot (\acmfXhat - \acmfX)}
\nonumber
\\&=
\norm{ Z'_{\bullet}(\Sigma) \odot
\acmfZ^{ \odot 2}
\odot (\acmfXhat - \acmfX)} \nonumber
\\&\leq
\norm{ Z'_{\bullet}(\Sigma) \odot
\acmfZ^{ \odot 2} }
\norm{ \acmfXhat - \acmfX } \label{al:woeikd1}
\\&\leq
\frac{1}{(1-\conZ^2)^2}
M_{2}(\conZ,0)
\norm{ \acmfZ }
\norm{ \acmfXhat - \acmfX }, \label{al:woeikd2}
\end{align}
where \eqref{al:woeikd1} follows by \eqref{le:matrixprops1-2} and \eqref{al:woeikd2} is a consequence of Lemma \ref{le:bounddercapZ} and $\| \acmfZ^{ \odot 2 } \|_{s} \leq \| \acmfZ \|_{s} $ which is satisfied due to \eqref{le:matrixprops1-1}, positive semidefiniteness of $\acmfZ$ and $\Gamma_{Z,ii}(0) = 1$ for all $ i =1, \dots, \Dim$.
\end{proof}

\subsubsection{Case of estimated link function} \label{se:estimated}
The following lemma is the analogue of Lemma \ref{le:derivofg} for an estimated link function.

\begin{lemma} \label{le:derivofgHAT}
Suppose Assumptions \ref{as:bound for entries}--\ref{ass:finitemoments3} and let $M_{2}$ be as in \eqref{eq:mumoments2}. For any $\delta, \varepsilon >0 $, 
\begin{equation}
\begin{aligned}
\Prob\left[
\norm{ (\widehat{\ell}^{\prime})_{\bullet}^{\odot (-1)}(\acmfZ) \odot (\acmfXhat - \acmfX)}
> \delta \right]
&\precsim
\Prob\left[ R(\acmfZ) \| \acmfXhat - \acmfX \|_{s} > \delta \right]
\\&\hspace{1cm}+
\Prob[ \| \widehat{\theta} - \theta \|_{\maxF} > \varepsilon ]
\end{aligned}
\end{equation}
with
\begin{equation} \label{eq:R(Sigma)}
R(\acmfZ) := R(\varepsilon, \acmfZ) = 
\Big( 24 \pi \frac{1}{(1-\conZ^2)^2}
M_{2}(\conZ, \varepsilon)
+ 8 \pi M_{1}(0,\varepsilon) \Big) \| \acmfZ \|_{s}
\end{equation}
and $M_{1}, M_{2}$ are as in \eqref{eq:mumoments2}. 
\end{lemma}

\begin{proof}
The proof is similar to that of Lemma \ref{le:derivofg} by replacing all functions with their estimated counterparts.
In particular, we can follow the proof of Lemma \ref{le:derivofg} up to \eqref{al:yy2}, that is, 
\begin{align}
&
\norm{ (\widehat{\ell}^{\prime})_{\bullet}^{\odot (-1)}(\acmfZ) \odot (\acmfXhat - \acmfX)}
\\&\leq
4 \pi
\Big(
\norm{ (\widehat{Z}_{\bullet}(\acmfZ^{\odot 2}) - \widehat{Z}_{\bullet}(0)) \odot (\acmfXhat - \acmfX)}
+
\norm{ (\widehat{z}_{\bullet}(\acmfZ) - \widehat{z}_{\bullet}(0)) \odot (\acmfXhat - \acmfX)} 
\nonumber
\\
&\hspace{1cm}
+
\norm{ M \odot \widehat{z}(0) \odot (\acmfXhat - \acmfX)} \Big). \label{al:HATyy2}
\end{align}
In contrast to Lemma \ref{le:derivofg}, the functions in \eqref{al:HATyy2} are random and one needs to control the error made by estimating the CDF parameters $\theta_{i}$.
Continuing with \eqref{al:HATyy2}, and with further explanations given below, 
\begin{align}
&\Prob\left[
\norm{ (\widehat{\ell}^{\prime})_{\bullet}^{\odot (-1)}(\acmfZ) \odot (\acmfXhat - \acmfX)}
> \delta \right]
\nonumber
\\&\leq
\Prob \bigg[
\bigg\{
4 \pi
\Big(
\norm{ (\widehat{Z}_{\bullet}(\acmfZ^{\odot 2}) - \widehat{Z}_{\bullet}(0)) \odot (\acmfXhat - \acmfX)}
\nonumber
\\
&\hspace{1cm}
+
\norm{ (\widehat{z}_{\bullet}(\acmfZ) - \widehat{z}_{\bullet}(0)) \odot (\acmfXhat - \acmfX)} 
\nonumber
\\
&\hspace{2cm}
+
\norm{ M \odot \widehat{z}(0) \odot (\acmfXhat - \acmfX)} \Big)
> \delta
\bigg\}
\nonumber
\\
&\hspace{3cm}
\cap
\Big(\{ \| \widehat{\theta} - \theta \|_{\maxF} \leq \varepsilon \} \cup \{ \| \widehat{\theta} - \theta \|_{\maxF} > \varepsilon \} \Big)
\bigg] 
\label{al:HATyy4}
\\&\leq
\Prob \bigg[
\bigg\{
4 \pi
\Big(
\norm{ (\widehat{Z}_{\bullet}(\acmfZ^{\odot 2}) - \widehat{Z}_{\bullet}(0)) \odot (\acmfXhat - \acmfX)}
\nonumber
\\
&\hspace{1cm}
+
\norm{ M \odot (\widehat{z}_{\bullet}(\acmfZ) - \widehat{z}_{\bullet}(0)) \odot (\acmfXhat - \acmfX)} 
\nonumber
\\
&\hspace{2cm}
+
\norm{ M \odot \widehat{z}(0) \odot (\acmfXhat - \acmfX)} \Big)
> \delta
\bigg\}
\cap
\{ \| \widehat{\theta} - \theta \|_{\maxF} \leq \varepsilon \}
\bigg]
\nonumber
\\
&\hspace{1cm}
+
\Prob[ \| \widehat{\theta} - \theta \|_{\maxF} > \varepsilon ]
\label{al:HATyy5}
\\&\precsim
\Prob\left[ 
\Big(24 \pi \frac{1}{(1-\conZ^2)^2}
M_{2}(\conZ,\varepsilon)
+ 8 \pi M_{1}(0,\varepsilon) \Big)
\| \acmfZ \|_{s} \| \acmfXhat - \acmfX \| > \delta \right]
\nonumber
\\
&\hspace{1cm}
+
\Prob[ \| \widehat{\theta} - \theta \|_{\maxF} > \varepsilon ]. \label{al:HATyy6}
\end{align}
The inequality \eqref{al:HATyy4} follows from \eqref{al:HATyy2} and we intersect with the event $\{ \| \widehat{\theta} - \theta \|_{\maxF} \leq \varepsilon \} $ in \eqref{al:HATyy4} to control the estimation error we made by using $\widehat{\theta}_{i}$; see also Remark \ref{re:intersectingwithevent}.
The bound \eqref{al:HATyy5} follows as in \eqref{al:yy2}. 
The three summands in \eqref{al:HATyy5} can then be handled as in the proof of Lemme \ref{le:derivofg} but through probabilistic versions of Lemmas \ref{le:hzerozero}, \ref{le:hzerosigmaminushzerozero} and \ref{le:hsigmasigmaminushzerosigma}, 
and using the fact that $\widehat{\theta}_{i}$ is in an $\varepsilon$-region of the true $\theta_{i}$.
\end{proof}

\begin{remark} \label{re:intersectingwithevent}
The step of intersecting with the event $\{ \| \widehat{\theta} - \theta \|_{\maxF} \leq \varepsilon \} $ in \eqref{al:HATyy4} is crucial in order to control how much the estimated CDF parameters deviate from the true model parameters. The function $Z$ in \eqref{eq:def:smallzcapitalZ} is not necessarily bounded for all possible values of $\theta$. For this reason, we can only ensure that $M_{1}(\conZ,\varepsilon), M_{2}(\conZ,\varepsilon)$ are finite for small $\varepsilon$. 
A similar approach was pursued in \cite{baek2021thresholding}. \cite{baek2021thresholding} deal with high-dimensional spectral density estimation under long-range dependence. In order to consistently estimate the spectral density matrix under long-range dependence, one needed to control the memory parameter matrix as well; see Proposition 3.5 in \cite{baek2021thresholding} and its proof.
\end{remark}

For completeness, we also state the result analogous to Lemma \ref{le:derivofgHAT} in terms of $\ell( \acmfZ)$ instead of $\acmfX$. We omit the proof since it is similar to the proof of Lemma \ref{le:derivofgHAT}.
\begin{lemma} \label{le:derivofgHATtimesellhatell}
Suppose Assumptions \ref{as:bound for entries}--\ref{ass:finitemoments3} and let $M_{2}$ be as in \eqref{eq:mumoments2}. For any $\delta, \varepsilon >0 $, 
\begin{equation}
\begin{aligned}
&
\Prob\left[
\norm{ (\widehat{\ell}^{\prime})_{\bullet}^{\odot (-1)}(\acmfZ) \odot (\widehat{\ell}_{\bullet}(\acmfZ) - \ell_{\bullet}( \acmfZ))}
> \delta \right]
\\
&\hspace{1cm}
\precsim
\Prob\left[ R(\acmfZ) \norm{ \widehat{\ell}_{\bullet}(\acmfZ) - \ell_{\bullet}( \acmfZ) } > \delta \right]
+
\Prob[ \| \widehat{\theta} - \theta \|_{\maxF} > \varepsilon ]
\end{aligned}
\end{equation}
with
$R(\acmfZ)$ as in \eqref{eq:R(Sigma)}.
\end{lemma}

\subsection{Link function and its derivatives} \label{se:link_function}
In this section, we provide results for the link function $\ell$ and its derivatives.

\begin{lemma} \label{le:ellhatSigmaellSigma}
Suppose Assumptions \ref{as:bound for entries}--\ref{ass:finitemoments3}. For any $\delta, \varepsilon > 0$,
\begin{equation*}
\Prob\left[
\norm{\widehat{\ell}_{\bullet}(\acmfZ) - \ell_{\bullet}(\acmfZ)} > S(\acmfZ) \delta \right]
\precsim
\Prob[
\| \widehat{\theta} - \theta \|_{\maxF} > \delta \wedge \varepsilon ]
\end{equation*}
with 
\begin{equation} \label{eq:capitalSofGamma}
S(\acmfZ) := S(\varepsilon, \acmfZ) = 4 \frac{18}{(1-\conZ^2)^{\frac{7}{2}}}
M(\conZ,\varepsilon) \mu(\conZ,\varepsilon) \| \acmfZ \|_{s}
\end{equation}
and $M, \mu$ are as in \eqref{eq:mumoments}. 
\end{lemma}

\begin{proof}
We have
\begin{align}
&\norm{\widehat{\ell}_{\bullet}(\acmfZ) - \ell_{\bullet}(\acmfZ)} \nonumber
\\&\leq
\norm{\widehat{\ell}(0) - \ell(0) + (\widehat{\ell}'_{\bullet}(\acmfZ) - \ell'_{\bullet}(\acmfZ)) \odot \acmfZ}
+
\norm{(\widehat{\ell}''_{\bullet}(\Sigma_{1}) - \ell''_{\bullet}(\Sigma_{1})) \odot \acmfZ^{ \odot 2}} \label{al:xxxxx0}
\\&=
\norm{(\widehat{\ell}'_{\bullet}(\acmfZ) - \ell'_{\bullet}(\acmfZ)) \odot \acmfZ}
+
\norm{(\widehat{\ell}''_{\bullet}(\Sigma_{1}) - \ell''_{\bullet}(\Sigma_{1})) \odot \acmfZ ^{\odot 2}} \label{al:xxxxx1}
\\&\leq
\norm{(\widehat{\ell}'_{\bullet}(\acmfZ) - \ell'_{\bullet}(\acmfZ) - (\widehat{\ell}'_{\bullet}(0) - \ell'_{\bullet}(0) )) \odot \acmfZ} 
+
2 \norm{(\widehat{\ell}'(0) - \ell'(0)) \odot \acmfZ} \nonumber
\\&\hspace{1cm}+
\norm{(\widehat{\ell}''_{\bullet}(\Sigma_{1}) - \ell''_{\bullet}(\Sigma_{1})) \odot \acmfZ ^{\odot 2}} \label{al:xxxxx3}
\\&\leq
\norm{(\widehat{\ell}''_{\bullet}(\Sigma_{2}) - \ell''_{\bullet}(\Sigma_{2}) ) \odot \acmfZ^{\odot 2}} 
+
2 \norm{(\widehat{\ell}'(0) - \ell'(0)) \odot \acmfZ}
\notag
\\&\hspace{1cm}+
\norm{(\widehat{\ell}''_{\bullet}(\Sigma_{1}) - \ell''_{\bullet}(\Sigma_{1})) \odot \acmfZ ^{\odot 2}}. \label{al:xxxxx4}
\end{align}
The bound \eqref{al:xxxxx0} follows for $\Sigma_{1}$, such that $| \Sigma_{1} | < | \acmfZ |$ by applying the second-order Taylor expansion to the function $x \mapsto \widehat{\ell}(x) - \ell(x)$ around $x=0$.
The equality \eqref{al:xxxxx1} follows since $\widehat{\ell}(0) = \ell(0) = 0$ due to \eqref{eq:function_ell}. 
By subtracting and adding the function $(\widehat{\ell}'_{\bullet}(0) - \ell'_{\bullet}(0) ) \odot \acmfZ$ and subsequent application of triangle inequality, we get 
\eqref{al:xxxxx3}. The second summand of \eqref{al:xxxxx3} is explained at the end of this proof.
The bound \eqref{al:xxxxx4} results from the mean value theorem for some $\Sigma_{2}$ such that $| \Sigma_{2} | < | \acmfZ |$.
It follows from \eqref{al:xxxxx4} that
\begin{align}
&
\Prob\left[ \norm{\widehat{\ell}_{\bullet}(\acmfZ) - \ell_{\bullet}(\acmfZ)} > S(\acmfZ) \delta \right]
\nonumber
\\&\leq
\Prob\left[ 2 \norm{(\widehat{\ell}'(0) - \ell'(0)) \odot \acmfZ}
> S(\acmfZ) \frac{\delta}{2} \right]
\notag
\\&\hspace{1cm}+
\Prob\left[ 2
\norm{(\widehat{\ell}''_{\bullet}(\Sigma) - \ell''_{\bullet}(\Sigma)) \odot \acmfZ ^{\odot 2}}
> S(\acmfZ) \frac{\delta}{2} \right]
\nonumber
\\&\leq
\Prob\left[ \norm{(\widehat{\ell}'(0) - \ell'(0)) \odot \acmfZ}
> s(\acmfZ) \delta \right]
\notag
\\&\hspace{1cm}+
\Prob\left[ 
\norm{(\widehat{\ell}''_{\bullet}(\Sigma) - \ell''_{\bullet}(\Sigma)) \odot \acmfZ ^{\odot 2}}
> S(\acmfZ) \frac{\delta}{4} \right]
\label{al:hhhhhhhhh1.1}
\\ & \leq
2 \Prob[
\| \widehat{\theta} - \theta \|_{\maxF} > \delta \wedge \varepsilon ],
\label{al:hhhhhhhhh2}
\end{align}
where \eqref{al:hhhhhhhhh1.1} follows since $S(\acmfZ) \geq s(\acmfZ)$ with $s(\acmfZ)$ in \eqref{eq:smallsGamma} and the summands in \eqref{al:hhhhhhhhh1.1} are respectively bounded through Lemmas \ref{le:hhathcsquarezero} and \ref{le:secderivlink} to get \eqref{al:hhhhhhhhh2}.

Regarding the second summand in \eqref{al:xxxxx3}, we can always write $\| \widehat{\ell}'_{\bullet}(0) - \ell'_{\bullet}(0) \|_{s}$ with diagonals set to zero as $ \| M \odot (\widehat{\ell}'(0) - \ell'(0)) \|_{s}$ with $M$ being a $0-1$ matrix with the diagonal entries equal to zero. 
Then,
\begin{align}
\norm{(\widehat{\ell}'_{\bullet}(0) - \ell'_{\bullet}(0)) \odot \acmfZ}
&=
\norm{M \odot (\widehat{\ell}'(0) - \ell'(0)) \odot \acmfZ} \nonumber
\\&\leq
\max_{i =1, \dots, \Dim \Lag} (\sqrt{MM^{*}})_{ii}
\norm{(\widehat{\ell}'(0) - \ell'(0)) \odot \acmfZ} \label{al:ffff1}
\\&\leq
2
\norm{(\widehat{\ell}'(0) - \ell'(0)) \odot \acmfZ} \label{al:ffff2},
\end{align}
where \eqref{al:ffff1} follows by \eqref{le:matrixprops1-3}. For \eqref{al:ffff2}, note that with $n = \Dim \Lag$,
$\sqrt{MM^{*}} = \sqrt{M_{1}} = M_{2}$ with 
$M_{1,ij} = ( n -1)\mathds{1}_{\{i =j\}} + (n -2)\mathds{1}_{\{i \neq j\}}$ and $M_{2,ij} = 2\frac{ n -1}{n}\mathds{1}_{\{i =j\}} + \frac{ n -2}{n} \mathds{1}_{\{i \neq j\}}$, so that 
$\max_{i =1, \dots, n} (\sqrt{MM^{*}})_{ii} = 2\frac{ n -1}{n} \leq 2$.
\end{proof}

\begin{lemma} \label{le:hhathcsquarezero}
Suppose Assumptions \ref{as:bound for entries}--\ref{ass:finitemoments3}. For $\delta, \varepsilon >0$,
\begin{equation*}
\Prob\left[
\norm{ (\widehat{\ell}'(0) - \ell'(0)) \odot \acmfZ }
> s(\acmfZ) \delta \right]
\precsim
\Prob[
\| \widehat{\theta} - \theta \|_{\maxF} > \delta \wedge \varepsilon ]
\end{equation*}
with 
\begin{equation} \label{eq:smallsGamma}
s(\acmfZ) := s(\varepsilon, \acmfZ) = 4 \max_{i =1, \dots, \Dim} \sup_{\theta_{i} \in \Theta(\varepsilon)} m_{i}^{(0)}(1) \max_{j=1,\dots,\Dim} \sup_{\theta_{j} \in \Theta(\varepsilon)} \mu_{j}^{(1)}(1) \| \acmfZ \|_{s}.
\end{equation}
\end{lemma}

\begin{proof}
Note that
\begin{equation*}
\ell'(0) 
= 
\left(
\frac{1}{2\pi} \sum_{n_{0},n_{1} = 0}^{\infty}
\exp\left(-\frac{1}{2} (Q_{i,n_{0}}^{2} + Q_{j,n_{1}}^{2}) \right) 
\right)_{i,j=1,\dots, \Dim}
=
QQ'
 \succcurlyeq 0,
\end{equation*}
where $Q' = ( m_{1}^{(0)}(1), \dots, m_{\Dim}^{(0)}(1))$. We write $(QQ')_{r,s = 1, \dots, \Lag}$ for a $\Dim \Lag \times \Dim \Lag$ block matrix, where each $\Dim \times \Dim$ block is the same matrix $QQ'$. Similar quantities can be defined for $\widehat{\ell}'(\cdot)$ in terms of $\widehat{Q}$.
For some $\widetilde{\theta}_{j}$ with $| \widetilde{\theta}_{j} - \theta_{j} | < | \widehat{\theta}_{j} - \theta_{j} | $, we further introduce $R' = ( R_{1}, \dots, R_{\Dim})$ with
\begin{equation*}
\begin{gathered}
R_{j}
=
\frac{1}{\sqrt{2\pi}} 
\sum_{n = 0}^{\infty}
\exp\left(-\frac{1}{2} \widetilde{Q}_{j,n}^{2} \right) (- \widetilde{Q}_{j,n}) 
\langle
\evaluat*{ \nabla Q_{n}(\theta_{j}) }{\widetilde{\theta}_{j}},
(\widehat{\theta}_{j} - \theta_{j})
\rangle,
\end{gathered}
\end{equation*}
where $\widetilde{Q}_{j,n} = Q_{n}(\widetilde{\theta}_{j})$ and $Q_{i,n} = Q_{n}(\theta_{i})$.
Using Lemma \ref{le:derivativethetafinite}, $R_{j}$ can be bounded as
\begin{align} \label{gath:boundonR}
|R_{j}|
&\leq 
\frac{1}{\sqrt{2\pi}} 
\sum_{n = 0}^{\infty}
\exp\left(-\frac{1}{2} \widetilde{Q}_{j,n}^{2} \right) | \widetilde{Q}_{j,n} |
\| \evaluat*{\nabla Q_{n}(\theta_{j}) }{\widetilde{\theta}_{j}} \|
\| \widehat{\theta}_{j} - \theta_{j} \|_{\max}
\notag
\\&\leq \widetilde{\mu}_{j}^{(1)}(1)
\| \widehat{\theta} - \theta \|_{\maxF},
\end{align}
where $\widetilde{\mu}_{j}^{(1)}(1)$ is defined as $\mu_{j}^{(1)}(1)$ in \eqref{eq:momentlikemkderiv} but $Q_{i,n}$ replaced with $\widetilde{Q}_{j,n} $.
Then, with explanations given below,
\begin{align}
&
\norm{ (\widehat{\ell}'(0) - \ell'(0)) \odot \acmfZ } 
\nonumber
\\&=
\norm{ (\widehat{Q}\widehat{Q}' - QQ' )_{r,s = 1, \dots, \Lag}    \odot \acmfZ} \nonumber
\\ &\leq
\norm{ (\widehat{Q}\widehat{Q}' - \widehat{Q}Q' )_{r,s = 1, \dots, \Lag}    \odot \acmfZ} 
+
\norm{ (\widehat{Q}Q' - QQ' )_{r,s = 1, \dots, \Lag}    \odot \acmfZ} 
\nonumber
\\ &=
\Big\| (\widehat{Q} R' )_{r,s = 1, \dots, \Lag}   \odot \acmfZ \Big\|_{s}
+
\Big\| ( R Q'  )_{r,s = 1, \dots, \Lag}  \odot \acmfZ \Big\|_{s}
\label{al:1mvtQR}
\\ &=
\Big\| ( [ \widehat{Q} : \cdots: \widehat{Q} ] \odot [R : \cdots : R ]' )_{r,s = 1, \dots, \Lag}   \odot \acmfZ \Big\|_{s}
\nonumber
\\&\hspace{1cm}
+
\Big\| ( [R : \cdots : R] \odot [ Q : \cdots: Q ]' )_{r,s = 1, \dots, \Lag} \odot \acmfZ \Big\|_{s}
\label{al:2OQhadamard}
\\ & \leq
\max_{i =1, \dots, \Dim} | \widehat{Q}_{i} | \max_{i =1, \dots, \Dim} |R_{i}| \| \acmfZ \|_{s}
+
\max_{i =1, \dots, \Dim} | Q_{i} | \max_{i =1, \dots, \Dim} |R_{i}| \| \acmfZ \|_{s}
\label{al:3hadamarddiag}
\\ & \leq
\max_{i =1, \dots, \Dim} \widehat{m}_{i}^{(0)}(1) \max_{j=1,\dots,\Dim} \widetilde{\mu}_{j}^{(1)}(1) \| \widehat{\theta} - \theta \|_{\maxF} \| \acmfZ \|_{s} \nonumber
\\&\hspace{1cm}+
\max_{i =1, \dots, \Dim} m_{i}^{(0)}(1) \max_{j=1,\dots,\Dim} \widetilde{\mu}_{j}^{(1)}(1) \| \widehat{\theta} - \theta \|_{\maxF} \| \acmfZ \|_{s}.
\label{al:4ineqabove}
\end{align}
The equality \eqref{al:1mvtQR} results from componentwise application of the mean value theorem. The relation \eqref{al:2OQhadamard} follows by noting that for two vectors $a, b \in \RR^{\Dim}$ one can write $ab' = [a: \dots : a] \odot [b: \dots : b]'$.
The inequality \eqref{al:3hadamarddiag} follows by \eqref{le:matrixprops2-2} in Lemma \ref{le:matrixprops2}.
The last inequality \eqref{al:4ineqabove} is due to \eqref{gath:boundonR}. 

It remains to get high probability bounds on the two summands in \eqref{al:4ineqabove}. We have
\begin{align}
&
\Prob[\max_{i =1, \dots, \Dim} \widehat{m}_{i}^{(0)}(1) \max_{j=1,\dots,\Dim} \widetilde{\mu}_{j}^{(1)}(1) \| \widehat{\theta} - \theta \|_{\maxF} \| \acmfZ \|_{s} > s(\acmfZ) \delta/2]
\nonumber
\\ & =
\Prob \bigg[
\bigg\{ \max_{i =1, \dots, \Dim} \widehat{m}_{i}^{(0)}(1) \max_{j=1,\dots,\Dim} \widetilde{\mu}_{j}^{(1)}(1) \| \widehat{\theta} - \theta \|_{\maxF} \| \acmfZ \|_{s} > s(\acmfZ) \delta/2 \bigg\}
\nonumber
\\ & \hspace{1cm} 
\cap
\Big(\{ \| \widehat{\theta} - \theta \|_{\maxF} \leq \varepsilon \} \cup \{ \| \widehat{\theta} - \theta \|_{\maxF} > \varepsilon \} \Big)
\bigg] \nonumber
\\ & \leq
\Prob \bigg[
\bigg\{ \max_{i =1, \dots, \Dim} \widehat{m}_{i}^{(0)}(1) \max_{j=1,\dots,\Dim} \widetilde{\mu}_{j}^{(1)}(1) \| \widehat{\theta} - \theta \|_{\maxF} \| \acmfZ \|_{s} > s(\acmfZ) \delta/2 \bigg\} 
\nonumber
\\
&\hspace{1cm}
\cap
\{ \| \widehat{\theta} - \theta \|_{\maxF} \leq \varepsilon \} \bigg]
+
\Prob[ \| \widehat{\theta} - \theta \|_{\maxF} > \varepsilon ]
\nonumber
%
\\ & \leq
\Prob \bigg[
\max_{i =1, \dots, \Dim} \sup_{\theta_{i} \in \Theta(\varepsilon)} m_{i}^{(0)}(1) \max_{j=1,\dots,\Dim} \sup_{\theta_{j} \in \Theta(\varepsilon)} \mu_{j}^{(1)}(1) \| \widehat{\theta} - \theta \|_{\maxF} \| \acmfZ \|_{s} > s(\acmfZ) \delta/2 \bigg]
\nonumber
\\ & \hspace{1cm}+
\Prob[ \| \widehat{\theta} - \theta \|_{\maxF} > \varepsilon ]
\nonumber
\\ & \leq
2
\Prob[
\| \widehat{\theta} - \theta \|_{\maxF} > \delta \wedge \varepsilon ],
\nonumber
\end{align}
due to the definition of $s(\acmfZ)$ in \eqref{eq:smallsGamma}.
Analogously, we can infer that
\begin{align*}
&
\Prob[\max_{i =1, \dots, \Dim} m_{i}^{(0)}(1) \max_{j=1,\dots,\Dim} \widetilde{\mu}_{j}^{(1)}(1) \| \widehat{\theta} - \theta \|_{\maxF} \| \acmfZ \|_{s} > s(\acmfZ) \delta/2]
\\&=
\Prob \bigg[
\left\{ \max_{i =1, \dots, \Dim} m_{i}^{(0)}(1) \max_{j=1,\dots,\Dim} \widetilde{\mu}_{j}^{(1)}(1) \| \widehat{\theta} - \theta \|_{\maxF} \| \acmfZ \|_{s} > s(\acmfZ) \delta/2 \right\} 
\nonumber
\\
&\hspace{1cm}
\cap
\{ \| \widehat{\theta} - \theta \|_{\max} \leq \varepsilon \} \bigg]
+
\Prob[ \| \widehat{\theta} - \theta \|_{\maxF} > \varepsilon ]
\\&\leq
\Prob \bigg[
\max_{i =1, \dots, \Dim} m_{i}^{(0)}(1) \max_{j=1,\dots,\Dim} \sup_{\theta_{j} \in \Theta(\varepsilon)} \mu_{j}^{(1)}(1) \| \widehat{\theta} - \theta \|_{\maxF} \| \acmfZ \|_{s} > s(\acmfZ) \delta/2 \bigg]
\nonumber
\\ & \hspace{1cm}+
\Prob[ \| \widehat{\theta} - \theta \|_{\maxF} > \varepsilon 
]
\\ & \leq
\Prob[
\| \widehat{\theta} - \theta \|_{\maxF} > \delta \wedge \varepsilon ],
\end{align*}
since
$s(\acmfZ) \geq 2\max_{i =1, \dots, \Dim} m_{i}^{(0)}(1) \max_{j=1,\dots,\Dim} \sup_{\theta_{j} \in \Theta(\varepsilon)} \mu_{j}^{(1)}(1) \| \acmfZ \|_{s}$.
\end{proof}

The following lemma provides a high probability bound on the difference between the estimated and true second derivatives of the link function.
\begin{lemma} \label{le:secderivlink}
Suppose Assumptions \ref{as:bound for entries}--\ref{ass:finitemoments3} and let $\Sigma$ be such that $| \Sigma | < | \acmfZ |$. Then, 
\begin{equation*}
\Prob\left[
\norm{
\left(
\widehat{\ell}''_{\bullet}(\Sigma) - \ell''_{\bullet}(\Sigma) \right)  \odot \acmfZ ^{\odot 2} } > S(\acmfZ) \delta \right]
\leq
\Prob[
\| \widehat{ \theta} - \theta \|_{\maxF} > \delta \wedge \varepsilon ]
\end{equation*}
with $S(\acmfZ)$ as in \eqref{eq:capitalSofGamma}.
\end{lemma}

\begin{proof}
With further explanations given below,
\begin{align}
&
\Prob\left[
\norm{
\left(
\widehat{\ell}''(\Sigma) - \ell''(\Sigma) \right)  \odot \acmfZ ^{\odot 2} } > S(\acmfZ) \delta \right]
\nonumber
\\&\leq
\Prob \bigg[
\bigg\{ \frac{18}{(1-\conZ^2)^{\frac{7}{2}}} 
\widetilde{M}(\conZ,0) \widetilde{\mu}(\conZ,0)
\| \widehat{\theta} - \theta \|_{\maxF} \| \acmfZ^{\odot 2} \|_{s}  > S(\acmfZ) \delta \bigg\} 
\nonumber
\\
&\hspace{1cm}
\cap
\{ \| \widehat{\theta} - \theta \|_{\maxF} \leq \varepsilon \} \bigg]
+
\Prob[ \| \widehat{\theta} - \theta \|_{\maxF} > \varepsilon ) 
\label{al:cehbejnnd}
\\ & \leq
\Prob\bigg[
\frac{18}{(1-\conZ^2)^{\frac{7}{2}}} 
M(\conZ,\varepsilon) \mu(\conZ,\varepsilon)
\| \acmfZ \|_{s}
\| \widehat{\theta} - \theta \|_{\maxF} > S(\acmfZ) \delta \bigg]
\nonumber
\\
&\hspace{1cm}
+
\Prob[  \| \widehat{\theta} - \theta \|_{\maxF} > \varepsilon ]
\label{al:cehbejnnd2}
\\ & \leq
2 \Prob[  \| \widehat{\theta} - \theta \|_{\maxF} > \delta \wedge \varepsilon  ],
\nonumber
\end{align}
where \eqref{al:cehbejnnd} follows by Lemma \ref{le:ellprimeprimedifferenceS1S2S3} and \eqref{al:cehbejnnd2} is a consequence of incorporating $\{ \| \widehat{\theta} - \theta \|_{\maxF} \leq \varepsilon \}$ and $\| \acmfZ \|_{s} ^{\odot 2} \leq \| \acmfZ \|_{s}$ due to \eqref{le:matrixprops1-1} and $\bm{\Gamma}_{Z,rr} = 1$ for all $r=1,\dots, \Dim\Lag$.
\end{proof}

\subsection{Diagonal elements} \label{se:diagonal_elements}
The lemma stated in this section concerns the diagonal elements of $\acmfZ$ and its estimators. The arguments follow a strategy very close to the ones used in the proofs of Lemmas \ref{le:main1APP} and \ref{le:main2APP}. However, we work only with the mean value theorem rather than a second-order Taylor approximation.

\begin{lemma} \label{le:diagonalelements}
Suppose Assumptions \ref{ass:finitemoments}--\ref{ass:finitemoments3}. Then, for any $\delta, \varepsilon > 0$, 
\begin{equation}
\begin{aligned}
&
\Prob\left[ \max_{i=1, \dots, \Dim} | \widehat{g}_{ii}(\widehat{\Gamma}_{X,ii}(0)) - g_{ii}(\Gamma_{X,ii}(0)) | > D(\acmfZ) \delta \right]
\\&\precsim
\Prob\left[ \big\| \acmfXhat - \acmfX \big\|_{s} > \delta   \right]
+
\Prob[
\| \widehat{\theta} - \theta \|_{\maxF} > \delta \wedge \varepsilon ]
\end{aligned}
\end{equation}
with
\begin{equation*}
D(\acmfZ) = M_{1}^{\frac{1}{2}} (1/2,\varepsilon) 2\max\{ 3\Delta(\varepsilon), 1 \},
\end{equation*}
where $M_{1}$ and $\Delta$ are as in \eqref{eq:mumoments2} and \eqref{eq:diagonal_d(epsilon)}.
\end{lemma}

\begin{proof}
As discussed in Section \ref{se:QoI}, we assume that the domain of $g$ (and $\widehat{g}$) is naturally extended such that
$\widehat{g}_{ii}( x ) = \widehat{g}_{ii}( \widehat{\ell}_{ii}(1)) = 1 $ for all $ x > \widehat{\ell}_{ii}(1)$. We then have $0 \leq \widehat{g}_{ii}(\widehat{\Gamma}_{X,ii}(0)) = \widehat{\Gamma}_{Z,ii}(0) \leq 1$.

Therefore, it is sufficient to consider the probability
\begin{align}
&\Prob\left[ \max_{i=1, \dots, \Dim} | \widehat{g}_{ii}(\widehat{\Gamma}_{X,ii}(0)) - g_{ii}(\Gamma_{X,ii}(0)) | > D(\acmfZ) \delta \right]
\notag
\\&=
\Prob\left[ \max_{i=1, \dots, \Dim} ( 1 - \widehat{g}_{ii}(\widehat{\Gamma}_{X,ii}(0) ) ) >  
D(\acmfZ) \delta \right].
\label{al:ww1234_0}
\end{align}
Set $\delta^{*} = D(\acmfZ) \delta$, with explanations given below, the probability on the right-hand side of \eqref{al:ww1234_0} can be bounded as
\begin{align}
&
\Prob\left[ \max_{i=1, \dots, \Dim} ( 1 - \widehat{g}_{ii}(\widehat{\Gamma}_{X,ii}(0) )) >  
\delta^{*} \right]
\nonumber
\\&=
\Prob\left[ \bigcup_{i=1, \dots, \Dim} \big\{ \widehat{\ell}_{ii}( 1 - \delta^{*} ) > \widehat{\Gamma}_{X,ii}(0)) \big\} \right]
\nonumber
\\&=
\Prob\left[ \bigcup_{i=1, \dots, \Dim} \big\{ \Gamma_{X,ii}(0) - \widehat{\Gamma}_{X,ii}(0) + \widehat{\ell}_{ii}( 1 ) - \ell_{ii}(1) > \widehat{\ell}_{ii}( 1 ) - \widehat{\ell}_{ii}( 1 - \delta^{*} )   \big\} \right]
\nonumber
\\&=
\Prob\left[ \bigcup_{i=1, \dots, \Dim} \big\{ \Gamma_{X,ii}(0) - \widehat{\Gamma}_{X,ii}(0) + \widehat{\ell}_{ii}( 1 ) - \ell_{ii}(1) > \widehat{\ell}_{ii}'( c_{\delta^{*}} ) \delta^{*} \big\} \right]
\label{al:ww1234_1}
\\ & \leq
\Prob \bigg[
\bigg\{ \max_{i =1, \dots, \Dim} \frac{1}{\widehat{\ell}_{ii}'( c_{\delta^{*}} )} | \Gamma_{X,ii}(0) - \widehat{\Gamma}_{X,ii}(0) + \widehat{\ell}_{ii}( 1 ) - \ell_{ii}(1) | > \delta^{*} \bigg\}
\nonumber
\\ & \hspace{1cm} \cap
\Big(\{ \| \widehat{\theta} - \theta \|_{\maxF} \leq \varepsilon \} \cup \{ \| \widehat{\theta} - \theta \|_{\maxF} > \varepsilon \} \Big)
\bigg]
\label{al:ww1234_2}
\\ & \leq
\Prob \bigg[
M_{1}^{\frac{1}{2}} (1/2,\varepsilon) \max_{i =1, \dots, \Dim} | \Gamma_{X,ii}(0) - \widehat{\Gamma}_{X,ii}(0) 
\nonumber
\\&\hspace{1cm}+
3 \sup_{\theta_{i} \in \Theta(\varepsilon)} \sum_{n = 0}^{\infty} n \left\| \nabla_{\theta_{i}}  C_{n}(\theta_{i}) \right\|_{1} \| \widehat{\theta} - \theta \|_{\maxF} | > \delta^{*} \bigg]
\nonumber
\\&\hspace{2cm}+
\Prob[ \| \widehat{\theta} - \theta \|_{\maxF} > \varepsilon ].
\label{al:ww1234_3}
\\&\precsim
\Prob\left[ \big\| \acmfXhat - \acmfX \big\|_{s} > \delta   \right]
+
\Prob[
\| \widehat{\theta} - \theta \|_{\maxF} > \delta \wedge \varepsilon ].
\label{al:ww1234_4}
\end{align}
Application of the mean value theorem gives \eqref{al:ww1234_1} for some $ c_{\delta^{*}} \in (1-\delta^{*},1)$. The relation \eqref{al:ww1234_2} is a consequence of intersecting with the event $\{ \| \widehat{\theta} - \theta \|_{\maxF} \leq \varepsilon \}$ and its complement. For \eqref{al:ww1234_3}, we derive a lower bound on $\ell^{\prime}_{ii}(u) $ in Lemma \ref{le:se4boundreciprocal}. Lemma \ref{le:ellhatofone} provides a bound on $| \widehat{\ell}_{ii}( 1 ) - \ell_{ii}(1) |$ which is finite under Assumptions \ref{ass:finitemoments}--\ref{ass:finitemoments3}. Finally, with $\delta^{*} = D(\acmfZ) \delta$, we can bound the first probability in \eqref{al:ww1234_3} further to get \eqref{al:ww1234_4}.
\end{proof}

\section{Additional results and their proofs} \label{se:proofs:additional}
Section \ref{se:normproperties} provides results for the mapping \eqref{eq:matrixmapping} and its interplay with the Hadamard product. Section \ref{se:lefiniteseries} states some results and their proofs to ensure that the constants in our main results are finite. Finally, Section \ref{se:prophigherderiv} collects the majority of derivatives of the link function used throughout this work.

\subsection{Hadamard product} \label{se:normproperties}
Our proofs make extensive use of multiple properties of the Hadamard product. Chapter 5 in \cite{horn1991topics} provides a survey on the Hadamard product. For the reader's convenience, we collect the properties used in our proofs here and will refer to those instead of the respective statements in \cite{horn1991topics}. Let $A,B \in \RR^{\Dim \times \Dim}$ with $A = (a_{ij})_{i,j=1,\dots, \Dim}$ and $B = (b_{ij})_{i,j=1,\dots, \Dim}$ be symmetric matrices. Then, the following statements are true.
\begin{enumerate}
\item Theorem 5.5.18 in \cite{horn1991topics}: If $A \succcurlyeq 0$, then
\begin{equation} \label{eq:HornJohn1}
\| A \odot B \| \leq
\max_{i =1, \dots, \Dim} |a_{ii}| \|B\|.
\end{equation}
\item Theorem 5.5.4 in \cite{horn1991topics}:
\begin{equation} \label{eq:HornJohn2}
\| A \odot B \| \leq \| A \| \| B \|.
\end{equation}
\item Chapter 5.2, Problem 3 in \cite{horn1991topics}: If $A,B \succcurlyeq 0$, then
\begin{equation} \label{eq:HornJohn4}
A \odot B \succcurlyeq 0.
\end{equation}
\item Theorem 5.5.19 in \cite{horn1991topics}:
\begin{equation} \label{eq:HornJohn4.1}
\| A \odot B \| \leq  \max_{i =1, \dots, p} (\sqrt{AA'})_{ii}  \| B \|.
\end{equation}
\item Problem 5.6.P42 in \cite{horn2012matrix}: If $|a_{ij}| \leq b_{ij}$ for all $i,j =1, \dots, \Dim$, then
\begin{equation} \label{eq:HornJohn3}
\| A \| \leq \| B \|.
\end{equation}
\end{enumerate}
Note that Problem 5.6.P42 in \cite{horn2012matrix} actually states that if $0 \leq a_{ij} \leq b_{ij}$ for all $i,j =1, \dots, \Dim$, then $\| A \| \leq \| B \|$. However, the proof given in Problem 5.6.P42 in \cite{horn2012matrix} can be easily adapted to our milder assumptions. A proof of the statement as written in \eqref{eq:HornJohn3} can also be found in the proof of Lemma 4.4. in \cite{wegkamp2016adaptive}.

The goal here is to prove \eqref{eq:HornJohn1}, \eqref{eq:HornJohn2}, \eqref{eq:HornJohn4.1}, \eqref{eq:HornJohn3} and one additional property for the norm (but not matrix norm) defined in \eqref{eq:matrixmapping}, that is, 
\begin{align} 
\label{eq:matrixmapping1}
A \mapsto \sup_{v \in \mathcal{K}(2s)} | v' A v|,
\end{align}
where $\mathcal{K}(2s) = \{ v \in \RR^{\Dim} : \| v \| \leq 1, \| v \|_{0} \leq 2s \}$.
The following Lemmas \ref{le:matrixprops1} and \ref{le:matrixprops2} state the same properties for \eqref{eq:matrixmapping1} as those known for the spectral norm.
\begin{lemma} \label{le:matrixprops1}
Let $A = (a_{ij})_{i,j=1,\dots, \Dim} \in \RR^{\Dim \times \Dim}$ and $B = (b_{ij})_{i,j=1,\dots, \Dim} \in \RR^{\Dim \times \Dim}$ be symmetric. Then, 
\begin{align}
\sup_{v \in \mathcal{K}(2s)} | v' ( A \odot B)v|
&\leq
\sup_{v \in \mathcal{K}(2s)} |v' A v| \sup_{v \in \mathcal{K}(2s)} |v' B v|, \label{le:matrixprops1-2}
\\
\sup_{v \in \mathcal{K}(2s)} | v' ( A \odot B)v|
&\leq
\max_{i =1, \dots, \Dim} (\sqrt{AA^{'}})_{ii}  \sup_{v \in \mathcal{K}(2s)} |v' B v|.
\label{le:matrixprops1-3}
\end{align}
If $A$ is also positive semidefinite, then
\begin{align}
\sup_{v \in \mathcal{K}(2s)} | v' ( A \odot B)v|
\leq
\max_{i = 1, \dots, \Dim} |a_{ii}| \sup_{v \in \mathcal{K}(2s)} |v' B v|.\label{le:matrixprops1-1}
\end{align}
\end{lemma}

\begin{proof}[Proof of Lemma \ref{le:matrixprops1}]
Write $D_{v}$ for the diagonal matrix which corresponds to a vector $v \in \RR^{\Dim}$ such that $D_{v} = \diag(v_{1},\dots,v_{\Dim})$. We further introduce the matrix $D_{s}(v)$, $\| v \|_{0} \leq 2s$, which is a $0-1$ matrix with the same sparsity pattern as $D_{v}$. We prove \eqref{le:matrixprops1-1} and \eqref{le:matrixprops1-2} separately and omit the proof of \eqref{le:matrixprops1-3} since it follows analogously by using \eqref{eq:HornJohn4.1}.

\textit{Proof of \eqref{le:matrixprops1-1}:}
With explanations given below,
\begin{align}
&\sup_{v \in \mathcal{K}(2s)} | v' ( A \odot B)v|
=
\sup_{v \in \mathcal{K}(2s)} | \tr ( A D_{v} B D_{v} ) |
\label{al:matrixprops1-1-0}
\\&=
\sup_{v \in \mathcal{K}(2s)} | \tr ( A D_{v} D_{s}(v) B D_{s}(v) D_{v} ) | \notag
\\&=
\sup_{v \in \mathcal{K}(2s)}  | v' ( A \odot D_{s}(v) B D_{s}(v) ) v | \notag
\\&\leq
\sup_{v \in \mathcal{K}(2s)}  \| A \odot D_{s}(v) B D_{s}(v) \| 
\notag
\\&\leq
\max_{i = 1, \dots, \Dim} |a_{ii}| \sup_{v \in \mathcal{K}(2s)} \| D_{s}(v) B D_{s}(v) \| 
\label{al:matrixprops1-1-1}
\\&=
\max_{i = 1, \dots, \Dim} |a_{ii}| \sup_{v \in \mathcal{K}(2s)} \max \{ \lambda_{\max}(D_{s}(v) B D_{s}(v)), -\lambda_{\min}(D_{s}(v) B D_{s}(v)) \} \label{al:matrixprops1-1-2}
\\&=
\max_{i = 1, \dots, \Dim} |a_{ii}| \sup_{v \in \mathcal{K}(2s)} \max \{ \lambda_{\max}(D_{s}(v) B D_{s}(v)), \lambda_{\max}(D_{s}(v) (-B) D_{s}(v)) \} \label{al:matrixprops1-1-3}
\\&\leq
\max_{i = 1, \dots, \Dim} |a_{ii}| \sup_{v \in \mathcal{K}(2s)} |v' B v|, 
\label{al:matrixprops1-1-4}
\end{align}
where \eqref{al:matrixprops1-1-0} is due to Lemma 5.1.5 in \cite{horn1991topics} and \eqref{al:matrixprops1-1-1} follows by \eqref{eq:HornJohn1}. The representations \eqref{al:matrixprops1-1-2} and \eqref{al:matrixprops1-1-3} can be used since $B$ is symmetric. Finally, \eqref{al:matrixprops1-1-4} follows by the min-max theorem for eigenvalues 
\begin{align} \label{eq:lambdamaxtosupv}
\sup_{v \in \mathcal{K}(2s)} \lambda_{\max}(D_{s}(v) B D_{s}(v))
=
\sup_{v \in \mathcal{K}(2s)} \sup_{x : \|x\|=1} x' D_{s}(v) B D_{s}(v) x
\leq
\sup_{v \in \mathcal{K}(2s)} v' B v,
\end{align}
since $\| D_{s}(v) x \| \leq \| x \| = 1$ and $\| D_{s}(v) x \|_{0} \leq \| D_{s}(v) \|_{0} \leq 2s$.

\textit{Proof of \eqref{le:matrixprops1-2}:} Proceeding as for \eqref{le:matrixprops1-1}.
\begin{align}
\sup_{v \in \mathcal{K}(2s)} | v' ( A \odot B)v|
&=
\sup_{v \in \mathcal{K}(2s)} | \tr ( A D_{v} B D_{v} ) | \notag
\\&=
\sup_{v \in \mathcal{K}(2s)} | \tr ( A D_{v} D_{s}(v) B D_{s}(v) D_{v} ) | \notag
\\&=
\sup_{v \in \mathcal{K}(2s)} | \tr ( D_{v}A D_{v} D_{s}(v) B D_{s}(v)  ) | \notag
\\&=
\sup_{v \in \mathcal{K}(2s)} | \tr ( D_{v} D_{s}(v) A D_{s}(v) D_{v} D_{s}(v) B D_{s}(v)  ) | \notag
\\&=
\sup_{v \in \mathcal{K}(2s)}  | v' ( D_{s}(v) A D_{s}(v) \odot D_{s}(v) B D_{s}(v) ) v | \notag
\\&\leq
\sup_{v \in \mathcal{K}(2s)}  \| D_{s}(v) A D_{s}(v) \odot D_{s}(v) B D_{s}(v) \| \notag
\\&\leq
\sup_{v \in \mathcal{K}(2s)}  \| D_{s}(v) A D_{s}(v) \| \| D_{s}(v) B D_{s}(v) \| \label{al:matrixprops1-2-1}
\\&\leq
\sup_{v \in \mathcal{K}(2s)} |v' A v| \sup_{v \in \mathcal{K}(2s)} |v' B v|, \label{al:matrixprops1-2-2}
\end{align}
where \eqref{al:matrixprops1-2-1} follows by \eqref{eq:HornJohn2} and \eqref{al:matrixprops1-2-2} by the same arguments to go from \eqref{al:matrixprops1-1-1} to \eqref{al:matrixprops1-1-4}.
\end{proof}

\begin{lemma} \label{le:matrixprops2}
Let $A, \widetilde{A}, B \in \RR^{\Dim \times \Dim}$ with $A = (a_{ij})_{i,j=1,\dots, \Dim}$, $\widetilde{A} = (\widetilde{a}_{ij})_{i,j=1,\dots, \Dim}$ and $B = (b_{ij})_{i,j=1,\dots, \Dim}$.
\begin{enumerate}
\item[(i)] If $|a_{ij}| \leq b_{ij}$ for all $i,j =1,\dots,\Dim$, then
\begin{align} \label{le:matrixprops2-1}
\sup_{v \in \mathcal{K}(2s)} | v' A v |
\leq
\sup_{v \in \mathcal{K}(2s)} | v' B v |.
\end{align}
\item[(ii)]
If $a_{ij} = a_{j}$, $\widetilde{a}_{ij} = \widetilde{a}_{i}$ for $i,j =1,\dots, \Dim$ and $B$ symmetric, then
\begin{align} \label{le:matrixprops2-2}
\sup_{v \in \mathcal{K}(2s)} | v' ( A \odot \widetilde{A} \odot B)v|
\leq
\max_{j =1,\dots,\Dim} |a_{j}| \max_{i =1,\dots,\Dim} |\widetilde{a}_{i}| \sup_{v \in \mathcal{K}(2s)} |v' B v|. 
\end{align}
\end{enumerate}
\end{lemma}

\begin{proof}[Proof of Lemma \ref{le:matrixprops2}]
We prove the two inequalities \eqref{le:matrixprops2-1} and \eqref{le:matrixprops2-2} separately.

\textit{Proof of \eqref{le:matrixprops2-2}:} The ideas follow the proof of Lemma \ref{le:matrixprops1}. 
With further explanations given below, we have
\begin{align}
&
\sup_{v \in \mathcal{K}(2s)} | v' ( A \odot \widetilde{A} \odot B)v|
\nonumber
\\&\leq
\sup_{v \in \mathcal{K}(2s)}  \| A \odot \widetilde{A} \odot D_{s}(v) B D_{s}(v) \|
\notag
\\&=
\sup_{v \in \mathcal{K}(2s)}  \| \diag(\widetilde{a}_{1},\dots,\widetilde{a}_{p}) D_{s}(v) B D_{s}(v) \diag(a_{1},\dots,a_{p}) \| 
\label{al:matrixprops2-1-1}
\\&\leq
\max_{j =1,\dots,\Dim} |a_{j}| \max_{i =1,\dots,\Dim} |\widetilde{a}_{i}| 
\sup_{v \in \mathcal{K}(2s)} \| D_{s}(v) B D_{s}(v) \| \label{al:matrixprops2-1-2}
\\&\leq
\max_{j =1,\dots,\Dim} |a_{j}| \max_{i =1,\dots,\Dim} |\widetilde{a}_{i}|
\sup_{v \in \mathcal{K}(2s)} |v' B v|. \label{al:matrixprops2-1-3}
\end{align}
Given that $A = (a_{ij})_{i,j=1,\dots, \Dim}$ with $a_{ij} = a_{j}$ for $j =1,\dots, \Dim$, one has $A \odot C = C \diag(a_{1},\dots,a_{\Dim})$ so that $ \| A \odot C \| = \| C \diag(a_{1},\dots,a_{\Dim}) \| \leq 
\max_{j =1,\dots, \Dim} |a_{j}| \| C \|$ due to the submultiplicativity of the spectral norm (and similarly for $\widetilde{A}$), which explains \eqref{al:matrixprops2-1-1} and \eqref{al:matrixprops2-1-2}. The inequality \eqref{al:matrixprops2-1-3} follows from \eqref{al:matrixprops1-1-1} to \eqref{al:matrixprops1-1-4}.

\textit{Proof of \eqref{le:matrixprops2-1}:}
If $ |a_{ij}| \leq b_{ij}$ for all $i,j =1, \dots, \Dim$, then
\begin{align*}
\sup_{v \in \mathcal{K}(2s)} | v' A v |
\leq
\sup_{v \in \mathcal{K}(2s)} \sum_{i,j=1}^{\Dim} |v_{i}| |a_{ij}| |v_{j} |
\leq
\sup_{v \in \mathcal{K}(2s)} \sum_{i,j=1}^{\Dim} |v_{i}| b_{ij} |v_{j} |
\leq
\sup_{v \in \mathcal{K}(2s)} | v' B v |.
\end{align*}
\end{proof}

\subsection{Moment conditions} \label{se:lefiniteseries}
The results in this section ensure that the constants in our main results are finite. 
Lemmas \ref{le:diagonalsderivativethetafinite}, \ref{le:finitesumpdelta} and \ref{le:derivativethetafinite} consider respectively the quantities in \eqref{eq:moment0}, \eqref{eq:momentlikemk} and \eqref{eq:momentlikemkderiv}.

\begin{lemma} \label{le:diagonalsderivativethetafinite}
Suppose Assumptions \ref{ass:finitemoments} and \ref{ass:finitemoments2}. Then, for an open set $S$, 
\begin{equation} \label{eq:jkjinin}
\begin{aligned}
\sup_{\theta_{i} \in S} \Delta_{i} 
&:= \sup_{\theta_{i} \in S} \sum_{n = 0}^{\infty} n \left\| \nabla_{\theta_{i}}  C_{i,n} \right\|_{1}
\\&= \sup_{\theta_{i} \in S} \frac{1}{\sqrt{2\pi}} \sum_{n=0}^{\infty} \exp\left(-\frac{1}{2u} Q^{2}_{i,n} \right) n
\| \nabla_{\theta_{i}} Q_{i,n} \|_{1}
< \infty.
\end{aligned}
\end{equation}
\end{lemma}

\begin{proof}
We have
\begin{align}
&
\sum_{n = 0}^{\infty} n \left\| \nabla_{\theta_{i}}  C_{n}(\theta_{i}) \right\|_{1}
\nonumber
\\&=
\sum_{n = 0}^{\infty} n (\Prob[X_{i,t} > n])^{\frac{1}{2}} (\Prob[X_{i,t} > n])^{-\frac{1}{2}} \left\| \nabla_{\theta_{i}}  C_{n}(\theta_{i}) \right\|_{1}
\nonumber
\\&\leq
(\E|X_{i,t}|^2 )^{\frac{1}{2}}
\sum_{n = 0}^{\infty} (\Prob[X_{i,t} > n])^{-\frac{1}{2}} \left\| \nabla_{\theta_{i}}  C_{n}(\theta_{i}) \right\|_{1}
\label{al:lllllll4}
\\&=
(\E|X_{i,t}|^2 )^{\frac{1}{2}}
\sum_{n = 0}^{\infty} (\Prob[X_{i,t} > n])^{-\frac{1}{2}} 
\sum_{j=1}^{K_{i}}
\left| \frac{\partial}{\partial \theta_{ij}} \Prob[X_{i,t} > n] \right| < \infty,
\label{al:lllllll5}
\end{align}
where we applied Markov's inequality in \eqref{al:lllllll4}. After taking the supremum over all $\theta_{i}$ in an open set $S$ on both sides of \eqref{al:lllllll5}, the expression \eqref{al:lllllll5} is uniformly bounded due to Assumption \ref{ass:finitemoments2}.
For the equality in \eqref{eq:jkjinin}, let $\phi$ denote the Gaussian density and note that
\begin{equation} \label{al:wwwwwwww0}
\nabla_{\theta_{i}} Q_{i,n} = \nabla_{\theta_{i}} \Phi^{-1}(C_{i,n}) = 
\frac{1}{\phi(\Phi^{-1}(C_{i,n})) } \nabla_{\theta_{i}} C_{n}(\theta_{i}) =
\frac{1}{\phi( Q_{i,n} ) } \nabla_{\theta_{i}} C_{n}(\theta_{i}),
\end{equation}
since $Q_{i,n} = \Phi^{-1}(C_{i,n})$ with $C_{i,n} = \Prob[X_{i,t} \leq n] = \sum_{j = 0}^{n} \Prob[X_{i,t} = j]  = \sum_{j = 0}^{n} p_{\theta_{i},t}(j) = C_{n}(\theta_{i})$.
Then, the relation \eqref{al:wwwwwwww0} allows us to write
\begin{equation*} 
\sum_{n = 0}^{\infty} n \left\| \nabla_{\theta_{i}}  C_{i,n} \right\|_{1}
= \frac{1}{\sqrt{2\pi}} \sum_{n=0}^{\infty} \exp\left(-\frac{1}{2u} Q^{2}_{i,n} \right) n
\| \nabla_{\theta_{i}} Q_{i,n} \|_{1}.
\end{equation*}
\end{proof}

The following lemma is similar to Lemma 2.1 in \cite{jia2023latent} and coincides with it for $u = 1$.
\begin{lemma} \label{le:finitesumpdelta}
Suppose $u>0$ and, Assumption \ref{ass:finitemoments} is satisfied for some $p > u$. Then, for an open set $S$ and any $k \in \NN_{0}$,
\begin{equation} \label{eq:finiteseries}
\sup_{\theta_{i} \in S}  m_{i}^{(k)}(u) := 
\sup_{\theta_{i} \in S}  \frac{1}{\sqrt{2\pi}} 
\sum_{n=0}^{\infty} \exp\left(- \frac{1}{2u} Q_{i,n}^2\right) |Q_{i,n}|^k < \infty.
\end{equation}
\end{lemma}

\begin{proof}
We follow the proof of Lemma 2.1 in \cite{jia2023latent}. Note that by Mill’s ratio, we have
\begin{equation} \label{eq:Millsratio}
1-\Phi(x) \sim e^{-\frac{x^2}{2}} \frac{1}{\sqrt{2\pi}x}, \hspace{0.2cm} \text{ as } x \to \infty.
\end{equation}
Then, substituting $x = \Phi^{-1}(y)$ in \eqref{eq:Millsratio} leads $1-y \sim e^{-\frac{\Phi^{-1}(y)^2}{2}} \frac{1}{\sqrt{2\pi}\Phi^{-1}(y)}$, as $y \uparrow 1$.
Taking the logarithm on both sides, we get
\begin{equation} \label{eq:afterlog1}
\log(1-y) \sim -\frac{\Phi^{-1}(y)^2}{2} -\log( \sqrt{2\pi}\Phi^{-1}(y)), \hspace{0.2cm} \text{ as } y \uparrow 1,
\end{equation}
and can infer
\begin{equation} \label{eq:afterlog2}
\sqrt{2} | \log(1-y) |^{\frac{1}{2}} \sim \Phi^{-1}(y), \hspace{0.2cm} \text{ as } y \uparrow 1.
\end{equation}
Finally, applying \eqref{eq:afterlog1} and \eqref{eq:afterlog2} for $y = C_{i,n}$ and substitution into \eqref{eq:finiteseries} with $Q_{i,n} = \Phi^{-1}(C_{i,n})$ show that $m_{i}^{(k)}(u)$ can be bounded (up to a constant) by
\begin{align}
&\sum_{n=0}^{\infty} \exp\left( \frac{1}{2u} 2  \log(1-C_{i,n}) \right) (2 | \log(1-C_{i,n})|)^{\frac{k}{2}} \nonumber
\\&\leq
c \sum_{n=0}^{\infty} (1-C_{i,n})^{\frac{1}{u}} | \log(1-C_{i,n}) |^{\frac{k}{2}} \nonumber
\\&\leq
c \sum_{n=0}^{\infty} (1-C_{i,n})^{\frac{1}{u}-\frac{k}{2} \delta} \left(\frac{1}{\delta}\right)^{\frac{k}{2}} \label{al:ww1}
\\&=
c \sum_{n=0}^{\infty} \Prob[X_{i,t} > n]^{\frac{1}{u}-\frac{k}{2} \delta} \left(\frac{1}{\delta}\right)^{\frac{k}{2}}, \label{al:ww2}
\end{align}
where \eqref{al:ww1} follows since for any $\delta > 0 $ and $x \in (0,1)$, $-\log(x) \leq \frac{x^{-\delta}}{\delta}$, and \eqref{al:ww2} follows since $C_{i,n} = 1 - \Prob[X_{i,t} > n]$. Using Markov's inequality with $\Prob[X_{i,t}>n] = \Prob[X_{i,t}^{p} > n^{p}] \leq \E|X_{i,t}|^p / n^{p}$, we get
\begin{equation} \label{eq:xdxdxdxdxdxdxdxddxd}
\sum_{n=0}^{\infty} \Prob[X_{i,t} > n]^{\frac{1}{u}-\frac{k}{2} \delta} 
\leq
(\E|X_{i,t}|^{p})^{(\frac{1}{u}-\frac{k}{2} \delta)} \sum_{n=0}^{\infty} \frac{1}{n^{p(\frac{1}{u}-\frac{k}{2} \delta)}},
\end{equation}
which converges as long as $p(\frac{1}{u}-\frac{k}{2} \delta)>1$ is satisfied. After taking the supremum of $m_{i}^{(k)}(u)$ over all $\theta_{i}$ in an open set $S$, the expression is uniformly bounded due to \eqref{eq:xdxdxdxdxdxdxdxddxd} and Assumption \ref{ass:finitemoments2}.
\end{proof}

\begin{lemma} \label{le:derivativethetafinite}
Suppose $u \in (0,2)$ and Assumption \ref{ass:finitemoments2}. Then, for an open set $S$ and any $k \in \NN_{0}$, 
\begin{equation*} 
\sup_{\theta_{i} \in S} \mu_{i}^{(k)}(u) := 
\sup_{\theta_{i} \in S} \frac{1}{\sqrt{2\pi}} \sum_{n=0}^{\infty} \exp\left(-\frac{1}{2u} Q^{2}_{i,n} \right) |Q_{i,n}|^k 
\| \nabla_{\theta_{i}} Q_{i,n} \|_{1} < \infty.
\end{equation*}
\end{lemma}

\begin{proof}
Recall $\nabla_{\theta_{i}} Q_{i,n}
= 
\frac{1}{\phi( Q_{i,n} ) } \nabla_{\theta_{i}} C_{n}(\theta_{i})$ in \eqref{al:wwwwwwww0}.
Then, 
\begin{align}
&\sum_{n=0}^{\infty} \exp\left(-\frac{1}{2u} Q^{2}_{i,n} \right) |Q_{i,n}|^k 
\| \nabla_{\theta_{i}} Q_{i,n} \|_{1} \nonumber
\\&=
\sum_{n = 0}^{\infty} \exp\left(-\frac{1}{2u} Q^{2}_{i,n} \right) |Q_{i,n}|^k \frac{1}{\phi( Q_{i,n} ) } 
\| \nabla_{\theta_{i}} C_{n}(\theta_{i}) \|_{1}
\label{al:wwwwwwww1}
\\&=
\sqrt{2 \pi}
\sum_{n = 0}^{\infty} \exp\left(-\frac{1-u}{2u} Q^{2}_{i,n} \right) |Q_{i,n}|^k 
\| \nabla_{\theta_{i}} C_{n}(\theta_{i}) \|_{1},
\label{al:wwwwwwww2}
\end{align}
where \eqref{al:wwwwwwww1} results from substituting \eqref{al:wwwwwwww0}. In order to continue bounding \eqref{al:wwwwwwww2}, we take advantage of several relations derived in the proof of Lemma \ref{le:finitesumpdelta}.
\begin{align}
&
\sum_{n = 0}^{\infty} \exp\left(-\frac{1-u}{2u} Q^{2}_{i,n} \right) |Q_{i,n}|^k 
\| \nabla_{\theta_{i}} C_{n}(\theta_{i}) \|_{1} \nonumber
\\&\leq
c
\sum_{n = 0}^{\infty} (1-C_{n}(\theta_{i}))^{ \frac{1-u}{u} } (2 | \log(1-C_{n}(\theta_{i}))|)^{\frac{k}{2}} 
\| \nabla_{\theta_{i}} C_{n}(\theta_{i}) \|_{1} \label{al:wwwwwwwwww01}
\\&\leq
c
\sum_{n = 0}^{\infty} (1-C_{n}(\theta_{i}))^{\frac{1-u}{u}-\frac{k}{2} \delta} \left(\frac{1}{\delta}\right)^{\frac{k}{2}} 
\| \nabla_{\theta_{i}} C_{n}(\theta_{i}) \|_{1} \label{al:wwwwwwwwww02}
\\&\leq
c
\sum_{n = 0}^{\infty} (\Prob[X_{i,t}>n])^{-\frac{1}{2}} 
\| \nabla_{\theta_{i}} C_{n}(\theta_{i}) \|_{1} \label{al:wwwwwwwwww02.1}
\\&\leq
c
\sum_{n = 0}^{\infty} (\Prob[X_{i,t} > n])^{-\frac{1}{2}} 
\sum_{j=1}^{K_{i}}
\left| \frac{\partial}{\partial \theta_{ij}} \Prob[X_{i,t} > n] \right|
< \infty,
\label{al:wwwwwwwwww02.2}
\end{align}
where \eqref{al:wwwwwwwwww01} follows by \eqref{eq:afterlog1} and \eqref{eq:afterlog2}. Since $-\log(x) \leq \frac{x^{-\delta}}{\delta}$ for any $\delta > 0 $ and $x \in (0,1)$ the inequality \eqref{al:wwwwwwwwww02} follows. In \eqref{al:wwwwwwwwww02.1}, we choose $\delta$ such that $\frac{1-u}{u}-\frac{k}{2} \delta \geq -\frac{1}{2}$ for $u \in (0,2)$. After taking the supremum over all $\theta_{i}$ in an open set $S$ on both sides of \eqref{al:wwwwwwwwww02.2}, the expression in \eqref{al:wwwwwwwwww02.2} is uniformly bounded due to Assumption \ref{ass:finitemoments2}.
\end{proof}

\subsection{Properties of higher order derivatives of the link function} \label{se:prophigherderiv}
This section collects and derives all bounds and derivatives of the link function as needed throughout this paper. 

We start with the first two derivatives of $\ell_{ij}$ and its inverse. For shortness sake, we introduce the following notation:
\begin{equation*} 
\capitalG(u) := Q_{i,n_{0}}^{2} + Q_{j,n_{1}}^{2} - 2uQ_{i,n_{0}}Q_{j,n_{1}},
\hspace{0.2cm}
\smallg := Q_{i,n_{0}}Q_{j,n_{1}}
\end{equation*}
and
\begin{equation} \label{eq:S1S2S3}
\begin{gathered}
S_{1}(u) := \sum_{n_{0},n_{1} = 0}^{\infty} \exp\left(-\frac{1}{2 (1-u^2)} \capitalG(u) \right),
\\
S_{2}^{G}(u) := \sum_{n_{0},n_{1} = 0}^{\infty} \exp\left(-\frac{1}{2 (1-u^2)} \capitalG(u) \right) \capitalG(u), \\
S_{3}^{g}(u) := \sum_{n_{0},n_{1} = 0}^{\infty} \exp\left(-\frac{1}{2 (1-u^2)} \capitalG(u) \right) \smallg.
\end{gathered}
\end{equation}
Then, the first derivative of the link function in Proposition \ref{prop:prop2.1} can be written as
\begin{equation} \label{eq:ellprimewithSone}
\begin{aligned}
\ell^{\prime}_{ij}(u) 
&= 
\frac{1}{2\pi \sqrt{1-u^2}} \sum_{n_{0},n_{1} = 0}^{\infty} \exp\left(-\frac{1}{2 (1-u^2)} (Q_{i,n_{0}}^{2} + Q_{j,n_{1}}^{2} - 2uQ_{i,n_{0}}Q_{j,n_{1}})\right)
\\&=
\frac{1}{2\pi \sqrt{1-u^2}} S_{1}(u).
\end{aligned}
\end{equation}
In order to derive the higher order derivatives of $\ell_{ij}$, we first derive the first derivatives of the quantities in \eqref{eq:S1S2S3}:
\begin{align}
\frac{\partial}{\partial u} S_{1}(u) 
&= 
\frac{-u}{(1-u^2)^{2}} S_{2}^{G}(u) + \frac{1}{1-u^2} S_{3}^{g}(u), \label{al:deriv1}
\\
\frac{\partial}{\partial u} S_{2}^{G}(u) 
&= 
-2 S_{3}^{g}(u)
\nonumber
\\&\hspace{1cm}-
\frac{u}{(1-u^{2})^{2}}
\sum_{n_{0},n_{1} = 0}^{\infty} \exp\left(-\frac{1}{2 (1-u^2)} \capitalG(u) \right) (\capitalG(u))^2 \nonumber
\\&\hspace{1cm}+
\frac{1}{1-u^{2}}
\sum_{n_{0},n_{1} = 0}^{\infty} \exp\left(-\frac{1}{2 (1-u^2)} \capitalG(u) \right) \capitalG(u) \smallg,
\label{al:deriv2}
\\
\frac{\partial}{\partial u} S_{3}^{g}(u)
&=
-\frac{u}{(1-u^{2})^{2}}
\sum_{n_{0},n_{1} = 0}^{\infty} \exp\left(-\frac{1}{2 (1-u^2)} \capitalG(u) \right) \capitalG(u) \smallg \nonumber
\\&\hspace{1cm}+
\frac{1}{1-u^{2}}
\sum_{n_{0},n_{1} = 0}^{\infty} \exp\left(-\frac{1}{2 (1-u^2)} \capitalG(u) \right) (\smallg)^{2}.
 \label{al:deriv3}
\end{align}
Using the introduced notation in \eqref{eq:S1S2S3} and the corresponding derivative \eqref{al:deriv1}, we can then write the second derivative of $\ell_{ij}$ as
\begin{align} \label{eq:second_derivative_of_ell}
\ell''_{ij}(u)
&=
\frac{\partial}{\partial u} \frac{1}{2\pi \sqrt{1-u^2}} S_{1}(u) \nonumber
\\&= 
\frac{u}{2 \pi (1-u^2)^{\frac{3}{2}}} S_{1}(u) - 
\frac{u}{2 \pi (1-u^2)^{\frac{5}{2}}} S_{2}^{G}(u) +
\frac{1}{2 \pi (1-u^2)^{\frac{3}{2}}} S_{3}^{g}(u).
\end{align}
The derived quantities will help expressing the first and second derivatives of the inverse link function $g = \ell^{-1}$. For this, note first that
\begin{align}
(\ell^{-1})' (x)
&=
\frac{1}{ \ell' (\ell^{-1}(x)) }, \label{eq:firstderivellinv}
\\
(\ell^{-1})'' (x) 
&= \frac{\partial}{\partial x} \frac{1}{ \ell' (\ell^{-1}(x)) } 
= -\frac{\ell''(\ell^{-1}(x))}{(\ell'(\ell^{-1}(x)))^{3}}
=: f(\ell^{-1}(x))
\hspace{0.2cm}
\text{ with }
\hspace{0.2cm}
f(u) = -\frac{\ell''(u)}{(\ell'(u))^{3}}.
\label{eq:secndderivellinv}
\end{align}

The following Sections \ref{sec:derivatives1}, \ref{sec:derivatives2} and \ref{sec:derivatives3} consider respectively the derivatives used in Sections \ref{se:inverted_link_function}, \ref{se:oneoverellprime} and \ref{se:link_function}.
We aim to express all bounds in terms of $M(\bm{c},\varepsilon)$, $\mu(\bm{c},\varepsilon)$, $M_{1}(\bm{c},\varepsilon)$ and $M_{2}(\bm{c},\varepsilon)$ in \eqref{eq:mumoments} and \eqref{eq:mumoments2}.
All bounds are simple consequences of finding upper and lower bounds on $\capitalG(u)$. Those bounds can be derived easily through $a^2 +b^2 -2uab \leq a^2 +b^2 + 2 |ab| |u| \leq (1+\bm{c})(a^{2} + b^{2})$ and $a^2 +b^2 -2uab \geq a^2 +b^2 - 2 |ab| |u| \geq (1- \bm{c})(a^{2} + b^{2})$ for $0 < \bm{c} <1$ and $|u | < \bm{c}$.

\subsubsection{Derivatives I} \label{sec:derivatives1}
In this section, we consider the derivatives and their bounds used in Section \ref{se:inverted_link_function}.
\begin{lemma} \label{le:gprimeprimemonotone}
For $f_{ij}(u)$ in \eqref{eq:secndderivellinv} and $|u| < \bm{c}$ for $\bm{c} \in (0,1)$, we have
\begin{equation} \label{eq:gprimeprimemonotone}
| f_{ij}(u) | \leq
\frac{6}{1-\bm{c}^2} 
\max_{k = 0,2}
\max_{i = 1, \dots, \Dim}
\frac{ \left(m^{(k)}_{i}\left( 1+ \bm{c} \right) \right)^{2} }{ \left( m^{(0)}_{i}\left( 1-\bm{c} \right) \right)^{6}}
\leq
\frac{6}{1-\bm{c}^2} 
M_{1}(\bm{c},0)M_{2}(\bm{c},0).
\end{equation}
\end{lemma}

\begin{proof}
Using \eqref{eq:secndderivellinv}, \eqref{eq:second_derivative_of_ell} and \eqref{eq:ellprimewithSone}, we can write
\begin{align}
&
\left| - \frac{\ell''_{ij}(u)}{(\ell'_{ij}(u))^{3}} \right|
\nonumber
\\&=
\frac{
\left|
\frac{u}{2 \pi (1-u^2)^{\frac{3}{2}}} S_{1}(u) - 
\frac{u}{2 \pi (1-u^2)^{\frac{5}{2}}} S_{2}^{G}(u) +
\frac{1}{2 \pi (1-u^2)^{\frac{3}{2}}} S_{3}^{g}(u) \right| }{
\left( \frac{1}{2\pi \sqrt{1-u^2}} S_{1}(u) \right)^{3} } \nonumber
\\&=
(2 \pi)^{2}
\left|
u \frac{1}{S^{2}_{1}(u)} - 
\frac{u}{1-u^2} \frac{S_{2}^{G}(u)}{S^{3}_{1}(u)} +
\frac{ S_{3}^{g}(u) }{S^{3}_{1}(u)} \right|
\label{al:pppppppppp1}
\\&\leq
\frac{3}{1-\bm{c}^2} 
\frac{
m^{(0)}_{j}\left( 1+\bm{c} \right) m^{(0)}_{i}\left( 1+\bm{c} \right) 
+
m^{(2)}_{i}\left( 1+\bm{c} \right) m^{(0)}_{j}\left( 1+\bm{c} \right) 
+ 
m^{(0)}_{j}\left( 1+\bm{c} \right) m^{(2)}_{i}\left( 1+\bm{c} \right) }{
\left(m^{(0)}_{i}\left( 1-\bm{c} \right) m^{(0)}_{j}\left( 1-\bm{c} \right) \right)^{3} }
\label{al:pppppppppp2}
\\&\leq
\frac{6}{1-\bm{c}^2} 
\max_{k = 0,2}
\max_{i = 1, \dots, \Dim}
\frac{ \left(m^{(k)}_{i}\left( 1+\bm{c} \right) \right)^{2} }{ \left( m^{(0)}_{i}\left( 1-\bm{c} \right) \right)^{6}},
\nonumber
\end{align}
where the quantities in \eqref{al:pppppppppp1} are bounded separately below to get \eqref{al:pppppppppp2}.

\textit{Bound on $\frac{1}{S_{1}(u)}$:}
\begin{align}
\frac{1}{S_{1}(u)}
&=
\left(\sum_{n_{0},n_{1} = 0}^{\infty} \exp\left(-\frac{1}{2 (1-u^2)} \capitalG(u) \right)\right)^{-1}
\nonumber
\\&\leq
\left( \sum_{n = 0}^{\infty} \exp\left(-\frac{1}{2(1-\bm{c})} Q_{i,n}^{2} \right) 
\sum_{n = 0}^{\infty} \exp\left(-\frac{1}{2(1-\bm{c})} Q_{j,n}^{2} \right) \right)^{-1} 
\label{al:cnebfhrfbrhfbrhf1}
\\&=
\frac{1}{2\pi}
\left(m^{(0)}_{i}\left( 1-\bm{c} \right) m^{(0)}_{j}\left( 1- \bm{c} \right) \right)^{-1}.
\label{al:cnebfhrfbrhfbrhf2}
\end{align}
where \eqref{al:cnebfhrfbrhfbrhf1} follows since $a^2 +b^2 -2uab \leq a^2 +b^2 + 2 |ab| |u| \leq (1+|u|)(a^{2} + b^{2})$ for all $a,b \in \RR$.

\textit{Bound on $S_{2}^{G}(u)$:}
\begin{align}
\frac{1}{2 \pi} \left| S_{2}^{G}(u) \right|
&\leq
\frac{1}{2 \pi} 
\sum_{n_{0},n_{1} = 0}^{\infty} \exp\left(-\frac{1}{2 (1-u^2)} \capitalG(u) \right) |\capitalG(u)| 
\nonumber
\\&\leq
\frac{1+\bm{c}}{2 \pi}
\sum_{n_{0},n_{1} = 0}^{\infty} \exp\left(-\frac{1-|u| }{2(1-u^2)} (Q^{2}_{i,n_{0}} + Q^{2}_{j,n_{1}}) \right) (Q^{2}_{i,n_{0}} + Q^{2}_{j,n_{1}})
\label{al:xxx1231}
\\&\leq
2 \Bigg(
m^{(2)}_{i}\left( 1+\bm{c} \right) m^{(0)}_{j}\left( 1+\bm{c} \right) 
+ 
m^{(0)}_{i}\left( 1+\bm{c} \right) m^{(2)}_{j}\left( 1+\bm{c} \right)
\Bigg),
\nonumber
\end{align}
where \eqref{al:xxx1231} follows since $a^2 +b^2 -2uab \geq a^2 +b^2 - 2 |ab| |u| \geq (1- |u|)(a^{2} + b^{2})$ for all $a,b \in \RR$.

\textit{Bound on $S_{3}^{g}(u)$:} Similarly as above, 
\begin{align}
\frac{1}{2 \pi} \left| S_{3}^{g}(u) \right|
&\leq
\frac{1}{2 \pi} \sum_{n_{0},n_{1} = 0}^{\infty} \exp\left(-\frac{1}{2 (1-u^2)} \capitalG(u) \right) |\smallg|
\nonumber
\\&\leq
\frac{1}{2 \pi} \sum_{n_{0},n_{1} = 0}^{\infty} \exp\left(-\frac{1-|u| }{2(1-u^2)}  (Q^{2}_{i,n_{0}} + Q^{2}_{j,n_{1}}) \right) |Q_{i,n_{0}} Q_{j,n_{1}} |
\label{al:xxx3231}
\\&\leq
m^{(1)}_{i}\left( 1+\bm{c} \right) m^{(1)}_{j}\left( 1+\bm{c} \right)
\nonumber
\\&\leq
m^{(2)}_{i}\left( 1+\bm{c} \right) m^{(0)}_{j}\left( 1+\bm{c} \right) 
+ 
m^{(0)}_{i}\left( 1+\bm{c} \right) m^{(2)}_{j}\left( 1+\bm{c} \right),
\nonumber
\end{align}
where the last inequality follows from $| Q_{i,n_{0}} Q_{j,n_{1}} | \leq (Q^2_{i,n_{0}} + Q^2_{j,n_{1}})/2 \leq Q^2_{i,n_{0}} + Q^2_{j,n_{1}}$.

Finally, for the last relation in \eqref{eq:gprimeprimemonotone}, note that $1 + \bm{c} \leq \frac{1}{1-\bm{c}}$.
\end{proof}

\subsubsection{Derivatives II} \label{sec:derivatives2}
In this section, we consider the derivatives and their bounds used in Section \ref{se:oneoverellprime}.

\begin{lemma} \label{le:zprimezero}
For $z_{ij}(y)$ in \eqref{eq:def:smallzcapitalZ}, $-z'(0) = (-z'_{ij}(0))_{i,j=1,\dots, \Dim}$ is positive semidefinite and 
\begin{align}
|z_{ii}^{\prime}(0) |
&\leq
\max_{i =1,\dots, \Dim} 
\frac{\left( m_{i}^{(1)}(1) \right)^{2}}
{\left( m_{i}^{(0)}(1) \right)^{4}}
\leq
M_{2}(0,0).
\end{align}
\end{lemma}

\begin{proof}
The componentwise derivative of $z$ can be derived as
\begin{align}
z_{ij}^{\prime}(y) \nonumber
&= \frac{\partial}{\partial y} \left( \sum_{n_{0},n_{1} = 0}^{\infty} \exp\left(-\frac{1}{2} \capitalG(y) \right) \right)^{-1} \nonumber
\\
&=  - \left( \sum_{n_{0},n_{1} = 0}^{\infty} \exp\left(-\frac{1}{2} \capitalG(y) \right) \right)^{-2}
\sum_{n_{0},n_{1} = 0}^{\infty} \exp\left(-\frac{1}{2} \capitalG(y) \right) \smallg.
\label{al:zprime}
\end{align}
In particular, $z'$ evaluated at zero, gives
\begin{align}
-z^{\prime}(0) \nonumber
&= \left( 
\frac{
\sum_{n_{0},n_{1} = 0}^{\infty} \exp\left(-\frac{1}{2} (Q_{i,n_{0}}^{2} + Q_{j,n_{1}}^{2} )\right) Q_{i,n_{0}}Q_{j,n_{1}} }{
\left( \sum_{n_{0},n_{1} = 0}^{\infty} \exp\left(-\frac{1}{2} (Q_{i,n_{0}}^{2} + Q_{j,n_{1}}^{2} )\right) \right)^{2}
}
\right)_{i,j=1,\dots, \Dim} \nonumber
\\&=
\left(
\frac{ m_{1}^{(1)}(1) }
{\left( m_{1}^{(0)}(1) \right)^{2}}, \dots, 
\frac{ m_{\Dim}^{(1)}(1) }
{\left( m_{\Dim}^{(0)}(1) \right)^{2}} \right)'
\left(
\frac{ m_{1}^{(1)}(1) }
{\left( m_{1}^{(0)}(1) \right)^{2}}, \dots, 
\frac{ m_{\Dim}^{(1)}(1) }
{\left( m_{\Dim}^{(0)}(1) \right)^{2}} \right) 
\succcurlyeq 0.
\notag
\end{align}
Furthermore, 
\begin{align*}
|z_{ii}^{\prime}(0) |
= 
\frac{
\left( \sum_{n = 0}^{\infty} \exp\left(-\frac{1}{2} Q_{i,n}^{2} \right) Q_{i,n} \right)^{2}
}{
\left( \sum_{n = 0}^{\infty} \exp\left(-\frac{1}{2} Q_{i,n}^{2} \right) \right)^{4}
}
=
\frac{\left( m_{i}^{(1)}(1) \right)^{2}}
{\left( m_{i}^{(0)}(1) \right)^{4}}.
\end{align*}
\end{proof}

\begin{lemma} \label{le:zprimeprime}
For $z_{ij}(y)$ in \eqref{eq:def:smallzcapitalZ} and $ |y| < \bm{c} $ for $\bm{c} \in (0,1)$,
\begin{align}
| z_{ij}^{\prime\prime}(y) |
\leq
3 \max_{i =1,\dots, \Dim} 
\frac{ \left( m_{i}^{(2)} \left( \frac{1}{1-\bm{c}} \right) \right)^{2} }{ \left( m_{i}^{(0)}(1-\bm{c}) \right)^{4} }
\leq
3 M_{2}(\bm{c},0).
\end{align}
\end{lemma}

\begin{proof}
Continuing from \eqref{al:zprime},
\begin{align}
&
z_{ij}^{\prime\prime}(y) 
\nonumber
\\&=
-
\frac{\partial}{\partial y}
\left( \sum_{n_{0},n_{1} = 0}^{\infty} \exp\left(-\frac{1}{2} \capitalG(y) \right) \right)^{-2} 
\sum_{n_{0},n_{1} = 0}^{\infty} \exp\left(-\frac{1}{2} \capitalG(y) \right) \smallg \nonumber
\\&= 
2 \left( \sum_{n_{0},n_{1} = 0}^{\infty} \exp\left(-\frac{1}{2} \capitalG(y) \right) \right)^{-3}
\left( \sum_{n_{0},n_{1} = 0}^{\infty} \exp\left(-\frac{1}{2} \capitalG(y) \right) \smallg \right)^{2} \nonumber
\\&\hspace{1cm} -
\left( \sum_{n_{0},n_{1} = 0}^{\infty} \exp\left(-\frac{1}{2} \capitalG(y) \right) \right)^{-2} 
\sum_{n_{0},n_{1} = 0}^{\infty} \exp\left(-\frac{1}{2} \capitalG(y) \right) \big(\smallg \big)^2. \label{al:gggg}
\end{align}
Then,
\begin{align}
| z_{ij}^{\prime\prime}(y) |
&\leq
3 \frac{
\sum_{n_{0},n_{1} = 0}^{\infty} \exp\left(-\frac{1}{2} \capitalG(y) \right) (Q_{i,n_{0}}Q_{j,n_{1}})^2 }{
\left( \sum_{n_{0},n_{1} = 0}^{\infty} \exp\left(-\frac{1}{2} \capitalG(y) \right) \right)^{2} }
\label{al:gggg1}
\\&\leq
3 \frac{
\sum_{n_{0},n_{1} = 0}^{\infty} \exp\left(-\frac{1-\bm{c}}{2} (Q_{i,n_{0}}^{2} + Q_{j,n_{1}}^{2} )\right) (Q_{i,n_{0}}Q_{j,n_{1}})^2 }{
\left( \sum_{n_{0},n_{1} = 0}^{\infty} \exp\left(-\frac{1+\bm{c}}{2} (Q_{i,n_{0}}^{2} + Q_{j,n_{1}}^{2})\right) \right)^{2} }
\label{al:gggg2}
\\&\leq
3 \frac{
\sum_{n = 0}^{\infty} \exp\left(-\frac{1-\bm{c}}{2} Q_{i,n}^{2} \right) Q_{i,n}^2
\sum_{n = 0}^{\infty} \exp\left(-\frac{1-\bm{c}}{2} Q_{j,n}^{2} \right) Q_{j,n}^2 }{
\left( \sum_{n = 0}^{\infty} \exp\left(- \frac{1+\bm{c}}{2} Q_{i,n}^{2}\right) \sum_{n = 0}^{\infty} \exp\left(- \frac{1+\bm{c}}{2} Q_{j,n}^{2}\right) \right)^{2} }
\nonumber
%
\\&\leq
3 \frac{ m_{i}^{(2)}\left( \frac{1}{1-\bm{c}} \right) m_{j}^{(2)}\left( \frac{1}{1-\bm{c}}\right) }{ \left( m_{i}^{(0)}(1-\bm{c}) m_{j}^{(0)}(1-\bm{c}) \right)^{2} }  
\label{al:gggg4}
\\&\leq
3 \max_{i =1,\dots, \Dim} 
\frac{ \left( m_{i}^{(2)} \left( \frac{1}{1-\bm{c}} \right) \right)^{2} }{ \left( m_{i}^{(0)}(1-\bm{c}) \right)^{4} },
\nonumber
\end{align}
where \eqref{al:gggg1} follows from \eqref{al:gggg} and the Cauchy-Schwarz inequality. 
Finally, \eqref{al:gggg2} is a consequence of $a^2 +b^2 -2uab \leq a^2 +b^2 + 2 |ab| |u| \leq (1+\bm{c})(a^{2} + b^{2})$ and $a^2 +b^2 -2uab \geq a^2 +b^2 - 2 |ab| |u| \leq (1- \bm{c})(a^{2} + b^{2})$.
\end{proof}

\begin{lemma} \label{le:bounddercapZ}
Suppose there is a constant $\bm{c} \in (0,1)$ such that $| \Gamma_{Z,ij}(h) | < \bm{c}$ for all $i \neq j$ and $h \neq 0$. Then, for $Z_{ij}(x)$ in \eqref{eq:def:smallzcapitalZ} and $|x| < \bm{c}^2$ for $\bm{c} \in (0,1)$,
\begin{equation}
\begin{gathered}
|Z'_{ij}(x)| \leq \frac{1}{(1-\bm{c}^2)^{2}}
M_{2}(\bm{c},0).
\end{gathered}
\end{equation}
\end{lemma}

\begin{proof}
Set $\sigma_{ij} := \Gamma_{Z,ij}(h)$. Then, for $i \neq j$ and $h \neq 0$, 
\begin{align*}
Z'_{ij}(x)
&=
\frac{\partial}{\partial x} 
\left( \sum_{n_{0},n_{1} = 0}^{\infty} \exp\left(-\frac{1}{2 (1-x)} \capitalG(\sigma_{ij}) \right) \right)^{-1}
\\&=
- \frac{
\sum_{n_{0},n_{1} = 0}^{\infty} \exp\left(-\frac{1}{2 (1-x)} \capitalG(\sigma_{ij}) \right) \frac{-1}{2 (1-x)^{2}} \capitalG(\sigma_{ij}) }{
\left( \sum_{n_{0},n_{1} = 0}^{\infty} \exp\left(-\frac{1}{2 (1-x)} \capitalG(\sigma_{ij}) \right) \right)^{2} }.
\end{align*}
Note that $\capitalG(u) = Q_{i,n_{0}}^{2} + Q_{j,n_{1}}^{2} - 2 u Q_{i,n_{0}}Q_{j,n_{1}} \geq (Q_{i,n_{0}} u + Q_{j,n_{1}})^{2} \geq 0$. Then, with explanations given below,
\begin{align}
&|Z'_{ij}(x)|
\notag
\\&\leq
\frac{
\sum_{n_{0},n_{1} = 0}^{\infty} \exp\left(-\frac{1}{2} \capitalG(\sigma_{ij}) \right) \frac{1}{2 (1-x)^{2}} |\capitalG(\sigma_{ij})| }{
\left( \sum_{n_{0},n_{1} = 0}^{\infty} \exp\left(-\frac{1}{2 (1-x)} \capitalG(\sigma_{ij}) \right) \right)^{2} }
\notag
\\&\leq
\frac{1}{2 (1-\bm{c}^{2})^{2}}
\frac{
\sum_{n_{0},n_{1} = 0}^{\infty} \exp\left(-\frac{1-\bm{c}}{2} (Q_{i,n_{0}}^{2} + Q_{j,n_{1}}^{2} ) \right) (1+ \bm{c}) (Q_{i,n_{0}}^{2} + Q_{j,n_{1}}^{2} ) }{
\left( \sum_{n_{0},n_{1} = 0}^{\infty} \exp\left(-\frac{1+\bm{c}}{2 (1-\bm{c}^{2})} (Q_{i,n_{0}}^{2} + Q_{j,n_{1}}^{2} ) \right) \right)^{2} }
\label{al:ggggfff2}
\\&\leq
\frac{1}{(1-\bm{c}^2)^{2}}
\left( \sum_{n = 0}^{\infty} \exp\left(-\frac{1+\bm{c}}{2 (1-\bm{c}^{2})} Q_{i,n}^{2} \right) \sum_{n = 0}^{\infty} \exp\left(-\frac{1+\bm{c}}{2 (1-\bm{c}^{2})} Q_{j,n}^{2} \right) \right)^{-2} \notag
\\&\hspace{1cm} \times
\Bigg(
\sum_{n = 0}^{\infty} \exp\left(-\frac{1-\bm{c}}{2} Q_{i,n}^{2} \right)
\sum_{n = 0}^{\infty} \exp\left(-\frac{1-\bm{c}}{2} Q_{j,n}^{2} \right) Q_{j,n}^{2} \notag
\\&\hspace{2cm} +
\sum_{n = 0}^{\infty} \exp\left(-\frac{1-\bm{c}}{2} Q_{j,n}^{2} \right)
\sum_{n = 0}^{\infty} \exp\left(-\frac{1-\bm{c}}{2} Q_{i,n}^{2} \right) Q_{i,n}^{2}
\Bigg)
\notag
\\&=
\frac{1}{(1-\bm{c}^2)^{2}}
\frac{
m_{i}^{(0)}\left( \frac{1}{1-\bm{c}} \right) m_{j}^{(2)}\left( \frac{1}{1-\bm{c}} \right) 
+
m_{i}^{(2)}\left( \frac{1}{1-\bm{c}} \right) m_{j}^{(0)}\left( \frac{1}{1-\bm{c}} \right) }{
\left(m_{i}^{(0)}\left( 1-\bm{c} \right) m_{j}^{(0)}\left( 1- \bm{c} \right)\right)^{2} }
\\&\leq 
\frac{1}{(1-\bm{c}^2)^{2}}
\max_{r = 0, 2}
\max_{i =1, \dots, \Dim}
\frac{ \left( m_{i}^{(r)}\left( \frac{1}{1-\bm{c}} \right) \right)^{2} }{ \left( m_{i}^{(0)}\left( 1- \bm{c} \right) \right)^{4} },
\notag
\end{align}
where \eqref{al:ggggfff2} is a consequence of $a^2 +b^2 -2uab \leq a^2 +b^2 + 2 |ab| |u| \leq (1+\bm{c})(a^{2} + b^{2})$ and $a^2 +b^2 -2uab \geq a^2 +b^2 - 2 |ab| |u| \leq (1- \bm{c})(a^{2} + b^{2})$. 
\end{proof}

\subsubsection{Derivatives III} \label{sec:derivatives3}
In this section, we consider the derivatives and their bounds used in Section \ref{se:link_function}. We define $\widetilde{\Delta}(0), \widetilde{M}(\bm{c},0)$ and $\widetilde{\mu}(\bm{c},0)$ as $M(\bm{c},0)$ and $\mu(\bm{c},0)$ in \eqref{eq:diagonal_d(epsilon)}--\eqref{eq:mumoments2} with the true $\theta$ replaced by some $\widetilde{\theta}$, further specified in the proofs below. Similarly, we define $\widehat{S}_{1}(u), \widehat{S}_{2}^{G}(u)$ and $\widehat{S}_{3}^{g}(u)$ as the estimated counterparts of $S_{1}(u), S_{2}^{G}(u)$ and $S_{3}^{g}(u)$ by replacing $\theta_{i}$ with $\widehat{\theta}_{i}$.

\begin{lemma} \label{le:ellprimeprimedifferenceS1S2S3}
Suppose $|u| < \bm{c}$ for $\bm{c} \in (0,1)$. 
Then, the second derivative of the link function given in \eqref{eq:second_derivative_of_ell} satisfies, for some 
$\widetilde{\theta}_{i}, \widetilde{\theta}_{j}$ such that $| \widetilde{\theta}_{i} - \theta_{i} | < | \widehat{\theta}_{i} - \theta_{i} | $, 
$| \widetilde{\theta}_{j} - \theta_{j} | < | \widehat{\theta}_{j} - \theta_{j} | $
and $|u| < \bm{c}$ for $\bm{c} \in (0,1)$, 
\begin{equation}
| \widehat{\ell}''_{ij}(u) - \ell''_{ij}(u) |
\leq
\frac{18}{(1-\bm{c}^2)^{\frac{7}{2}}} 
\widetilde{M}(\bm{c},0) 
\widetilde{\mu}(\bm{c},0)
\| \widehat{\theta} - \theta \|_{\maxF}.
\end{equation}
\end{lemma}

\begin{proof}
From \eqref{eq:second_derivative_of_ell}, we have
\begin{align}
&
| \widehat{\ell}''_{ij}(u) - \ell''_{ij}(u) | \nonumber
\\&= 
\Bigg|
\frac{u}{2 \pi (1-u^2)^{\frac{3}{2}}} (\widehat{S}_{1}(u) - S_{1}(u)) - 
\frac{u}{2 \pi (1-u^2)^{\frac{5}{2}}} (\widehat{S}_{2}^{G}(u) - S_{2}^{G}(u))
\nonumber
\\&\hspace{1cm}+
\frac{1}{2 \pi (1-u^2)^{\frac{3}{2}}} (\widehat{S}_{3}^{g}(u) - S_{3}^{g}(u))
\Bigg| \nonumber
\\&\leq
\frac{1}{2 \pi (1-\bm{c}^2)^{\frac{5}{2}}} 
( | \widehat{S}_{1}(u) - S_{1}(u) | +
|\widehat{S}_{2}^{G}(u) - S_{2}^{G}(u)| +
|\widehat{S}_{3}^{g}(u) - S_{3}^{g}(u)| ) 
\label{al:ellprimeprimedifferenceS1S2S3_1}
\\&\leq
\frac{1}{2 \pi (1-\bm{c}^2)^{\frac{5}{2}}} 
\max_{i,j = 1,\dots, \Dim}
\Big(
\| \evaluat*{\nabla_{\theta_{i}} S_{1}(u) }{( \widetilde{\theta}_{i}, \widetilde{\theta}_{j})} \|_{1} 
+
\| \evaluat*{\nabla_{\theta_{i}} S_{2}^{G}(u) }{( \widetilde{\theta}_{i}, \widetilde{\theta}_{j})} \|_{1} 
\nonumber
\\&\hspace{1cm}+
\| \evaluat*{\nabla_{\theta_{i}} S_{3}^{g}(u) }{( \widetilde{\theta}_{i}, \widetilde{\theta}_{j})} \|_{1} 
\Big)
\| \widehat{\theta} - \theta \|_{\maxF}
\label{al:ellprimeprimedifferenceS1S2S3_1.1}
\\&\leq
\frac{18}{(1-\bm{c}^2)^{\frac{7}{2}}} 
\max_{k =0,\dots,3} \max_{i=1,\dots, \Dim} \widetilde{m}_{i}^{(k)}\left( 1+\bm{c} \right)
\max_{k =0,\dots,3} \max_{i=1,\dots, \Dim} \widetilde{\mu}_{i}^{(k)}\left( 1+\bm{c} \right)
\| \widehat{\theta} - \theta \|_{\maxF}
\notag
\\&=
\frac{18}{(1-\bm{c}^2)^{\frac{7}{2}}} 
\widetilde{M}(\bm{c},0) 
\widetilde{\mu}(\bm{c},0)
\| \widehat{\theta} - \theta \|_{\maxF}.
\nonumber
\end{align}
We consider the first summand in \eqref{al:ellprimeprimedifferenceS1S2S3_1} in detail. The remaining ones can be handled analogously.
\begin{align}
&
| \widehat{S}_{1}(u) - S_{1}(u) |
\nonumber
\\&=
| \langle \evaluat*{\nabla_{\theta_{i}} S_{1}(u) }{( \widetilde{\theta}_{i}, \widetilde{\theta}_{j})}, \widehat{\theta}_{j} -\theta_{j} \rangle
+
\langle \evaluat*{ \nabla_{\theta_{j}} S_{1}(u)  } {( \widetilde{\theta}_{i}, \widetilde{\theta}_{j})}, \widehat{\theta}_{i} -\theta_{i} \rangle |
\nonumber
\\&\leq
\| \evaluat*{\nabla_{\theta_{i}} S_{1}(u)}{( \widetilde{\theta}_{i}, \widetilde{\theta}_{j})} \|_{1} \| \widehat{\theta}_{j} -\theta_{j} \|_{\max}
+
\| \evaluat*{ \nabla_{\theta_{j}} S_{1}(u) } {( \widetilde{\theta}_{i}, \widetilde{\theta}_{j})} \|_{1} \| \widehat{\theta}_{i} -\theta_{i} \|_{\max}
\nonumber
\\&\leq
\max_{i,j = 1,\dots, \Dim} \| \evaluat*{\nabla_{\theta_{i}} S_{1}(u) }{( \widetilde{\theta}_{i}, \widetilde{\theta}_{j})} \|_{1} 
\max_{j = 1,\dots, \Dim} \| \widehat{\theta}_{j} -\theta_{j} \|_{\max}.
\notag
\end{align}
The subsequent relation \eqref{al:ellprimeprimedifferenceS1S2S3_1} can then be bounded through Lemma \ref{le:boundsonderivativesS1S2S3} below.
\end{proof}

The following Lemma \ref{le:boundsonderivativesS1S2S3} provides bounds on the derivatives of \eqref{eq:S1S2S3} with respect to the model parameters $\theta$.
\begin{lemma} \label{le:boundsonderivativesS1S2S3}
Suppose $|u| < \bm{c}$ for $\bm{c} \in (0,1)$. 
Then, the derivatives of the quantities \eqref{eq:S1S2S3} with respect to the model parameter $\theta_{i}$ can be bounded as
\begin{equation*}
\begin{aligned}
&
\| \nabla_{\theta_{i}} S_{1}(u) \|_{1}
+
\| \nabla_{\theta_{i}} S_{2}^{G}(u) \|_{1}
+
\| \nabla_{\theta_{i}} S_{3}^{g}(u) \|_{1}
\\&\leq
\frac{18}{1-\bm{c}^2} 2\pi
\max_{k =0,\dots,3} \max_{i=1,\dots, \Dim} m_{i}^{(k)}\left( 1+\bm{c} \right)
\max_{k =0,\dots,3} \max_{i=1,\dots, \Dim} \mu_{i}^{(k)}\left( 1+\bm{c} \right)
\end{aligned}
\end{equation*}
with $m_{i}^{(k)}, \mu_{i}^{(k)}$ in \eqref{eq:momentlikemk} and \eqref{eq:momentlikemkderiv}.
\end{lemma}

\begin{proof}
We consider the three quantities separately.

\textit{Bound on $\| \nabla_{\theta_{i}} S_{1}(u) \|_{1}$:}
\begin{align}
&
\| \nabla_{\theta_{i}} S_{1}(u) \|_{1}
\nonumber \\
&= 
\| \nabla_{\theta_{i}} 
\sum_{n_{0},n_{1} = 0}^{\infty} \exp\left(-\frac{1}{2 (1-u^2)} \capitalG(u) \right) \|_{1}
\nonumber
\\&\leq
\sum_{n_{0},n_{1} = 0}^{\infty} \exp\left(-\frac{1}{2 (1-u^2)} \capitalG(u) \right)
\left| -\frac{1}{1-u^2} (Q_{i,n_{0}} - u Q_{j,n_{1}}) \right|
\| \nabla_{\theta_{i}} Q_{i,n_{0}} \|_{1}
\nonumber
\\&\leq 
\frac{1}{1-\bm{c}^2}
\sum_{n_{0},n_{1} = 0}^{\infty} \exp\left(-\frac{1- |u| }{2(1-u^2)} (Q^{2}_{i,n_{0}} + Q^{2}_{j,n_{1}}) \right) 
| Q_{i,n_{0}} - u Q_{j,n_{1}} | 
\| \nabla_{\theta_{i}} Q_{i,n_{0}} \|_{1}
\nonumber
\\&\leq 
\frac{1}{1-\bm{c}^2}
\Bigg(
\sum_{n = 0}^{\infty} \exp\left(-\frac{1}{2(1+\bm{c})} Q^{2}_{j,n} \right)
\sum_{n = 0}^{\infty} \exp\left(-\frac{1}{2(1+\bm{c})} Q^{2}_{i,n} \right) 
| Q_{i,n} | \| \nabla_{\theta_{i}} Q_{i,n_{0}} \|_{1}
\nonumber
\\&\hspace{1cm}+
\sum_{n = 0}^{\infty} \exp\left(-\frac{1}{2(1+\bm{c})} Q^{2}_{j,n} \right) | Q_{j,n} | 
\sum_{n = 0}^{\infty} \exp\left(-\frac{1}{2(1+\bm{c})} Q^{2}_{i,n} \right) \| \nabla_{\theta_{i}} Q_{i,n_{0}} \|_{1}
\Bigg)
\nonumber
\\&\leq
\frac{1}{1-\bm{c}^2} 2\pi
\Bigg(
m_{j}^{(0)}\left( 1+\bm{c} \right) \mu_{i}^{(1)}\left( 1+\bm{c} \right) 
+
m_{j}^{(1)}\left( 1+\bm{c} \right) \mu_{i}^{(0)}\left( 1+\bm{c} \right) 
\Bigg)
\nonumber
\\&\leq
\frac{2}{1-\bm{c}^2} 2\pi
\max_{k =0,1} \max_{i=1,\dots, \Dim} m_{i}^{(k)}\left( 1+\bm{c} \right)
\max_{k =0,1} \max_{i=1,\dots, \Dim} \mu_{i}^{(k)}\left( 1+\bm{c} \right).
\label{al:bounddevS1}
\end{align}

\textit{Bound on $\| \nabla_{\theta_{i}} S_{2}^{G}(u) \|_{1}$:}
\begin{align}
&
\| \nabla_{\theta_{i}} S^{G}_{2}(u) \|_{1}
\nonumber
\\&= 
\| \nabla_{\theta_{i}} 
\sum_{n_{0},n_{1} = 0}^{\infty} \exp\left(-\frac{1}{2 (1-u^2)} \capitalG(u) \right)\capitalG(u) \|_{1}
\nonumber
\\&\leq
\sum_{n_{0},n_{1} = 0}^{\infty} \exp\left(-\frac{1}{2 (1-u^2)} \capitalG(u) \right) 
\nonumber
\\&\hspace{1cm}\times
\left| \capitalG(u) \left( -\frac{1}{1-u^2} (Q_{i,n} - u Q_{j,n}) \right) \right|
\| \nabla_{\theta_{i}} Q_{i,n_{0}} \|_{1}
\nonumber
\\&\hspace{2cm}+
\sum_{n_{0},n_{1} = 0}^{\infty} \exp\left(-\frac{1}{2 (1-u^2)} \capitalG(u) \right) 
\nonumber
\\&\hspace{3cm}\times
\left| -\frac{1}{1-u^2} (Q_{i,n} - u Q_{j,n}) \right|
\| \nabla_{\theta_{i}} Q_{i,n_{0}} \|_{1}
\nonumber
\\&\leq 
\frac{1}{1-\bm{c}^2}
\sum_{n_{0},n_{1} = 0}^{\infty} \exp\left(-\frac{1- |u| }{2(1-u^2)} (Q^{2}_{i,n_{0}} + Q^{2}_{j,n_{1}}) \right)
\nonumber
\\&\hspace{1cm}\times
2(Q^{2}_{i,n_{0}} + Q^{2}_{j,n_{1}})
(| Q_{i,n_{0}} | + |Q_{j,n_{1}} | ) 
\| \nabla_{\theta_{i}} Q_{i,n_{0}} \|_{1}
\nonumber
\\&\hspace{2cm} +
\frac{1}{1-\bm{c}^2}
\sum_{n_{0},n_{1} = 0}^{\infty} \exp\left(-\frac{1- |u| }{2(1-u^2)} (Q^{2}_{i,n_{0}} + Q^{2}_{j,n_{1}}) \right) 
\nonumber
\\&\hspace{3cm}\times
| Q_{i,n_{0}} - u Q_{j,n_{1}} | 
\| \nabla_{\theta_{i}} Q_{i,n_{0}} \|_{1}
\nonumber
\\&\leq
\frac{2}{1-\bm{c}^2} 2\pi
\Bigg(
m_{j}^{(3)} \left( 1+\bm{c} \right) 
\mu_{i}^{(0)}\left( 1+\bm{c} \right)
+
m_{j}^{(0)} \left( 1+\bm{c} \right) 
\mu_{i}^{(3)}\left( 1+\bm{c} \right)
\nonumber
\\&\hspace{3cm} +
m_{j}^{(2)} \left( 1+\bm{c} \right) 
\mu_{i}^{(1)}\left( 1+\bm{c} \right)
+
m_{j}^{(1)} \left( 1+\bm{c} \right) 
\mu_{i}^{(2)}\left( 1+\bm{c} \right)
\nonumber
\\&\hspace{4cm} +
m_{j}^{(1)} \left( 1+\bm{c} \right) 
\mu_{i}^{(0)}\left( 1+\bm{c} \right)
+
m_{j}^{(0)} \left( 1+\bm{c} \right) 
\mu_{i}^{(1)}\left( 1+\bm{c} \right)
\Bigg)
\nonumber
\\&\leq
\frac{12}{1-\bm{c}^2} 2\pi
\max_{k =0,\dots,3} \max_{i=1,\dots, \Dim} m_{i}^{(k)}\left( 1+\bm{c} \right)
\max_{k =0,\dots,3} \max_{i=1,\dots, \Dim} \mu_{i}^{(k)}\left( 1+\bm{c} \right).
\label{al:bounddevS2}
\end{align}

\textit{Bound on $\| \nabla_{\theta_{i}} S_{3}^{g}(u) \|_{1}$:}
\begin{align}
&
\| \nabla_{\theta_{i}} S^{g}_{3}(u) \|_{1}
\nonumber
\\&= 
\| \nabla_{\theta_{i}} 
\sum_{n_{0},n_{1} = 0}^{\infty} \exp\left(-\frac{1}{2 (1-u^2)} \capitalG(u) \right) \smallg \|_{1}
\nonumber
\\&\leq
\sum_{n_{0},n_{1} = 0}^{\infty} \exp\left(-\frac{1}{2 (1-u^2)} \capitalG(u) \right)
\nonumber
\\&\hspace{1cm}\times
\left| -\frac{1}{1-u^2} (Q_{i,n_{0}} - u Q_{j,n_{1}}) \smallg \right|
\| \nabla_{\theta_{i}} Q_{i,n_{0}} \|_{1}
\nonumber
\\&\hspace{2cm}+
\sum_{n_{0},n_{1} = 0}^{\infty} \exp\left(-\frac{1}{2 (1-u^2)} \capitalG(u) \right) 
|Q_{j,n_{1}}| \| \nabla_{\theta_{i}} Q_{i,n_{0}} \|_{1} 
\nonumber
\\&\leq 
\frac{1}{1-\bm{c}^2}
\sum_{n_{0},n_{1} = 0}^{\infty} \exp\left(-\frac{1- |u| }{2(1-u^2)} (Q^{2}_{i,n_{0}} + Q^{2}_{j,n_{1}}) \right)
\nonumber
\\
&\hspace{1cm}
\times
(| Q_{i,n_{0}} - u Q_{j,n_{1}} | | Q_{i,n_{0}} Q_{j,n_{1}} | + |Q_{j,n_{1}}|)
\| \nabla_{\theta_{i}} Q_{i,n_{0}} \|_{1}
\nonumber
\\&\leq
\frac{2}{1-\bm{c}^2} 2\pi
\Big(
m_{j}^{(1)}\left( 1+\bm{c} \right) 
\mu_{i}^{(1)}\left( 1+\bm{c} \right)
+
m_{j}^{(2)}\left( 1+\bm{c} \right)
\mu_{i}^{(1)}\left( 1+\bm{c} \right)
\nonumber
\\&\hspace{3cm}
+
m_{j}^{(2)}\left( 1+\bm{c} \right) 
\mu_{i}^{(0)}\left( 1+\bm{c} \right)
+
m_{j}^{(1)}\left( 1+\bm{c} \right)
\mu_{i}^{(2)}\left( 1+\bm{c} \right)
\Big)
\nonumber
\\&\leq
\frac{4}{1-\bm{c}^2} 2\pi
\max_{k =1,2} \max_{i=1,\dots, \Dim} m_{i}^{(k)}\left( 1+\bm{c} \right)
\max_{k =0,1,2} \max_{i=1,\dots, \Dim} \mu_{i}^{(k)}\left( 1+\bm{c} \right).
\label{al:bounddevS3}
\end{align}

Finally, combining \eqref{al:bounddevS1}, \eqref{al:bounddevS2} and \eqref{al:bounddevS3}, we infer that
\begin{equation*}
\begin{aligned}
&
\| \nabla_{\theta_{i}} S_{1}(u) \|_{1}
+
\| \nabla_{\theta_{i}} S^{G}_{2}(u) \|_{1}
+
\| \nabla_{\theta_{i}} S^{g}_{3}(u) \|_{1}
\\&\leq
\frac{18}{1-\bm{c}^2} 2\pi
\max_{k =0,\dots, 3} \max_{i=1,\dots, \Dim} m_{i}^{(k)}\left( 1+\bm{c} \right)
\max_{k =0,\dots,3} \max_{i=1,\dots, \Dim} \mu_{i}^{(k)}\left( 1+\bm{c} \right).
\end{aligned}
\end{equation*}
\end{proof}

\subsubsection{Derivatives IV} \label{sec:derivatives4}
In this section, we consider the derivatives and their bounds used in Section \ref{se:diagonal_elements}.

\begin{lemma} \label{le:ellhatofone}
Suppose Assumptions \ref{ass:finitemoments} and \ref{ass:finitemoments2}. Then, for some 
$\widetilde{\theta}_{i}$ such that $| \widetilde{\theta}_{i} - \theta_{i} | < | \widehat{\theta}_{i} - \theta_{i} | $, 
\begin{align}
| \widehat{\ell}_{ii}(1) - \ell_{ii}(1) |
\leq
3 \widetilde{\Delta}(0)
\| \widehat{\theta} - \theta \|_{\maxF}.
\end{align}
\end{lemma}

\begin{proof}
Note that $\ell_{ii}(1)$ can be written as
\begin{align} 
&
\ell_{ii}(1) = \sum_{k=1}^{\infty} \frac{c^2_{i,k}}{k!} = \Var[X_{i,t}]
\nonumber
\\&=
\sum_{n = 0}^{\infty} (2n+1) \Prob[X_{i,t} > n]
- \left(\sum_{n = 0}^{\infty} \Prob[X_{i,t} > n] \right)^2
\nonumber
\\&=
\sum_{n = 0}^{\infty} (2n+1) (1-C_{n}(\theta_{i}))
- \left(\sum_{n = 0}^{\infty} (1-C_{n}(\theta_{i}))\right)^2
\label{eq:gsgagsgdgsgag}
\end{align}
and is a function in $\theta_{i}$. Using \eqref{eq:gsgagsgdgsgag}, its partial derivative with respect to $\theta_{i}$ can be derived as
\begin{align} 
\nabla_{\theta_{i}} \ell_{ii}(1)
&=
\sum_{n = 0}^{\infty} (2n+1) \nabla_{\theta_{i}} C_{n}(\theta_{i})
- 2 \sum_{n = 0}^{\infty} (1-C_{n}(\theta_{i})) \sum_{n = 0}^{\infty} \nabla_{\theta_{i}} C_{n}(\theta_{i})
\nonumber
\\&=
\sum_{n = 0}^{\infty} (2n+1 - 2\E [X_{i,t}]) \nabla_{\theta_{i}} C_{n}(\theta_{i}).
\label{eq:gsgagsgdgsgagderiv}
\end{align}
The estimated counterpart $\widehat{\ell}_{ii}(1)$ of $\ell_{ii}(1)$ is then given by replacing $\theta_{i}$ with its estimator $\widehat{\theta}_{i}$ in \eqref{eq:gsgagsgdgsgag}. Using the representation \eqref{eq:gsgagsgdgsgag} and the derivative \eqref{eq:gsgagsgdgsgagderiv}, we can write
\begin{align}
| \widehat{\ell}_{ii}(1) - \ell_{ii}(1) |
&=
\left| \left \langle \evaluat*{ \sum_{n = 0}^{\infty} (2n +1 - \E [X_{i,t}]) \nabla_{\theta_{i}} C_{n}(\theta_{i}) }{ \widetilde{\theta}_{i} }, \widehat{\theta}_{i} -\theta_{i} \right \rangle \right|
\label{al:lllllll1}
\\&\leq
3 \evaluat*{ \sum_{n = 0}^{\infty} n \| \nabla_{\theta_{i}} C_{n}(\theta_{i}) \|_{1} }{ \widetilde{\theta}_{i} } 
\| \widehat{\theta} -\theta \|_{\max},
\nonumber
\end{align}
where we applied the mean value theorem in \eqref{al:lllllll1}.
\end{proof}

\begin{lemma} \label{le:se4boundreciprocal}
On the diagonal, the reciprocal of $\ell^{\prime}$ satisfies $\frac{1}{\ell^{\prime}_{ii}(u) } \leq M_{1}^{\frac{1}{2}}(1/2,0)$ with $M_{1}$ as in \eqref{eq:mumoments2}.
\end{lemma}

\begin{proof}
For the diagonal elements of the first derivative $\ell'_{ii}$, we can find a lower bound across all $u \in (-1,1)$. That is, 
\begin{equation*}
\begin{aligned}
\ell^{\prime}_{ii}(u) 
&= 
\frac{1}{2\pi \sqrt{1-u^2}} \sum_{n_{0},n_{1} = 0}^{\infty} \exp\left(-\frac{1}{2 (1-u^2)} (Q_{i,n_{0}}^{2} + Q_{i,n_{1}}^{2} - 2uQ_{i,n_{0}}Q_{i,n_{1}})\right)
\\&\geq
\sum_{n = 0}^{\infty} \exp\left(-\frac{1}{1+u} Q_{i,n_{0}}^{2} \right)
\\&\geq
\sum_{n = 0}^{\infty} \exp\left(-Q_{i,n_{0}}^{2} \right) = m^{(0)}_{i}(1/2)
\end{aligned}
\end{equation*}
such that $\frac{1}{\ell^{\prime}_{ii}(u) } \leq M_{1}^{\frac{1}{2}}(1/2,0)$ with $M_{1}$ as in \eqref{eq:mumoments2}.
\end{proof}

\section{Discussion on Assumption \ref{ass:finitemoments2}} \label{se:Discussion on assumption}
In this section, we verify Assumption \ref{ass:finitemoments2} for mixture Poisson, Conway-Maxwell-Poisson, binomial and negative binomial distributions.

For shortness' sake, we state an inequality for discrete random variables used in the subsequent examples in a small lemma at the end of this section.

\begin{example}[Mixture Poisson]
The mixture Poisson distribution is given by
\begin{equation*}
\Prob[X_{i,t} = k ] = \sum_{m=1}^{M} p_{m} e^{-\lambda_{m}} \frac{\lambda_{m}^{k}}{k!},
\hspace{0.2cm}
k = 0,1,\dots,
\end{equation*}
with mixture probabilities $\bm{p} = (p_{1}, \dots, p_{M})$ such that $\sum_{m=1}^{M} p_{m} = 1$, $p_{m} > 0$ and $\bm{\lambda} = (\lambda_{1}, \dots, \lambda_{M})$, $\lambda_{m} > 0$.

In order to verify Assumption \ref{ass:finitemoments2}, set $\theta_{i} = (\theta_{i1}, \dots, \theta_{i2M}) = (\bm{p}, \bm{\lambda})$. Then, with explanations given below, there is a constant $c>0$ not depending on any model parameters such that,
\begin{align} \label{al:example1sup}
&
\sup_{ \theta_{i} \in S} \sum_{n = 0}^{\infty} (1-C_{n}(\theta_{i}))^{-\frac{1}{2}} 
\sum_{j=1}^{2M}
\left| \frac{\partial}{\partial \theta_{ij}} C_{n}(\theta_{i}) \right| 
\nonumber
\\&\leq c
\sup_{(\bm{p}, \bm{\lambda}) \in S}
\max_{m =1, \dots,M} \frac{1}{p_{m}} \left(
(\E|X_{i,t}|^3)^{\frac{1}{2}} \left(1 + \lambda_{1}^{-\frac{1}{2}} \right) + 1 \right) < \infty.
\end{align}
Boundedness on $S$ follows by Assumption \ref{ass:finitemoments} and since the functions $x \mapsto \frac{1}{x}$ and $x \mapsto \frac{1}{\sqrt{x}}$ are both locally bounded on $(0, \infty)$.
We turn to explaining the inequality in \eqref{al:example1sup}.
With more details given below and with constant $c>0$ not depending on any model parameters and possibly changing from line to line,
\begin{align}
&
\sum_{n = 0}^{\infty} (1-C_{n}(\theta_{i}))^{-\frac{1}{2}} 
\sum_{j=1}^{2M}
\left| \frac{\partial}{\partial \theta_{ij}} C_{n}(\theta_{i}) \right|
\nonumber
\\&=
\sum_{n = 0}^{\infty} (\Prob[X_{i,t} > n])^{-\frac{1}{2}} 
\sum_{j=1}^{2M}
\left| \frac{\partial}{\partial \theta_{ij}} \Prob[X_{i,t} > n] \right|
\nonumber
\\&\leq
\max_{m =1, \dots,M} \frac{1}{p_{m}}
\sum_{n = 0}^{\infty} (\Prob[X_{i,t} > n])^{-\frac{1}{2}} 
(
\Prob[X_{i,t} > n]
+
\Prob[X_{i,t} = n]
)
\label{al:wwwqqq000}
\\&=
\max_{m =1, \dots,M} \frac{1}{p_{m}}
\sum_{n = 0}^{\infty} 
\left(
\Prob[X_{i,t} > n]^{\frac{1}{2}}
+
\frac{
\Prob[X_{i,t} = n]}{
\Prob[X_{i,t} > n]^{\frac{1}{2}}}
\right)
\nonumber
\\&=
\max_{m =1, \dots,M} \frac{1}{p_{m}}
\left(
\sum_{n = 0}^{\infty}
\Prob[X_{i,t} > n]^{\frac{1}{2}}
+
\sum_{n = 0}^{\infty}
(\Prob[X_{i,t} = n] )^{\frac{1}{2}}
\left(
\frac{
\Prob[X_{i,t} = n]}{
\Prob[X_{i,t} > n]}
\right)^{\frac{1}{2}}
\right)
\nonumber
\\&\leq
\max_{m =1, \dots,M} \frac{1}{p_{m}}
\left(
\sum_{n = 0}^{\infty}
(\Prob[X_{i,t} > n] )^{\frac{1}{2}}
+
c
\sum_{n = 0}^{\infty} 
(\Prob[X_{i,t} = n] )^{\frac{1}{2}}
\left( \frac{n}{\lambda_{1}} \right)^{\frac{1}{2}}
\right)
\label{al:wwwqqq111}
\\&\leq
\max_{m =1, \dots,M} \frac{1}{p_{m}}
\left(
c(\E|X_{i,t}|^3)^{\frac{1}{2}} + 1
+
c \sum_{n = 1}^{\infty} 
(\Prob[X_{i,t} = n] )^{\frac{1}{2}}
n^{\frac{3}{2}} n^{-1}
\right)
\label{al:wwwqqq1111}
\\&\leq
\max_{m =1, \dots,M} \frac{1}{p_{m}}
\left(
c(\E|X_{i,t}|^3)^{\frac{1}{2}} + 1
+
c
\left(
\sum_{n = 0}^{\infty} n^{3}
\Prob[X_{i,t} = n]
\sum_{n = 1}^{\infty} n^{-2}
\right)^{\frac{1}{2}} \lambda_{1}^{-\frac{1}{2}}
\right)
\label{al:wwwqqq11111}
\\&=
c \max_{m =1, \dots,M} \frac{1}{p_{m}} \left(
(\E|X_{i,t}|^3)^{\frac{1}{2}} \left(1 + \lambda_{1}^{-\frac{1}{2}} \right) + 1 \right),
\nonumber
\nonumber
\end{align}
where \eqref{al:wwwqqq1111} follows by Lemma \ref{le:momentformulasumtail} and \eqref{al:wwwqqq11111} 
by H\"older's inequality. We consider \eqref{al:wwwqqq000} and \eqref{al:wwwqqq111} separately. 
The bound \eqref{al:wwwqqq000} follows since 
\begin{align*}
&
\sum_{j=1}^{2M}
\left| \frac{\partial}{\partial \theta_{ij}} \Prob[X_{i,t} > n] \right|
=
\sum_{j=1}^{2M}
\left| \frac{\partial}{\partial \theta_{ij}} \sum_{k = n+1}^{\infty} \Prob[X_{i,t} = k] \right|
\\&=
\sum_{j=1}^{M}
\left| \frac{\partial}{\partial p_{j}} \sum_{k = n+1}^{\infty} \sum_{m=1}^{M} p_{m} e^{-\lambda_{m}} \frac{\lambda_{m}^{k}}{k!} \right|
+
\sum_{j=1}^{M}
\left| \frac{\partial}{\partial \lambda_{j}} \sum_{k = n+1}^{\infty} \sum_{m=1}^{M} p_{m} e^{-\lambda_{m}} \frac{\lambda_{m}^{k}}{k!} \right|
\\&=
\sum_{j=1}^{M}
\left| \sum_{k = n+1}^{\infty}
e^{-\lambda_{j}} \frac{\lambda_{j}^{k}}{k!}
\right|
+
\sum_{j=1}^{M}
\left| \sum_{k = n+1}^{\infty} p_{j} \left( e^{-\lambda_{j}} \frac{\lambda_{j}^{k-1}}{(k-1)!}
-
e^{-\lambda_{j}} \frac{\lambda_{j}^{k}}{k!}
\right)
\right|
\\&=
\sum_{j=1}^{M}
\left| \sum_{k = n+1}^{\infty}
e^{-\lambda_{j}} \frac{\lambda_{j}^{k}}{k!}
\right|
+
\sum_{j=1}^{M} p_{j}
\left| 
e^{-\lambda_{j}} \frac{\lambda_{j}^{n}}{n!}
\right|
\\&\leq
\max_{m =1, \dots,M} \frac{1}{p_{m}}
\sum_{k = n+1}^{\infty} \sum_{j=1}^{M} p_{j}
e^{-\lambda_{j}} \frac{\lambda_{j}^{n}}{n!}
+
\Prob[X_{i,t} = n]
\\&=
\max_{m =1, \dots,M} \frac{1}{p_{m}}
\Prob[X_{i,t} > n]
+
\Prob[X_{i,t} = n].
\end{align*}
For the bound \eqref{al:wwwqqq111}, assume without loss of genarality that $\lambda_{1} > \dots > \lambda_{M}$. Note that the ratio between probability function and tail distribution behaves asymptotically as 
\begin{align}
&
\frac{
\Prob[X_{i,t} = n]}{
\Prob[X_{i,t} > n]}
\nonumber
\\&=
\frac{
\sum_{m=1}^{M} p_{m} e^{-\lambda_{m}} \frac{\lambda_{m}^{n}}{n!} }{
\sum_{k = n+1}^{\infty} \sum_{m=1}^{M} p_{m} e^{-\lambda_{m}} \frac{\lambda_{m}^{k}}{k!}}
\nonumber
\\&=
\frac{
p_{1} e^{-\lambda_{1}} \frac{\lambda_{1}^{n}}{n!}
+
\sum_{m=2}^{M} p_{m} e^{-\lambda_{m}} \frac{\lambda_{m}^{n}}{n!} }{
p_{1} e^{-\lambda_{1}} \frac{\lambda_{1}^{n+1}}{(n+1)!}
+
\sum_{m=2}^{M} p_{m} e^{-\lambda_{m}} \frac{\lambda_{m}^{n+1}}{(n+1)!}
+
\sum_{k = n+2}^{\infty} \sum_{m=1}^{M} p_{m} e^{-\lambda_{m}} \frac{\lambda_{m}^{k}}{k!}
}
\nonumber
\\&=
\frac{
p_{1} e^{-\lambda_{1}} \frac{\lambda_{1}^{n}}{n!}
\left(1
+
\sum_{m=2}^{M} \frac{p_{m}}{p_{1}} e^{-\lambda_{m}+\lambda_{1}} \left(\frac{\lambda_{m}}{\lambda_{1}}\right)^{n}
\right)
}{
p_{1} e^{-\lambda_{1}} \frac{\lambda_{1}^{n+1}}{(n+1)!}
M_{n}(\lambda_{1},\dots,\lambda_{M})
}
\label{al:mnlambda}
\\&\sim
\frac{n}{\lambda_{1}}
\frac{
1
+
\sum_{m=2}^{M} \frac{p_{m}}{p_{1}} e^{-\lambda_{m}+\lambda_{1}} \left(\frac{\lambda_{m}}{\lambda_{1}}\right)^{n}
}{
M_{n}(\lambda_{1},\dots,\lambda_{M})
}
\sim 
\frac{n}{\lambda_{1}},
\nonumber
\end{align}
where $M_{n}(\lambda_{1},\dots,\lambda_{M})$ in \eqref{al:mnlambda} is defined in \eqref{eq:Mnlambda} below. Furthermore, the last relation follows 
since $\frac{(n+1)!}{k!} \to 0$ for $k \geq n+2$ and $\left(\frac{\lambda_{m}}{\lambda_{1}}\right)^{k} \to 0$ as $n\to \infty$. We conclude with the definition of $M_{n}(\lambda_{1},\dots,\lambda_{M})$,
\begin{equation}\label{eq:Mnlambda}
\begin{aligned}
M_{n}(\lambda_{1},\dots,\lambda_{M})
=
1
&+
\sum_{m=2}^{M} \frac{p_{m}}{p_{1}} e^{-\lambda_{m}+\lambda_{1}} \left(\frac{\lambda_{m}}{\lambda_{1}}\right)^{n+1}
\\&+
\sum_{k = n+2}^{\infty} \lambda_{1}^{k-(n+1)} \sum_{m=1}^{M} \frac{p_{m}}{p_{1}} e^{-\lambda_{m}+\lambda_{1}} \left(\frac{\lambda_{m}}{\lambda_{1}}\right)^{k} \frac{(n+1)!}{k!}.
\end{aligned}
\end{equation}
\end{example}

\begin{example}[Conway-Maxwell-Poisson]
The Conway-Maxwell-Poisson distribution is
\begin{equation*}
\Prob[X_{i,t} = k ] = \frac{\lambda^{k}}{(k!)^{\nu} Z(\lambda,\nu)},
\hspace{0.2cm}
k = 0,1,\dots, 
\hspace{0.2cm}
\text{ with }
\hspace{0.2cm}
Z(\lambda,\nu) = \sum_{j=0}^{\infty} \frac{\lambda^{j}}{(j!)^{\nu}},
\end{equation*}
for $\lambda, \nu >0$. Let $\theta_{i} = (\theta_{i1}, \theta_{i2}) = (\lambda, \nu)$. Then, with explanations given below, there is a constant $c>0$ not depending on any model parameters such that,
\begin{align}
&
\sup_{\theta_{i} \in S} \sum_{n = 0}^{\infty} (1-C_{n}(\theta_{i}))^{-\frac{1}{2}} 
\sum_{j=1}^{2}
\left| \frac{\partial}{\partial \theta_{ij}} C_{n}(\theta_{i}) \right| 
\nonumber
\label{al:example2sup}
\\&\leq
\sup_{(\lambda, \nu) \in S}
2 c
\left((\E|X_{i,t}|^3)^{\frac{1}{2}} + 1\right)
\left( \frac{1}{\lambda} 
\E|X_{i,t}|
+
\E|\log(X_{i,t}!)|
\right) < \infty.
\end{align}
Boundedness on $S$ follows by Assumption \ref{ass:finitemoments} and since the function $x \mapsto \frac{1}{x}$ is locally bounded on $(0, \infty)$.
We turn to explaining the inequality in \eqref{al:example2sup}.
With more details given below,
\begin{align}
&
\sum_{n = 0}^{\infty} (1-C_{n}(\theta_{i}))^{-\frac{1}{2}} 
\sum_{j=1}^{2}
\left| \frac{\partial}{\partial \theta_{ij}} C_{n}(\theta_{i}) \right| 
\nonumber
\\&=
\sum_{n = 0}^{\infty} (\Prob[X_{i,t} > n])^{-\frac{1}{2}} 
\sum_{j=1}^{2}
\left| \frac{\partial}{\partial \theta_{ij}} \Prob[X_{i,t} > n] \right|
\nonumber
\\&=
\sum_{n = 0}^{\infty} (\Prob[X_{i,t} > n])^{-\frac{1}{2}} 
\left(
\left|\frac{\partial}{\partial \lambda} \Prob[X_{i,t} > n ] \right| +
\left|\frac{\partial}{\partial \nu} \Prob[X_{i,t} > n ] \right|
\right)
\nonumber
\\&\leq
2
\sum_{n = 0}^{\infty} (\Prob[X_{i,t} > n])^{-\frac{1}{2}} 
\left( \frac{1}{\lambda} 
\Prob[X_{i,t} > n ]
\E|X_{i,t}|
+
\Prob[X_{i,t} > n ]
\E|\log(X_{i,t}!)|
\right)
\label{al:CMP1}
\\&=
2
\sum_{n = 0}^{\infty} (\Prob[X_{i,t} > n])^{\frac{1}{2}}
\left( \frac{1}{\lambda} 
\E|X_{i,t}|
+
\E|\log(X_{i,t}!)|
\right)
\nonumber
\\&\leq
2 c
\left((\E|X_{i,t}|^3)^{\frac{1}{2}} + 1\right)
\left( \frac{1}{\lambda} 
\E|X_{i,t}|
+
\E|\log(X_{i,t}!)|
\right),
\nonumber
\end{align}
where the last step follows by Lemma \ref{le:momentformulasumtail}. The bound \eqref{al:CMP1} is obtained as follows.
The derivative with respect to $\lambda$ can be bounded as
\begin{align}
\frac{\partial}{\partial \lambda} \Prob[X_{i,t} > n ]
&=
\frac{\partial}{\partial \lambda} \sum_{ k = n+1}^{\infty}
\frac{\lambda^{k}}{(k!)^{\nu} Z(\lambda,\nu)}
\nonumber
\\&=
\sum_{ k = n+1}^{\infty}
k \frac{\lambda^{k-1}}{(k!)^{\nu} Z(\lambda,\nu)}
-
\sum_{ k = n+1}^{\infty}
\frac{\lambda^{k}}{(k!)^{\nu} Z(\lambda,\nu)^2}
\sum_{ j = 0}^{\infty}
j \frac{\lambda^{j-1}}{(j!)^{\nu}}
\nonumber
\\&\leq
\frac{1}{\lambda Z^2(\lambda,\nu)}
\left(
\sum_{ k = n+1}^{\infty}
k \frac{\lambda^{k}}{(k!)^{\nu}}
\sum_{ j = 0}^{\infty}
\frac{\lambda^{j}}{(j!)^{\nu}}
+
\sum_{ k = n+1}^{\infty}
\frac{\lambda^{k}}{(k!)^{\nu}}
\sum_{ j = 0}^{\infty}
j \frac{\lambda^{j}}{(j!)^{\nu}}
\right)
\nonumber
\\&\leq
\frac{2}{\lambda} 
\sum_{ k = n+1}^{\infty}
\frac{\lambda^{k}}{(k!)^{\nu} Z(\lambda,\nu)}
\sum_{ j = 0}^{\infty}
j \frac{\lambda^{j}}{(j!)^{\nu}Z(\lambda,\nu)}
\nonumber
\\&\leq
\frac{2}{\lambda} 
\Prob[X_{i,t} > n ]
\E|X_{i,t}|.
\notag
\end{align}
The derivative with respect to $\nu$ can be bounded as
\begin{align}
&
\frac{\partial}{\partial \nu} \Prob[X_{i,t} > n ]
\nonumber
\\&=
\frac{\partial}{\partial \nu} \sum_{ k = n+1}^{\infty}
\frac{\lambda^{k}}{(k!)^{\nu} Z(\lambda,\nu)}
\nonumber
\\&=
\frac{1}{Z(\lambda,\nu)^2} 
\left(
\sum_{ k = n+1}^{\infty}
(-1)\log(k!) \frac{\lambda^{k}}{(k!)^{\nu}} Z(\lambda,\nu)
+
\sum_{ k = n+1}^{\infty}
\frac{\lambda^{k}}{(k!)^{\nu}}
\sum_{ j = 0}^{\infty}
\log(j!)
\frac{\lambda^{j}}{(j!)^{\nu}}
\right)
\nonumber
\\&\leq
\frac{1}{Z(\lambda,\nu)^2} 
\left(
\sum_{ k = n+1}^{\infty}
\log(k!) \frac{\lambda^{k}}{(k!)^{\nu}} \sum_{j=0}^{\infty} \frac{\lambda^{j}}{(j!)^{\nu}}
+
\sum_{ k = n+1}^{\infty}
\frac{\lambda^{k}}{(k!)^{\nu}}
\sum_{ j = 0}^{\infty}
\log(j!)
\frac{\lambda^{j}}{(j!)^{\nu}}
\right)
\nonumber
\\&\leq
2
\sum_{ k = n+1}^{\infty}
\frac{\lambda^{k}}{(k!)^{\nu} Z(\lambda,\nu)}
\sum_{ j = 0}^{\infty}
\log(j!)
\frac{\lambda^{j}}{(j!)^{\nu} Z(\lambda,\nu)}
\nonumber
\\&=
2
\Prob[X_{i,t} > n ]
\E|\log(X_{i,t}!)|.
\notag
\end{align}
Note that $\E|\log(X_{i,t}!)|$ maximizes the Conway-Maxwell-Poisson likelihood function and does not have an analytic solution.
\end{example}

\begin{example}[Binomial] \label{ex:Ass3bernoulli}
The binomial distribution is
\begin{equation*}
\Prob[X_{i,t} = k ] = \binom{N}{k}p^k(1-p)^{N-k},
\hspace{0.2cm} k = 0,1,\dots,N.
\end{equation*}
In order to verify Assumption \ref{ass:finitemoments2}, note that we assume that the number of trials $N$ is known and $\theta_{i} = p$ for the unknown probability of success in each trial. Then, with explanations given below, there is a constant $c>0$ not depending on any model parameters such that,
\begin{align} \label{al:example3sup}
\sup_{\theta_{i} \in S} \sum_{n = 0}^{\infty} (1-C_{n}(\theta_{i}))^{-\frac{1}{2}} 
\left| \frac{\partial}{\partial \theta_{i}} C_{n}(\theta_{i}) \right| 
\leq
c \sup_{p \in S}
\frac{N}{p} \left((\E|X_{i,t}|^3)^{\frac{1}{2}} + 1\right)
< \infty.
\end{align}
Boundedness on $S$ follows by Assumption \ref{ass:finitemoments} and since the function $x \mapsto \frac{1}{x}$ and is locally bounded on $(0, 1)$.
We turn to explaining the inequality in \eqref{al:example3sup}.
With more details given below and with constant $c>0$ not depending on any model parameters,
\begin{align}
\sum_{n = 0}^{\infty} (\Prob[X_{i,t} > n])^{-\frac{1}{2}} 
\left| \frac{\partial}{\partial p} \Prob[X_{i,t} > n ] \right|
&\leq
\frac{N}{p} \sum_{n = 0}^{\infty} (\Prob[X_{i,t} > n])^{\frac{1}{2}}
\label{al:binomineq0}
\\&\leq
c \frac{N}{p} \left((\E|X_{i,t}|^3)^{\frac{1}{2}} + 1\right), \label{al:binomineq1}
\end{align}
where \eqref{al:binomineq1} follows by Lemma \ref{le:momentformulasumtail}
and \eqref{al:binomineq0} is true since
\begin{align*}
\frac{\partial}{\partial p} \Prob[X_{i,t} > n ]
&=
\frac{\partial}{\partial p} \sum_{ k = n+1}^{N} \binom{N}{k}p^k(1-p)^{N-k}
\\&=
\sum_{ k = n+1}^{N} \binom{N}{k} (kp^{k-1}(1-p)^{N-k} - (N-k)p^k(1-p)^{N-k-1})
\\&=
\sum_{ k = n+1}^{N} \binom{N}{k} p^k(1-p)^{N-k} \frac{k(1-p)-Np}{p(1-p)}
\\&\leq
\frac{N}{p} \Prob[X_{i,t} > n ].
\end{align*}
\end{example}

\begin{example}[Negative binomial] \label{ex:Ass3negbinomial}
The negative binomial distribution is
\begin{equation*}
\Prob[X_{i,t} = k ] = \binom{k+r-1}{r-1}(1-p)^k p^{r},
\hspace{0.2cm} k = 0,1,\dots,
\end{equation*}
for $r \in \NN$. Let $\theta_{i} = (\theta_{i1}, \theta_{i2}) = (p, r)$. Note that in contrast to Example \ref{ex:Ass3bernoulli}, both the number of failures until the experiment is stopped and the success probability in each experiment are assumed to be unknown. Then, with explanations given below, there is a constant $c>0$ not depending on any model parameters such that,
\begin{align} \label{al:example4sup}
&
\sup_{\theta_{i} \in S} \sum_{n = 0}^{\infty} (1-C_{n}(\theta_{i}))^{-\frac{1}{2}} 
\sum_{j=1}^{K_{i}}
\left| \frac{\partial}{\partial \theta_{ij}} C_{n}(\theta_{i}) \right| 
\nonumber
\\&\leq c
\sup_{(p,r) \in S}
\left((\E|X_{i,t}|^3)^{\frac{1}{2}} + 1\right) \left(
\frac{r}{p} + |\log(p)| \right)
\end{align}
Boundedness on $S$ follows by Assumption \ref{ass:finitemoments} and since the functions $x \mapsto \frac{1}{x}$, $x \mapsto x$ and $x \mapsto |\log(x)|$ are locally bounded on $(0, \infty)$.
We turn to explaining the inequality in \eqref{al:example4sup}.
With more details given below and with constant $c>0$ not depending on any model parameters,
\begin{align}
&
\sum_{n = 0}^{\infty} (\Prob[X_{i,t} > n])^{-\frac{1}{2}} 
\left(
\left| \frac{\partial}{\partial p} \Prob[X_{i,t} > n ] \right|
+
\left| \frac{\partial}{\partial r} \Prob[X_{i,t} > n ] \right|
\right)
\nonumber
\\&\leq
\sum_{n = 0}^{\infty} (\Prob[X_{i,t} > n])^{-\frac{1}{2}}
\sum_{ k = n+1}^{\infty} \binom{k+r-1}{r-1} (1-p)^k p^{r} 
\left(
\frac{r}{p} - \frac{k}{1-p}
+
\frac{k}{r} + \log(p) \right)
\label{al:NB1}
\\&\leq
\sum_{n = 0}^{\infty} (\Prob[X_{i,t} > n])^{-\frac{1}{2}}
\sum_{ k = n+1}^{\infty} \binom{k+r-1}{r-1} (1-p)^k p^{r} 
\left(
\frac{r}{p} + |\log(p)| \right)
\label{al:NB2}
\\&=
\sum_{n = 0}^{\infty} (\Prob[X_{i,t} > n])^{-\frac{1}{2}}
\Prob[X_{i,t} > n ]
\left(
\frac{r}{p} + |\log(p)| \right)
\nonumber
\\&=
\sum_{n = 0}^{\infty} (\Prob[X_{i,t} > n])^{\frac{1}{2}} \left(
\frac{r}{p} + |\log(p)| \right)
\leq
c\left((\E|X_{i,t}|^3)^{\frac{1}{2}} + 1\right) \left(
\frac{r}{p} + |\log(p)| \right).
\label{al:NB3}
\end{align}
The inequality \eqref{al:NB1} follows by calculating the two derivatives with respect to $p$ and $r$. The results are given in \eqref{al:NB4} and \eqref{al:NB6} below, respectively. Then, \eqref{al:NB2} is due to $r \geq 1-p$ since $r \in \NN$. The last line \eqref{al:NB3} follows by Lemma \ref{le:momentformulasumtail}.

The derivative of $\Prob[X_{i,t} > n ]$ with respect to $p$ is calculated as follows
\begin{align}
\frac{\partial}{\partial p} \Prob[X_{i,t} > n ]
&=
\frac{\partial}{\partial p} \sum_{ k = n+1}^{\infty} \binom{k+r-1}{r-1}(1-p)^k p^{r}
\nonumber
\\&=
\sum_{ k = n+1}^{\infty} \binom{k+r-1}{r-1} (-k(1-p)^{k-1} p^{r} + (1-p)^k rp^{r-1})
\nonumber
\\&=
\sum_{ k = n+1}^{\infty} \binom{k+r-1}{r-1} (1-p)^k p^{r} \left( \frac{r}{p} - \frac{k}{1-p} \right).
\label{al:NB4}
\end{align}
For the derivative with respect to $n$, we get
\begin{align}
\frac{\partial}{\partial r} \Prob[X_{i,t} > n ]
&=
\frac{\partial}{\partial r} \sum_{ k = n+1}^{\infty} \binom{k+r-1}{r-1}(1-p)^k p^{r}
\nonumber
\\&=
\sum_{ k = n+1}^{\infty} (1-p)^k \left( p^{r} \frac{\partial}{\partial r} \binom{k+r-1}{r-1} + \log(p)p^{r} \binom{k+r-1}{r-1} \right)
\nonumber
\\&\leq
\sum_{ k = n+1}^{\infty} (1-p)^k \left( p^{r} \frac{k}{r} \binom{k+r-1}{r-1} + \log(p)p^{r} \binom{k+r-1}{r-1} \right)
\label{al:NB5}
\\&=
\sum_{ k = n+1}^{\infty} \binom{k+r-1}{r-1} (1-p)^k p^{r} \left( \frac{k}{r} + \log(p) \right),
\label{al:NB6}
\end{align}
where the derivative of the binomial coefficient in \eqref{al:NB5} can be written and bounded as
\begin{align*}
\frac{\partial}{\partial r} \binom{k+r-1}{r-1}
&=
\sum_{j=1}^{k} \frac{1}{k!} \prod_{i \in \{ 1, \dots, k\} / \{j\}} (k+r-i)
\\&=
\sum_{j=1}^{k} \frac{1}{k!} \frac{k+r-j}{k+r-j} \prod_{i \in \{ 1, \dots, k\} / \{j\}} (k+r-i)
\\&=
\sum_{j=1}^{k} \frac{1}{k+r-j} \binom{k+r-1}{r-1}
\leq
\frac{k}{r} \binom{k+r-1}{r-1}.
\end{align*}
\end{example}

We state here a generic result which has found extensive use in showing that Assumption \ref{ass:finitemoments2} is satisfied.
\begin{lemma} \label{le:momentformulasumtail}
For a nonnegative discrete random variable $X$, there is a constant $c>0$ such that
\begin{align*} 
\sum_{n = 0}^{\infty} (\Prob[X > n])^{\frac{1}{2}}
\leq
c (\E|X|^3)^{\frac{1}{2}} + 1.
\end{align*}
\end{lemma}

\begin{proof}
We make use of the following formula for moments of discrete random variables
\begin{align} \label{eq:momentformulasumtail}
\E X^p =
\sum_{n = 0}^{\infty} ( (n+1)^{p} - n^{p}) \Prob[X > n], \hspace{0.2cm} \text{ for } p=1,2,\dots .
\end{align}
Then,
\begin{align}
\sum_{n = 0}^{\infty} (\Prob[X > n])^{\frac{1}{2}}
&=
\sum_{n = 1}^{\infty} (n^2 \Prob[X > n])^{\frac{1}{2}} n^{-1} + (\Prob[X > 0])^{\frac{1}{2}}
\nonumber
\\&\leq
\left( \sum_{n = 1}^{\infty} n^2 \Prob[X > n] \right)^{\frac{1}{2}} \left( \sum_{n=1}^{\infty} n^{-2} \right)^{\frac{1}{2}} +1
\label{al:lololololo1}
\\&\leq
\left( \sum_{n = 0}^{\infty} ((n+1)^{3} - n^{3}) \Prob[X > n] \right)^{\frac{1}{2}} \left( \frac{\pi^2}{6} \right)^{\frac{1}{2}} + 1
\label{al:lololololo2}
\\&=
c (\E|X|^3)^{\frac{1}{2}} + 1,
\nonumber
\end{align}
where \eqref{al:lololololo1} is a consequence of applying H\"older's inequality and \eqref{al:lololololo2} follows since $n^2 \leq 3n^2 +3n +1 = (n+1)^{3} - n^{3}$ which allows us to use \eqref{eq:momentformulasumtail}.
\end{proof}

\small
\bibliographystyle{plainnat}
\bibliography{countTS_bib}

\flushleft
\begin{tabular}{lp{1 in}l}
Marie-Christine D\"uker & & Robert Lund\\
Dept.\ of Statistics and Data Science & & Dept.\ of Statistics\\
Cornell University & & UC Santa Cruz\\
129 Garden Ave, Comstock Hall & & 1156 High Street, Engineering 2\\
Ithaca, NY 14850, USA & & Santa Cruz, CA 95064, USA\\
{\it duker@cornell.edu} & & {\it rolund@ucsc.edu}\\
\end{tabular}

\flushleft
\begin{tabular}{lp{1 in}l}
Vladas Pipiras\\
Dept.\ of Statistics and Operations Research\\
UNC Chapel Hill\\
CB\#3260, Hanes Hall\\
Chapel Hill, NC 27599, USA\\
{\it pipiras@email.unc.edu}\\
\end{tabular}

\end{document}